\documentclass[a4paper]{amsart}
\usepackage[utf8]{inputenc}

\usepackage[a4paper,width={16cm},left=2.5cm,bottom=3cm, top=3cm]{geometry}

\usepackage{latexsym}
\usepackage{a4wide}
\usepackage{graphics}
\usepackage{amsmath}
\usepackage{amssymb}
\usepackage{bm}
\usepackage{mathtools}
\usepackage{mathrsfs}
\usepackage{enumerate,enumitem}
\usepackage{csquotes}
\usepackage{pgf,tikz}
\usepackage[normalem]{ulem}

\usepackage{hyperref}
\usepackage{cleveref}
\hypersetup{
	linktocpage=true,
	colorlinks=true,
	linkcolor=blue!70!black,
	citecolor=orange!40!red,
	urlcolor=magenta!80!black,
}

\usepackage{graphicx}
\usepackage{color}

\usepackage{etoolbox,kvoptions,logreq,pdftexcmds,shortmathj}
\usepackage[style=ext-alphabetic,backend=biber,maxbibnames=5,maxalphanames=5,minalphanames=3,giveninits=true,url=false,articlein=false,maintitleaftertitle=true,isbn=false]{biblatex}
\AtBeginDocument{}
\AtEveryBibitem{\clearfield{note}}

\numberwithin{equation}{section} 

\newtheorem{theorem}{Theorem}[section]
\newtheorem{corollary}[theorem]{Corollary}
\newtheorem{lemma}[theorem]{Lemma}
\newtheorem{proposition}[theorem]{Proposition}
\theoremstyle{definition} 
\newtheorem{definition}[theorem]{Definition}
\newtheorem{example}[theorem]{Example}

\newtheorem{remark}[theorem]{Remark}

\newtheorem{question}{Question}

\newcommand{\R}{\mathbb{R}}	
\newcommand{\N}{\mathbb{N}} 
\newcommand{\dx}{\,\mathrm{d}x}	
\newcommand{\ds}{\,\mathrm{d}S}	
\newcommand{\dr}{\,\mathrm{d}r}	
\newcommand{\dt}{\,\mathrm{d}x}	
\renewcommand{\d}{\mathrm{d}}
\newcommand{\e}{\varepsilon}	
\newcommand{\weak}{\rightharpoonup}
\newcommand{\nnu}{{\bm{\nu}}}  

\newcommand{\D}{\mathcal{D}}

\newcommand{\Te}{\mathcal{T}_\e}
\newcommand{\Ee}{\mathsf{E}_\e}

\newcommand{\Oe}{{\Omega_\e}}

\newcommand{\norm}[1]{\left\lVert #1 \right\lVert}
\newcommand{\abs}[1]{\left| #1 \right|}
\newcommand{\sub}{\subseteq}
\newcommand\restr[1]{\raisebox{-.5ex}{$|$}_{#1}}

\DeclareMathOperator{\Span}{\mathrm{span}}
\DeclareMathOperator{\dist}{\mathrm{dist}}

\newenvironment{bvp}{\left\{\begin{aligned}  }{\end{aligned}\right.}

\title[Spectral stability for the Neumann Laplacian in domains with small holes]{Quantitative spectral stability for the Neumann Laplacian in domains with small holes }
\author{Veronica Felli, Lorenzo Liverani and Roberto Ognibene}

\address{Veronica Felli
  \newline \indent Dipartimento di Matematica e Applicazioni
  \newline \indent
Universit\`a degli Studi di Milano–Bicocca
\newline\indent Via Cozzi 55, 20125 Milano, Italy}
\email{veronica.felli@unimib.it}

\address{Lorenzo Liverani
	\newline \indent Department Mathematik
    \newline \indent Chair of Dynamics, Control, Machine Learning and Numerics
	\newline \indent Friedrich-Alexander-Universität Erlangen-Nürnberg 
	\newline\indent  Cauerstraße 11, 91058 Erlangen, Germany}
\email{lorenzo.liverani@fau.de}

\address{Roberto Ognibene
	\newline \indent Dipartimento di Matematica
	\newline \indent Universit\`a di Pisa
	\newline\indent  Largo Bruno Pontecorvo, 5, 56127 Pisa, Italy}
\email{roberto.ognibene@dm.unipi.it}


\date{Revised version, March 3 2025}

\begin{document}

\begin{abstract}
The aim of the present paper is to investigate the behavior of the spectrum of 
the Neumann Laplacian  in domains  with little holes excised from 
the interior. More precisely, we consider the 
eigenvalues of the Laplacian with homogeneous Neumann boundary conditions on a 
bounded, Lipschitz domain. Then, we singularly perturb the domain by 
removing Lipschitz sets which are “small” in a suitable sense and satisfy a 
uniform extension property. In this context, we provide an asymptotic expansion for all 
the eigenvalues of the perturbed problem which are converging to simple 
eigenvalues of the limit one. The eigenvalue variation turns out to depend 
on a geometric quantity resembling the notion of (boundary) torsional rigidity: 
understanding this fact is one of the main contributions of the present paper. 
In the particular case of a hole shrinking to a point, through a fine 
blow-up analysis, we identify the exact vanishing order of such a quantity and 
we establish some connections between the location of the hole and the sign of the eigenvalue variation.
\end{abstract}

\maketitle

\tableofcontents

\section{Introduction}
In the present paper, we investigate the stability of the spectrum of the Neumann Laplacian under singular perturbations, consisting in the removal of small holes from a
bounded domain. 

Eigenvalues and eigenfunctions of differential operators 
are ubiquitous  in the 
theory of partial differential equations. Understanding how these are sensitive to small perturbations, such as variations in the domain, is of interest in several fields of physical applications, e.g. quantum mechanics, material sciences, heat conduction, climate modeling and acoustics. See, in particular, \cite{rayleigh} for perturbation theory in acoustics, \cite[Chapter V]{Courant-Hilbert} for eigenvalue problems in connection with vibrating systems and heat conduction  and \cite{DelSole2015} (see also \cite{Yang2022}) for links to climate analysis. We also quote \cite{hale} for a thorough survey on the dependence of eigenvalues and eigenfunctions on smooth and nonsmooth perturbations of the domain. Furthermore, a comprehensive understanding of the shape of  eigenfunctions holds great significance in many numerical analysis problems. Nonetheless, the computational cost of determining eigenelements in domains with minute cavities is considerably high: indeed, to ensure precision in such cases, i.e. to discern even small variations, dense mesh structures are needed around these cavities. Consequently, theoretical approximation results in this specific context assume a pivotal role. We refer to \cite{Boffi-Gardini-Gastaldi} and references therein for a wide discussion of the topic. We also mention  \cite{Marshall} for some recent applications to machine learning of spectral stability of the Neumann Laplacian under domain deformations. Finally, as pointed out in \cite[Section 1.4]{henrot2006}, asymptotic expansions of eigenvalues in domains with small holes might find applications in shape optimization, e.g. in the proof of non-existence of minimizers.

The problem of spectral stability  for differential operators in perforated domains has a long history, and presents intrinsically different
features depending on which kind of boundary conditions are taken into account. 
Let us consider a bounded open set 
$\Omega\sub\R^N$, which we call \emph{unperturbed domain}, and a compact subset $K\sub 
\Omega$, which we call \emph{hole}; we refer to the set $\Omega\setminus K$ 
as the \emph{perturbed domain}. 
In order for 
$\Omega\setminus K$ to be regarded as a perturbation of $\Omega$, the hole $K$ 
needs to be sufficiently small, in a suitable sense depending on the operator under investigation. In this regard, a key role is played by the  conditions  prescribed  on the boundaries of both the unperturbed domain  and  the hole; among the most studied cases, we find homogeneous Dirichlet and
Neumann boundary conditions, as well as Robin-type ones. 
Under each of the  boundary conditions mentioned above and under suitable regularity assumptions on the sets,
the eigenvalue problem for the Laplace operator on the perturbed domain $\Omega\setminus K$ typically admits
  a  sequence of diverging eigenvalues 
\[
    \Lambda_0(\Omega\setminus K)\leq\Lambda_1(\Omega\setminus K)\leq \Lambda_2(\Omega\setminus K)\leq\cdots\leq \Lambda_n(\Omega\setminus K)\leq \cdots,
\]
by classical spectral theory.
Analogously, the unpertubed problem (corresponding to the case $K=\emptyset$)
typically admits a sequence of diverging eigenvalues
\[
    \Lambda_0(\Omega)\leq\Lambda_1(\Omega)\leq \Lambda_2(\Omega)\leq\cdots\leq \Lambda_n(\Omega)\leq \cdots.
\]
In this setting, the stability of the spectrum is a main object of investigation. More precisely, a major question is the following:
\begin{question}\label{q:1}
    Under which conditions on the hole $K$, are the perturbed eigenvalues $\Lambda_n(\Omega\setminus K)$ arbitrarily close to the corresponding unperturbed ones $\Lambda_n(\Omega)$?
\end{question}
Once conditions on $K$ that ensure spectral stability are found, the further following question naturally arises:
\begin{question}\label{q:2}
    Is it possible to quantify the difference $\Lambda_n(\Omega\setminus K)-\Lambda_n(\Omega)$ in terms of some \emph{measurement} of $K$?
\end{question}
In the case of homogeneous Neumann boundary conditions on both the external boundary $\partial\Omega$ and the hole's boundary  $\partial K$,  question \ref{q:1} has been answered in \cite{Rauch1975}. In the present paper we focus on question \ref{q:2}.   

\medskip
We precede the presentation of our results with a brief overview of the literature dealing with the spectral stability for the Laplacian in perforated domains.
This problem  is widely investigated. In particular, the case of Dirichlet boundary conditions is one of the most studied and, being the literature on the topic so vast, we cite here just some of the most relevant papers. In the Dirichlet case, it is well known that a key quantity in the study of spectral stability is the capacity of the hole. 
Some first estimates of the variation of the Dirichlet eigenvalues in terms of the capacity of the removed set date back to \cite{samarskii1948}. 
The paper \cite{Rauch1975}, published in 1975, still stands as a pivotal reference in this research field; it contains a more systematic study of spectral stability in domains with small holes, taking also into account more general boundary conditions. 
Subsequent studies are carried out in a series of papers by   Ozawa in the 80s,  deriving sharp asymptotic expansions of perturbed eigenvalues, especially in small dimensions, see e.g. \cite{ozawa1981}. Another relevant result is contained in \cite{Courtois1995} (recently revisited in \cite{AFHL}), which provides  an asymptotic expansion for any perturbed 
(possibly multiple) eigenvalue; in particular, if  $\Lambda_n(\Omega)$ is a simple Dirichlet eigenvalue and $u_n$ is a corresponding $L^2$-normalized eigenfunction, then
\begin{equation}\label{eq:soa_1}
    \Lambda_n(\Omega\setminus K_\e)=\Lambda_n(\Omega)+\mathrm{cap}_\Omega(K_\e,u_n)+o(\mathrm{cap}_\Omega(K_\e,u_n))
    \quad\text{as $\e\to0$,}
\end{equation}
where $\{K_\e\}_{\e>0}$ is a family  of compact sets concentrating to a zero-capacity set
as $\e\to0$,
 and
\[
    \mathrm{cap}_{\Omega}(K_\e,u_n):=\inf\left\{ \int_\Omega\abs{\nabla u}^2\dx\colon u\in H^1_0(\Omega),~u-u_n\in H^1_0(\Omega\setminus K_\e)\right\}.
\]
We also cite \cite{ALM2022}, treating the case of multiple limit eigenvalues. For simple eigenvalues, an analogue of \eqref{eq:soa_1} 
is derived in \cite{AFN2020} 
in a fractional setting and in \cite{FR23} for polyharmonic operators.
The results of \cite{FNO1}  seem to suggest that  only the boundary conditions prescribed on the hole essentially play a role in the asymptotics of eigenvalues when the hole disappears: indeed, in \cite{FNO1},
in the case of Neumann conditions prescribed on the outer boundary and Dirichlet conditions on the hole, an asymptotic expansion similar to  \eqref{eq:soa_1} is proved. Again a suitable notion of capacity of the hole comes into play.

As for Neumann boundary conditions prescribed on the hole, less is known, and a richer phenomenology can be observed. After the work of Rauch and Taylor \cite{Rauch1975}, where sufficient conditions for stability of the Neumann spectrum are provided, several papers investigate the asymptotic behavior of perturbed eigenvalues. 
In dimension $2$ and in the case of a disk-shaped hole, Ozawa \cite{ozawa1983} proves that, if $\Lambda_n(\Omega)$ is a simple eigenvalue of the Dirichlet Laplacian in $\Omega$, and $\Lambda_{n,\e}$ is the $n$-th eigenvalue of
\begin{equation}\label{eq:mixed}
    \begin{bvp}
        -\Delta u &=\Lambda u, &&\text{in }\Omega_\e:=\Omega\setminus\overline{\Sigma_\e}, \\
        u &=0, &&\text{on }\partial \Omega, \\
        \partial_\nnu u &=0, &&\text{on }\partial \Sigma_\e,
        \end{bvp}
\end{equation}
with $\Sigma_\e= \{x\in \R^2:|x-x_0|<\e\}$ for some $x_0\in\Omega$, then
\[  
    \Lambda_{n,\e}=\Lambda_n(\Omega)-\pi\e^2\left(2\abs{\nabla u_n(x_0)}^2-\Lambda_n(\Omega)u_n^2(x_0)\right)+O(\e^3|\log\e|^2)\quad\text{as }\e\to 0,
\]
where $u_n\in H^1_0(\Omega)$ is a $L^2$-normalized Dirichlet eigenfunction corresponding to $\Lambda_n(\Omega)$. See also \cite{ozawa1985} for asymptotic properties of   eigenfunctions of \eqref{eq:mixed} in dimension $2$. For $N=3$, asymptotic expansions for the perturbed Neumann eigenvalues $\mu_n(\Omega_\e)$ are obtained in \cite{nazarov2011} in the case of a hole shrinking to a point: more precisely, the expansion
\[
    \mu_n(\Omega_\e)=\mu_n(\Omega)+ C_n\e^3+o(\e^3)\quad\text{as }\e\to 0,
\]
is proved, where $C_n\in\R$ is explicitly characterized  in \parencite[(2.46)]{nazarov2011}. A more general framework is taken into account in \cite{jimbo}: here, 
the removed holes are tubular neighborhoods of $d$-dimensional manifolds $\mathcal{M}\sub\Omega$, i.e.
\begin{equation}\label{eq:tubular}
    \Sigma_\e=\left\{ x\in\Omega\colon \dist (x,\mathcal{M})<\e\right\},
\end{equation}
and both Neumann and Robin conditions on $\partial\Sigma_\e$ are considered. 
Denoting $q:=N-d\geq 2$, if $\Lambda_n(\Omega)$ is a simple eigenvalue of the Dirichlet Laplacian in $\Omega$ (with a $L^2$-normalized eigenfunction $u_n$) and $\Lambda_{n,\e}$ is the $n$-th eigenvalue of \eqref{eq:mixed} with $\Sigma_\e$ as in \eqref{eq:tubular}, then
\cite{jimbo} proves the expansion
\[
    \Lambda_{n,\e}=\Lambda_n(\Omega)-\omega_q\,\e^q\int_{\mathcal{M}}\left[\frac{q}{q-1}\abs{\nabla_{\perp}u_n}^2+\abs{\nabla_{\mathcal{M}}u_n}^2-\Lambda_n(\Omega)u_n^2+u_n H[u_n] \right]\d \mathcal{H}^d+o(\e^q)
\]
as $\e\to 0$, where
\[
    H[u_n](x):=\lim_{t\to 0}\frac{u_n(x+tH(x))-u_n(x)}{t}
\]
and $H(x)$ denotes the mean curvature vector field on $\mathcal{M}$. 
For Neumann conditions prescribed on both $\partial\Omega$ and the hole's boundary $\partial\Sigma_\e$, full asymptotic expansions in terms of analytic functions are obtained in \cite{LdC2012}, in the case of a hole shrinking to a point. We also cite \cite{kabeya}, where Neumann eigenvalues are studied for  a zonal subdomain of the $N$-sphere, which converges to the whole sphere itself. Finally, it is worth mentioning \cite{BurenkovDavies,LMS2013,BGU2016} for other quantitative spectral stability results for the Neumann Laplacian and \cite{hempel,Kozlov,colette2021,BSH2022,cardone,barseghyan2023magnetic} for 
 qualitative studies for more general types of holes.

\medskip
As emerged from the previous discussion, the stability of Neumann eigenvalues in domains with small holes is not understood as well as in the Dirichlet case and presents different and peculiar features. For instance, stability of the spectrum of the Neumann Laplacian is not guaranteed under assumptions that would otherwise ensure stability in the Dirichlet case, see e.g. \parencite[p. 420]{Courant-Hilbert} or \cite{schweizer}. 
In \cite{Rauch1975} (see also \cite{colette2021}) it is proved that, by removing sets $\overline{\Sigma_\e}$ whose measure tends to zero as $\e\to0$, a sufficient condition for stability is a 
uniform extension property (in the Sobolev sense) inside the hole, which 
rules out too wild behaviors of the disappearing hole itself; see assumption 
{\bf (H)} in section \ref{sec:mainresults}.
The main novelty of the present paper lies in the identification of a geometric quantity, related to the shape of the hole $\overline{\Sigma_\e}$ and to the limit eigenfunction $\varphi_n$, that plays in the Neumann context the same role as the capacity does for the Dirichlet case, concerning quantitative spectral stability. This quantity, denoted as $\mathcal{T}_{\overline{\Omega}\setminus\Sigma_\e}(\partial\Sigma_\e,\partial_\nnu\varphi_n)$, is introduced in \Cref{def:sobolev_torsion} and resembles a notion of torsional rigidity.

Now, we provide a description of the most significant contents of the present paper, referring  to section \ref{sec:mainresults} for the rigorous  statements. 
Our first main result \Cref{thm:main1} contains an  asymptotic expansion
 for an eigenvalue $\lambda_n(\Omega\setminus\overline{\Sigma_\e})$ of the perturbed Neumann problem \eqref{eq:Peps}, in case it converges to a simple  eigenvalue $\lambda_n(\Omega)$ of the  unperturbed one \eqref{eq:P0}. In the asymptotic expansion of the variation 
 $\lambda_n(\Omega\setminus\overline{\Sigma_\e})-\lambda_n(\Omega)$,
  the sum of two contributions appears:
  the geometric quantity
\begin{equation}\label{eq:comm_1}
      -      \mathcal{T}_{\overline{\Omega}\setminus\Sigma_\e}(\partial\Sigma_\e,\partial_\nnu\varphi_n)<0 
\end{equation}
which always has a negative sign, and the additional term
\begin{equation}\label{eq:comm_2}
    -\int_{\Sigma_\e}\left(\abs{\nabla\varphi_n}^2-(\lambda_n(\Omega)-1)\varphi_n^2 \right)\dx,
\end{equation}
whose sign depends on where the hole is located, with respect to the nodal, regular and singular sets of $\varphi_n$.
The presence of this additional term causes the eigenvalue variation  not to have a fixed sign, in stark contrast to what happens in the Dirichlet setting, where the monotonicity of the eigenvalues with respect to  the inclusion of domains  always results in positive differences
$\Lambda_n(\Omega\setminus\overline{\Sigma_\e})-\Lambda_n(\Omega)$.

Next, we focus on holes shrinking to a point by maintaining the same fixed shape, that is
of the form
\begin{equation}\label{eq:hole-shrinking}
    \Sigma_\e:= x_0+\e\Sigma=\{x_0+\e x\colon x\in \Sigma\},
\end{equation}
for some $x_0\in\Omega$ and $\Sigma\sub\R^N$. In this case, if $N\geq3$, we succeed in performing a blow-up analysis, which  provides the explicit rate of convergence of the quantity \eqref{eq:comm_1}. Moreover, by analyzing the behavior of the limit eigenfunction $\varphi_n$ near the point $x_0$, we determine the explicit rate of convergence of the additional term \eqref{eq:comm_2}. Combining these two sharp expansions, we obtain our second and third main results, namely \Cref{thm:blowup_intr} and \Cref{thm:blow_up_eigenfunct}, which provide, for $N\geq3$ and the hole being as in \eqref{eq:hole-shrinking}, a precise description of the asymptotic behaviour of eigenvalues and eigenfunctions, respectively. 

From  \Cref{thm:blowup_intr} we can  deduce some interesting information about the sign of the eigenvalue variation  and the sharpness of the derived expansion. 
Notably, these aspects appear to depend on whether $x_0$ 
is  or is not located on the singular set of the limit eigenfunction, and, in the latter case, on its specific positioning relative to the interface $\Gamma$ introduced in  \eqref{eq:def_Gamma}.
See Remark \ref{rem:Gamma} for details.

Finally, in Theorems \ref{thm:spher} and \ref{thm:spher-Dim2}  we derive more explicit expansions in the case of spherical holes, in dimensions $N\geq3$ and $N=2$ respectively.

\section{Statement of the main results}\label{sec:mainresults}
For any open, bounded, connected, Lipschitz set $\Omega\sub\R^N$, $N\geq2$, we consider  the following eigenvalue 
problem: 
\begin{equation}\label{eq:P0}
\begin{bvp}
-\Delta \varphi +\varphi&=\lambda\varphi, &&\text{in }\Omega, \\
\partial_{\nnu}\varphi &=0, &&\text{on }\partial\Omega,
\end{bvp}
\end{equation}
where $\nnu$ denotes the outer unit normal vector to $\partial\Omega$. Problem \eqref{eq:P0}  is meant  in a weak sense; i.e.,  $\lambda\in\R$ is an \emph{eigenvalue} if there exists $u\in H^1(\Omega)\setminus\{0\}$, called \emph{eigenfunction}, such that
\begin{equation}\label{eq:weak_form}
    \int_\Omega(\nabla u\cdot\nabla \varphi+u\varphi)\dx=\lambda\int_\Omega u\varphi\dx,\quad\text{for all }\varphi\in H^1(\Omega).
\end{equation}
 In view of the compact embedding $H^1(\Omega)\hookrightarrow L^2(\Omega)$, classical spectral theory ensures the existence of a diverging sequence of eigenvalues
\[
        0<1=\lambda_0(\Omega)< \lambda_1(\Omega)\leq\lambda_2(\Omega)\leq \cdots\leq \lambda_n(\Omega)\leq \cdots.
\]
It is evident that the eigenvalues $\{\mu_n(\Omega)\}$ of the standard Neumann Laplacian can be obtained from those of  \eqref{eq:P0} with a translation, i.e.
\begin{equation*}
    \mu_n(\Omega):=\lambda_n(\Omega)-1,
\end{equation*}
the eigenfunctions being  exactly the same. Moreover, since $\lambda_0(\Omega)$ is  equal to $1$ for any choice of the set $\Omega$, in the following we only  consider  eigenvalues with index $1$ or higher.

Let us now perturb  $\Omega$,  by removing a small hole from the interior. More precisely,  we consider  a family $\{\Sigma_\e\}_{\e\in(0,\e_0)}$ of subsets of $\R^N$ satisfying the following assumption.

\medskip\noindent{\bf Assumption (H).}
    We assume that, for some $\e_0>0$,
\begin{align}
\label{eq:ass_sigma_1}\tag{H1} &\text{for every $\e\in (0,\e_0)$, $\Sigma_\e$ is an open, Lipschitz set such that $\overline{\Sigma_\e}\sub\Omega$};\\
\label{eq:ass_sigma_3}\tag{H2} &\text{for every $\e\in(0,\e_0)$, there exists $\Ee\colon H^1(\Omega\setminus\overline{\Sigma_\e})\to H^1(\Omega)$ such that\hspace{3cm}}\\
\notag & \qquad(\Ee u)\restr{\Omega\setminus\overline{\Sigma_\e}}=u \quad\text{and}\quad 
            \norm{\Ee u}_{H^1(\Omega)}\leq\mathfrak{C}\norm{u}_{H^1(\Omega\setminus\overline{\Sigma_\e})} \quad\text{for all $u\in H^1(\Omega\setminus\overline{\Sigma_\e})$},\\
&\notag\text{where  $\mathfrak{C}>0$ is a constant independent  of $\e\in (0,\e_0)$;}\\
 \label{eq:ass_sigma_4}\tag{H3} &\text{$\abs{\Sigma_\e}\to 0$ as $\e\to 0$ (where $\abs{\,\cdot\,}$ denotes the $N$-dimensional Lebesgue measure).}
\end{align}
\smallskip\noindent For every $\e\in (0,\e_0)$, we denote the perturbed domain by
\begin{equation*}
    \Omega_\e:=\Omega\setminus\overline{\Sigma_\e},
\end{equation*}
and consider the  perturbed problem
\begin{equation}\label{eq:Peps}
    \begin{bvp}
        -\Delta\varphi+\varphi&=\lambda\varphi, &&\text{in }\Omega_\e, \\
        \partial_\nnu\varphi &=0, &&\text{on }\partial\Omega_\e,
    \end{bvp}
\end{equation}
meant in a weak sense as in \eqref{eq:weak_form}. This  produces  the perturbed spectrum,  which consists of an increasing diverging sequence 
$\{\lambda_j(\Omega_\e)\}_{j\in\N}$. 
In \cite[Theorem 3.1]{Rauch1975} Rauch and Taylor prove that,  under
assumption {\bf (H)}, 
\begin{equation}\label{eq:conv-eigenvalues}
	\lambda_{j}(\Omega_\e)\to \lambda_j(\Omega)\quad\text{as }\e\to 0,\quad \text{for all }j\in\N.
\end{equation}
Moreover, by classical spectral theory, there exist
\begin{equation}\label{eq:bases}
\{\varphi_j\}_{j\geq 0}\sub H^1(\Omega)\quad\text{and}\quad\{\varphi_j^\e\}_{j\geq 0}\sub H^1(\Oe)
\end{equation}
orthonormal bases of $L^2(\Omega)$, respectively $L^2(\Omega_\e)$, 
such that, for every $j$, $\varphi_j$ and $\varphi_j^\e$ are eigenfunctions associated to $\lambda_j(\Omega)$ and $\lambda_j(\Omega_\e)$, respectively. 

Hereafter, we fix $n\in\N\setminus\{0\}$ such that
\begin{equation}\label{eq:simple}
    \lambda_n(\Omega)\text{ is simple}.
\end{equation}
A key role in our asymptotic expansion is played by the geometric quantity defined below, which provides a measurement of
the hole $\Sigma_\e$ and resembles the notion of torsional rigidity of a set; see e.g. \cite{henrot2018,PZ1951} for the classical notion of \emph{torsional rigidity} and \cite{BGI} for the \emph{boundary torsional rigidity}. 
\begin{definition}\label{def:sobolev_torsion}
Let $E\sub\R^N$ be an open Lipschitz set such that
$\overline{E}\subset\Omega$ and  $f\in L^2(\partial E)$. Let
\begin{equation*}
J_{\Omega,E,f}:H^1(\Omega\setminus\overline{E})\to\R,\quad 
J_{\Omega,E,f}(u):=\frac{1}{2}\int_{\Omega\setminus \overline{E}}(\abs{\nabla u}^2+u^2)\dx
-\int_{\partial E}uf\ds.
\end{equation*}
We define the \emph{Sobolev $f$-torsional rigidity of 
$\partial E$ relative to $\overline{\Omega}\setminus E$} 
(briefly, the \emph{$f$-torsional rigidity of $\partial E$})
as 
\begin{equation*}
\mathcal{T}_{\overline{\Omega}\setminus E}(\partial E,f):
=-2\inf\left\{J_{\Omega,E,f}(u)\colon u\in H^1(\Omega\setminus\overline{E})\right\} .
\end{equation*}
\end{definition}
By standard minimization arguments,  there exists a unique 
$U_{\Omega,E,f}\in H^1(\Omega \setminus\overline{E})$ achieving the infimum defining $\mathcal{T}_{\overline{\Omega}\setminus E}(\partial E,f)$, i.e. such that 
\begin{equation}\label{eq:U-OEf}
\mathcal{T}_{\overline{\Omega}\setminus E}(\partial E,f)=-2 J_{\Omega,E,f}(U_{\Omega,E,f}),    
\end{equation}
see Proposition \ref{prop:torsion_functions}.
We also recall the definition of Sobolev capacity of a set.
\begin{definition}\label{def:capacity}
    Let $K\sub\R^N$ be a compact set. The \emph{Sobolev capacity} of $K$ is defined as
        \[
        \mathrm{Cap}\,(K):=\inf\left\{\int_{\R^N}\!\!\left(|\nabla u|^2+u^2\right)\dx\colon u\in H^1(\R^N),~u=1\text{ a.e. in an open neighborhood of }K\right\}.
    \]
\end{definition}
If the family $\{\Sigma_\e\}_{\e\in(0,\e_0)}$ satisfies  {\bf (H)} and 
\begin{equation}\label{eq:cap-to-0}
    \lim_{\e\to 0}\mathrm{Cap}\,(\overline{\Sigma_\e})=0,
\end{equation}
under assumption \eqref{eq:simple} 
it is possible to uniquely choose the $n$-th eigenfunction of the orthonormal basis 
$\{\varphi_j^\e\}_{j\geq 0}$ in \eqref{eq:bases} in such a way that
\begin{equation}\label{eq:choose-phi-eps-1}
   \int_{\Omega_\e}\varphi_n^\e\varphi_n\dx\geq0\quad \text{for $\e$ sufficiently small}.
\end{equation}
If $\varphi_n^\e$ is chosen as above, one can prove that
\begin{equation}\label{eq:choose-phi-eps-2}
    \norm{\varphi_n^\e-\varphi_n}_{H^1(\Omega_\e)}\to 0\quad\text{as }\e\to 0,
\end{equation}
see Lemma \ref{l:con-auto}.
Furthermore, assumption {\bf (H)} implies that 
\begin{equation*}
    \mathcal{T}_{\overline{\Omega}_\e}(\partial\Sigma_\e,\partial_\nnu\varphi_n)\to 0\quad\text{as }\e\to 0,
\end{equation*}
see Corollary \ref{cor:torsion-to-zero}.

\begin{remark}\label{rem:cap-non-to-zero}
It could happen that \eqref{eq:conv-eigenvalues} holds true while \eqref{eq:cap-to-0} fails (in contrast to what happens in the Dirichlet case).
An example of this phenomenon can be found in \cite[Section 4]{Rauch1975}. More precisely, for every $j\in\N\setminus\{0\}$, let $E_j$ be the union of $j$ disjoint open balls of radius $r_j>0$, evenly spaced inside  a bounded region $\mathcal U\subset\R^N$. Let us choose the radii $r_j$ in such a way that $\lim_{j\to\infty}j r_j^N=0$ (so that $|E_j|\to0$) and 
\begin{equation}\label{eq:cond-solid}
\begin{cases}
    \lim_{j\to\infty}j r_j^{N-2}=+\infty,&\text{if }N\geq3,\\
    \lim_{j\to\infty}\frac{j}{|\log r_j|}=+\infty,&\text{if }N=2.
\end{cases}
\end{equation}
In \cite[Section 4]{Rauch1975} it is observed that, under condition \eqref{eq:cond-solid}, $K_j=\overline{E_j}$ \emph{becomes solid} in $\mathcal U$ as $j\to\infty$, i.e. 
\begin{equation*}
    \alpha_j:=\inf\left\{\frac{\int_{\mathcal U\setminus \overline{E_j}}|\nabla v|^2\dx}{\int_{\mathcal U\setminus \overline{E_j}}v^2\dx}:v\in H^1(\mathcal U\setminus \overline{E_j}), v=0\text{ on }\partial E_j\right\}
    \mathop{\longrightarrow}\limits_{j\to\infty}+\infty.
\end{equation*}
This implies that it cannot happen that 
$\lim_{j\to\infty}\mathrm{Cap}\,(\overline{E_j})=0$. 
Indeed, let us argue by contradiction and assume that $\lim_{j\to \infty}\mathrm{Cap}\,(\overline{E_j})=0$. Then, for every $j\in\N\setminus\{0\}$ there exists $u_j\in H^1(\R^N)$ such that $u_j=1$ a.e. in an open neighborhood of $\overline{E_j}$ and $\lim_{j\to\infty}\|u_j\|_{H^1(\R^N)}=0$. Let $\eta\in C^\infty_{\rm c}(\R^N)$ be such that $\eta\equiv 1$ in $\mathcal U$. Then 
\begin{equation*}
\frac{ \int_{\mathcal U}|\nabla (u_j-\eta)|^2\dx}{\int_{\mathcal U}(u_j-\eta)^2\dx}
=\frac{ \int_{\mathcal U\setminus \overline{E_j}}|\nabla (u_j-\eta)|^2\dx}{\int_{\mathcal U\setminus \overline{E_j}}(u_j-\eta)^2\dx}\geq \alpha_j
\end{equation*}
and a contradiction arises letting $j\to\infty$, since the left hand side converges to $\frac{1}{|\mathcal U|} \int_{\mathcal U}|\nabla \eta|^2\dx$.

On the other hand, the sequence of sets $\{E_j\}_j$ satisfies assumption {\bf (H)}, so that \cite[Theorem 3.1]{Rauch1975} ensures spectral stability as $j\to\infty$ for the Neumann problem under removal of the sets $E_j$.
\end{remark}

Our first result provides an  asymptotic expansion of a perturbed eigenvalue  (and  its corresponding eigenfunction) in the case it  converges to a simple eigenvalue of the limit problem.
\begin{theorem}\label{thm:main1}
            Let $n\geq 1$ be such that \eqref{eq:simple} is satisfied. Let $\{\Sigma_\e\}_{\e\in(0,\e_0)}$ satisfy assumptions {\bf (H)} and \eqref{eq:cap-to-0}.
 Then 
    \begin{multline}\label{eq:asym_eigenvalues}
        \lambda_n(\Omega_\e)=\lambda_n(\Omega)-\mathcal{T}_{\overline{\Omega}_\e}(\partial\Sigma_\e,\partial_\nnu\varphi_n)
        -\int_{\Sigma_\e}\left(\abs{\nabla\varphi_n}^2-(\lambda_n(\Omega)-1)\varphi_n^2 \right)\dx
        \\
        +o\left(\mathcal{T}_{\overline{\Omega}_\e}(\partial\Sigma_\e,\partial_\nnu\varphi_n)\right)+o\left(
        \int_{\Sigma_\e}\left(\abs{\nabla\varphi_n}^2-(\lambda_n(\Omega)-1)\varphi_n^2 \right)\dx
        \right)
        \quad\text{as $\e\to 0$}.
    \end{multline}
      In addition, if $U_\e:=U_{\Omega,\Sigma_\e,\partial_\nnu\varphi_n}\in H^1(\Omega_\e)$ denotes the function achieving $\mathcal{T}_{\overline{\Omega}_\e}(\partial\Sigma_\e,\partial_\nnu\varphi_n)$ as in \eqref{eq:U-OEf}
      and $\varphi_n^\e$ is chosen as in \eqref{eq:choose-phi-eps-1},    
       then 
   \begin{equation}\label{eq:asym_eigenfunctions}
           \norm{\varphi_n^\e-(\varphi_n-U_\e)}_{H^1(\Omega_\e)}^2=o\left(\mathcal{T}_{\overline{\Omega}_\e}(\partial\Sigma_\e,\partial_\nnu\varphi_n)\right) +O(\|\varphi_n\|_{L^2(\Sigma_\e)}^4)\quad \text{as }\e\to0
       \end{equation}
and 
\begin{multline}\label{eq:asym_eigenfunctions-2}
           \|\varphi_n^\e-\varphi_n\|_{H^1(\Omega_\e)}^2=\mathcal{T}_{\overline{\Omega}_\e}(\partial\Sigma_\e,\partial_\nnu\varphi_n)+o\left(\mathcal{T}_{\overline{\Omega}_\e}(\partial\Sigma_\e,\partial_\nnu\varphi_n)\right) \\+O(\|\varphi_n\|_{L^2(\Sigma_\e)}^4)
+O\left(\|\varphi_n\|_{L^2(\Sigma_\e)}^2\sqrt{\mathcal{T}_{\overline{\Omega}_\e}(\partial\Sigma_\e,\partial_\nnu\varphi_n)}\right)
           \quad \text{as }\e\to0.
       \end{multline}
\end{theorem}
Let us briefly describe the basic idea behind the proof of \Cref{thm:main1}. We consider the function
\[
    f_\e:=\varphi_n-U_\e.
\]
It turns out that $f_\e$ is a good approximation of the perturbed eigenfunction $\varphi_n^\e$, while encoding information from given quantities (the unperturbed eigenfunction $\varphi_n$ and the hole $\Sigma_\e$). Indeed, since by assumption 
\[
    \norm{U_\e}_{H^1(\Omega_\e)}^2=\mathcal{T}_{\overline{\Omega}_\e}(\partial\Sigma_\e,\partial_\nnu\varphi_n)\to 0\quad\text{as }\e\to 0,
\]
see Remark \ref{rmk:torsion_norm}, then $f_\e$ is close to $\varphi_n$ for small $\e$. Moreover, it satisfies an equation rather similar to that of $\varphi_n^\e$, i.e.
\[
    \begin{bvp}
        -\Delta f_\e+f_\e&=\lambda_n(\Omega)\varphi_n, && \text{in }\Omega_\e, \\
        \partial_\nnu f_\e &=0, &&\text{on }\partial\Sigma_\e.
    \end{bvp}
\]
By estimating the difference $\varphi_n^\e-f_\e$, through an abstract result known in the literature as \emph{Lemma on small eigenvalues} (originally proved in \cite{ColindeV1986}), see \Cref{lemma:sm_eig}, we obtain the expansion of the eigenvalue variation stated in \Cref{thm:main1}.

\medskip
Theorem \ref{thm:main1} applies to a fairly general framework and provides an expansion in terms of $\mathcal{T}_{\overline{\Omega}_\e}(\partial\Sigma_\e,\partial_\nnu\varphi_n)$ and $\int_{\Sigma_\e}(\abs{\nabla\varphi_n}^2-(\lambda_n(\Omega)-1)\varphi_n^2 )\dx$. We now direct our attention towards the asymptotic behaviour of these quantities, with the aim of deriving  an explicit  expansion of the eigenvalue variation in some relevant examples: the case of a hole shrinking to an interior point in  dimension $N\geq 3$ and the case of a disk-shaped hole in dimension $N=2$. 

Let us fix $x_0\in\Omega$ and an open, bounded, Lipschitz set $\Sigma\sub\R^N$, and consider a hole $\Sigma_\e=x_0+\e \Sigma$ as in \eqref{eq:hole-shrinking}.
In this case, the family $\{\Sigma_\e\}$ is concentrating  to the point $x_0$ by shrinking and maintaining the same shape $\Sigma$. Let us assume that, for some  $r_0>0$ and $\e_0>0$,
\begin{equation}\label{eq:inOmega}
x_0+\overline{B_{r_0}}\subset\Omega,\quad
B_{r_0}\setminus \e\overline{\Sigma}\text{ is connected and }
x_0+\e\overline{\Sigma}\subset x_0+B_{r_0}
\quad\text{for all }\e\in(0,\e_0),
\end{equation}
where $B_{r_0}:=\{x\in \R^N:|x|<r_0\}$ is the ball in $\R^N$ with center at $0$ and radius $r_0$.
The family $\{\Sigma_\e\}_{\e\in(0,\e_0)}$ turns out to satisfy 
assumptions {\bf (H)} and \eqref{eq:cap-to-0},  see Lemma \ref{lemma:ext};   therefore, \Cref{thm:main1} applies. Hence, if $\lambda_n(\Omega)$ is simple, the problem of finding explicit asymptotic expansions for the perturbed eigenvalue $\lambda_n(\Omega_\e)$ boils down to the analysis of  the behavior of the quantities
\begin{equation*}
    \mathcal{T}_{\overline{\Omega}_\e}(\partial\Sigma_\e,\partial_\nnu\varphi_n)\quad\text{and}\quad
    \int_{\Sigma_\e}\left(\abs{\nabla\varphi_n}^2-(\lambda_n(\Omega)-1)\varphi_n^2 \right)\dx
\end{equation*}
as $\e\to 0$.  Similarly to what happens in other singularly perturbed spectral problems (see e.g. \cite{AFHL,AO2023,AFN2020,FNO1,FNO22}), the local behavior of the normalized eigenfunction $\varphi_n$ (which is unique, up to a sign) near the point $x_0$ plays a crucial role, as described below.

For every $u\in C^\infty(\Omega)$, $y\in\Omega$ and $i\in \N$, we consider the polynomial 
\begin{equation}\label{eq:polinomi}
	P_{y,i}^u(x):=\sum_{\substack{\beta\in\N^N \\\abs{\beta}=i}}\frac{1}{\beta!}D^\beta u(y)\,x^\beta,\quad x\in\R^N,
\end{equation}
 where $|\beta|=\beta_1+\ldots +\beta_N$ and $\beta! = \beta_1!\cdot \ldots \cdot \beta_N!$ 
for all $\beta=(\beta_1,\dots,\beta_N) \in \N^N$,
with the tacit convention that
\begin{equation}\label{eq:polinomi-0}
P_{y,0}^u(x) := u(y)\quad\text{for all }x\in\R^N.
\end{equation}
In the case $y=0$, we drop the index and write
\begin{equation}\label{eq:index0}
    P_i^u:=P_{0,i}^u.
\end{equation}
\begin{definition}\label{def:vanish}
Let $u\in C^\infty(\Omega)$.  We say that $u$ \emph{vanishes of order $k\in\N$ at $y$} if 
\[
P_{y,k}^u(x)\not\equiv 0\quad\text{and}\quad 
P_{y,i}^u(x)\equiv 0\quad\text{for all $i<k$}.
\]
\end{definition}
We observe that every nontrivial solution to problem \eqref{eq:P0} is analytic in $\Omega$, hence, at any point $y\in\Omega$, it vanishes with some finite order $k\in\N$ (which is $0$ if $u(y)\neq0$) in the sense of Definition \ref{def:vanish}. 
For every analytic function $u:\Omega\to\R$, the nodal set of $u$ is defined as 
\[
    \mathcal{Z}(u):=\bigcup_{k=1}^\infty\mathcal{Z}_k(u),
\]
where,  for every $k\in\N\setminus \{0\}$, 
\[
    \mathcal{Z}_k(u):=\{x\in\Omega\colon u~\text{vanishes of order $k$ at $x$}\}.
\]
We define the \emph{regular  part} of the nodal set  as
\[
    \mathrm{Reg}\,(u):=\mathcal{Z}_1(u)=\{x\in\Omega\colon u(x)=0 ~\text{and} ~\nabla u(x)\neq 0\}
\]
and the \emph{singular part} as
\[
    \mathrm{Sing}\,(u)=\mathcal{Z}(u)\setminus\mathrm{Reg}\,(u).
\]
Our second main result establishes that, in the case of a shrinking hole, the rate of convergence of the perturbed eigenvalue to the unperturbed one depends on whether the hole is made on the singular part or not. In order to state  the result, we need  a notion of limit boundary torsional rigidity, to introduce which we recall the definition of Beppo Levi spaces.
\begin{definition}\label{def:beppo-levi}
  Let $N\geq 3$ and $E\sub \R^N$ be an open Lipschitz set. The space
  $\mathcal{D}^{1,2}(\R^N\setminus E)$ is defined as the completion of
  $C_c^\infty(\R^N\setminus E)$ with respect to the norm
\[
\norm{u}_{\mathcal{D}^{1,2}(\R^N\setminus E)}
:=\bigg(\int_{\R^N\setminus E} \abs{\nabla u}^2\dx \bigg)^{\frac{1}{2}}.
\]
\end{definition}
By classical Sobolev's inequality,  $\mathcal{D}^{1,2}(\R^N\setminus E)=\left\{u\in L^{\frac{2N}{N-2}}(\R^N\setminus E):\nabla u\in L^2(\R^N\setminus E)\right\}$.
\begin{definition}\label{def:torsion_Rn}
Let $N\geq 3$, $E\sub\R^N$ be a bounded open Lipschitz set, and $f\in L^2(\partial E)$. Let
\begin{equation*}
  \tilde{J}_{E,f}:\mathcal{D}^{1,2}(\R^N\setminus E)\to\R,\quad 
  \tilde{J}_{E,f}(u):=\frac{1}{2}\int_{\R^N\setminus E}\abs{\nabla
    u}^2\dx
  -\int_{\partial E}uf\ds.
\end{equation*}
We
define the \emph{$f$-torsional rigidity of $\partial E$ relative to
  $\R^N\setminus E$} as 
  \begin{equation*}
\tau_{\R^N\setminus E}(\partial E,f):
=-2\inf\left\{\tilde{J}_{E,f}(u)\colon u\in\mathcal{D}^{1,2}(\R^N\setminus E) \right\}.
\end{equation*}
\end{definition}
By standard minimization arguments, there exists a unique $\tilde U_{E,f}\in \mathcal{D}^{1,2}(\R^N\setminus E)$ achieving the infimum defining $\tau_{\R^N\setminus E}(\partial E,f)$, i.e.
\begin{equation}\label{eq:achieve-tau}
    \tau_{\R^N\setminus E}(\partial E,f)=-2\tilde{J}_{E,f}(\tilde U_{E,f}),
\end{equation}
see Proposition \ref{prop:torsion_functions}.

We are now ready to state our second main result, which is  based on a blow-up analysis for the quantities appearing in the asymptotic expansion in \Cref{thm:main1}. This provides the explicit  rate  of convergence of the perturbed eigenvalues, in terms of the behavior of the limit eigenfunction near the point where the hole is excised.
\begin{theorem}\label{thm:blowup_intr}
    Let $N\geq  3$, $x_0\in\Omega$, $\Sigma\sub\R^N$ be an open, bounded, Lipschitz set, $\e_0>0$ be as in  \eqref{eq:inOmega} and, for every $\e\in(0,\e_0)$, $\Sigma_\e:=x_0+\e\Sigma$.
   Let $n\geq1$ be such that $\lambda_n(\Omega)$ is simple. 
\begin{itemize}
    \item[\rm (i)] If $x_0\in\Omega\setminus \mathrm{Sing}\,(\varphi_n)$, then, as $\e\to 0$,
    \begin{align*}
        &\lambda_n(\Omega_\e)=\lambda_n(\Omega)\\ 
        &\quad-\e^N \left( \tau_{\R^N\setminus \Sigma}(\partial \Sigma,\nabla\varphi_n(x_0)\cdot\nnu)
 +|\Sigma| (|\nabla \varphi_n(x_0)|^2  - (\lambda_n(\Omega)- 1) \varphi_n^2(x_0))\right)+ o(\e^N).
    \end{align*}     
    \item[\rm (ii)]
    If $x_0\in\mathrm{Sing}\,(\varphi_n)$, then, as $\e\to 0$,
        \begin{equation*}
        \lambda_n(\Omega_\e)=\lambda_n(\Omega)  -\e^{N + 2k - 2} \left( \tau_{\R^N\setminus \Sigma}(\partial \Sigma,\partial_\nnu P^{\varphi_n}_{x_0,k}) + \int_\Sigma |\nabla P^{\varphi_n}_{x_0,k}|^2 \dx \right)+o(\e^{N+2k-2}),
    \end{equation*}
where $k\geq 2$ is the vanishing order of $\varphi_n$ at $x_0$ and  $P_{x_0,k}^{\varphi_n}$ is as in \eqref{eq:polinomi}.
    \end{itemize}
\end{theorem}

Thanks to the estimates for the norm convergence of perturbed eigenfunctions, see \eqref{eq:asym_eigenfunctions}, we are able to obtain the explicit rate of convergence in the case of a shrinking hole.

\begin{theorem}\label{thm:blow_up_eigenfunct}
    Let $N\geq 3$, $x_0\in\Omega$, $\Sigma\sub\R^N$ be an open, bounded, Lipschitz set, $\e_0>0$ be as in  \eqref{eq:inOmega}, and, for every $\e\in(0,\e_0)$,  $\Sigma_\e:=x_0+\e\Sigma$. Let $n\geq1$ be such that
     $\lambda_n(\Omega)$ is simple and $k\geq 1$ be the vanishing order of $\varphi_n-\varphi_n(x_0)$ at $x_0$. 
    Let 
    \begin{equation*}
        \Phi_\e(x):=\frac{\varphi_n^\e(\e x+x_0)-\varphi_n(x_0)}{\e^k},
    \end{equation*}
where 
    $\varphi_n^\e$ is chosen as in \eqref{eq:choose-phi-eps-1}. Then, for all $R>0$ such that $\overline{\Sigma}\sub B_R$,
    \begin{equation}\label{eq:bl_eigenfunct_th1}
        \Phi_\e\to P_{x_0,k}^{\varphi_n}-
    \tilde{U}_{\Sigma,\partial_\nnu P_{x_0,k}^{\varphi_n}}
        \quad\text{in }H^1(B_R\setminus\overline{\Sigma})~\text{as }\e\to 0,
    \end{equation}
where
    $\tilde{U}_{\Sigma,\partial_\nnu P_{x_0,k}^{\varphi_n}}\in\mathcal{D}^{1,2}(\R^N\setminus\Sigma)$ is the function achieving $\tau_{\R^N\setminus \Sigma}(\partial\Sigma,\partial_\nnu P_{x_0,k}^{\varphi_n})$ as in \eqref{eq:achieve-tau}.  
    Moreover
   \begin{equation}\label{eq:rate-conv-autofunz}
        \lim_{\e\to 0}\e^{-(N+2k-2)}\|\varphi_n^\e-\varphi_n\|_{H^1(\Omega_\e)}^2=
        \tau_{\R^N\setminus \Sigma}(\partial\Sigma,\partial_\nnu P_{x_0,k}^{\varphi_n}).
    \end{equation}
    \end{theorem}
We observe that, in Theorem \ref{thm:blowup_intr}--(ii), $k$ is actually equal to the vanishing order of $\varphi_n-\varphi_n(x_0)$, since $\varphi_n(x_0)=0$ when $x_0\in \mathrm{Sing}\,(\varphi_n)$, consistently with the notation used in Theorem \ref{thm:blow_up_eigenfunct}. We refer to Remark \ref{rem:k12} for further discussion on vanishing orders of eigenfunctions.

From \Cref{thm:blowup_intr}, one can see that the sign of the leading term in the asymptotic expansion of $\lambda_n(\Omega_\e)-\lambda_n(\Omega)$ might change depending on the position of the hole. Indeed, 
the function $f\mapsto \tau_{\R^N\setminus \Sigma}(\partial \Sigma,f)$ is continuous from $L^2(\partial\Sigma)$ into $\R$; hence $\tau_{\R^N\setminus \Sigma}(\partial \Sigma,\nabla\varphi_n(x_0)\cdot\nnu)$ is small if $|\nabla\varphi_n(x_0)|$ is small. It follows that,
if $x_0$ is close to critical points of
$\varphi_n$ which are not zeroes, then the coefficient of the leading term in the expansion 
is strictly positive (since, for $n\geq1$, we have $\lambda_n(\Omega)>1$), while close to the nodal set $\mathcal{Z}(\varphi_n)$ the coefficient is negative. A more detailed discussion is contained in the following remark.

\begin{remark}\label{rem:Gamma}
In the case of holes of type \eqref{eq:hole-shrinking} shrinking to a point $x_0$, the vanishing order of $\lambda_n(\Omega_\e)-\lambda_n(\Omega)$  is strongly influenced by the position of the point $x_0\in\Omega$. If $x_0$ lies on the singular part of the nodal set of $\varphi_n$, which is known to be at most $(N-2)$-dimensional (see \cite{Caffarelli-Friedman1985}), the eigenvalue variation vanishes with the same order as $\e^{N+2k-2}$, being $k\geq 2$ the vanishing order of $\varphi_n$ at $x_0$, and the coefficient of the term $\e^{N+2k-2}$ in the expansion of $\lambda_n(\Omega_\e)-\lambda_n(\Omega)$ is strictly negative; this implies that the expansion is sharp and 
\[
    \lambda_n(\Omega_\e)<\lambda_n(\Omega),\quad\text{for $\e$ sufficiently small}.
\]
On the other hand, if $x_0$ is outside the singular set of $\varphi_n$ and outside the set
\begin{align}\label{eq:def_Gamma}
    \Gamma=\Gamma_{\Sigma,n}:=\Big\{x\in \Omega\colon &\tau_{\R^N\setminus \Sigma}(\partial \Sigma,\nabla\varphi_n(x)\cdot\nnu)
 \\\notag
 &+|\Sigma| (|\nabla \varphi_n(x)|^2  - (\lambda_n(\Omega)- 1) \varphi_n^2(x))=0\Big\}
 \setminus \mathrm{Sing}\,(\varphi_n),
\end{align}
the rate of convergence is $\e^N$. If $x_0\in\Gamma$,   \Cref{thm:blowup_intr} just lets us know that
\[
    \lambda_n(\Omega_\e)-\lambda_n(\Omega)=o(\e^N)\quad\text{as }\e\to 0,
\]
without further information about the next non-zero term in the expansion or about the sign. 
The complement of the set $\Gamma$ in
$\Omega$ is the disjoint union of the two regions
\[
    \Omega^+:=\left\{x\in \Omega\colon \tau_{\R^N\setminus \Sigma}(\partial \Sigma,\nabla\varphi_n(x)\cdot\nnu)
 +|\Sigma| (|\nabla \varphi_n(x)|^2  - (\lambda_n(\Omega)- 1) \varphi_n^2(x))<0\right\}
\]
and
\[
    \Omega^-:=\left\{x\in \Omega\colon \tau_{\R^N\setminus \Sigma}(\partial \Sigma,\nabla\varphi_n(x)\cdot\nnu)
 +|\Sigma| (|\nabla \varphi_n(x)|^2  - (\lambda_n(\Omega)- 1) \varphi_n^2(x))>0\right\}\cup 
 \mathrm{Sing}\,(\varphi_n),
\]
in each of which the mutual position of the perturbed eigenvalue and the limit one is different. Indeed, recalling that $\lambda_n(\Omega)>1$, if $x_0\in\Omega^+$, then
\[
    \lambda_n(\Omega_\e)>\lambda_n(\Omega),\quad\text{for $\e$ sufficiently small},
\]
while, if $x_0\in\Omega^-$, then
\[
    \lambda_n(\Omega_\e)<\lambda_n(\Omega),\quad\text{for $\e$ sufficiently small}.
\]
In particular, $\mathcal{Z}(\varphi_n)\sub\Omega^-$, while $\mathop{\rm Crit}(\varphi_n)\sub\Omega^+$, where
\[
\mathop{\rm Crit}(\varphi_n):=\{x\in\Omega\colon \varphi_n(x)\neq 0~\text{and }\nabla\varphi_n(x)=0\}
\]
denotes the set of critical points outside $\mathcal{Z}(\varphi_n)$. 
\end{remark}

The asymptotic expansion obtained in \eqref{eq:asym_eigenvalues} 
 can be made completely explicit in the case of a  spherical hole. In dimension $N\geq 3$ this can be done by  calculating the limit quantity $\tau_{\R^N\setminus \Sigma}(\partial \Sigma,\partial_\nnu P^{\varphi_n}_{x_0,k})$ that appears in  Theorem \ref{thm:blowup_intr}.

\begin{theorem}\label{thm:spher}
    Let $N\geq 3$,  $x_0\in\Omega$, and $\Sigma_\e:=x_0+\e B_1$. Let $n\geq1$ be such that $\lambda_n(\Omega)$ is simple. 
    \begin{itemize}
    \item[\rm (i)]    
    If  $x_0\in\Omega\setminus\mathrm{Sing}\,(\varphi_n)$, then 
\begin{equation*}
        \lambda_n(\Oe)=\lambda_n(\Omega)-\omega_N\e^N\left( \frac{N}{N-1}\abs{\nabla\varphi_n(x_0)}^2-(\lambda_n(\Omega)-1)\varphi_n^2(x_0)\right)+o(\e^N)\quad\text{as }\e\to 0,
    \end{equation*}
    where $\omega_N=|B_1|$.
    \item[\rm (ii)]    
   If  $x_0\in\mathrm{Sing}\,(\varphi_n)$, then 
    \begin{equation*}
        \lambda_n(\Omega_\e)=\lambda_n(\Omega)  - \frac{k(N+2k-2)}{N+k-2}\e^{N + 2k - 2}\int_{\partial B_1}Y^2\ds+o(\e^{N+2k-2})\quad\text{as }\e\to 0,
    \end{equation*}
where $k\geq 2$ is the vanishing order of $\varphi_n$ at $x_0$ and 
$Y$   
is the spherical harmonic of degree $k$ given by $Y=P_{x_0,k}^{\varphi_n}\big|_{\partial B_1}$, $P_{x_0,k}^{\varphi_n}$ being as in \eqref{eq:polinomi}.
    \end{itemize}
\end{theorem}
In the case $N=2$, the blow-up argument is not helpful due to the unavailability of Hardy-type inequalities, which prevents us from identifying a concrete functional space to which the blow-up limits belong. In this case, direct computations, carried out by expanding the torsion function for the perturbed problem in Fourier series, allow us to prove the following result.
\begin{theorem}\label{thm:spher-Dim2}
    Let $N=2$,  $x_0\in\Omega$, and $\Sigma_\e:=x_0+\e B_1$. Let $n\geq1$ be such that  $\lambda_n(\Omega)$ is simple.
    \begin{itemize}
    \item[\rm (i)] If $x_0\in\Omega\setminus\mathrm{Sing}\,(\varphi_n)$, then
    \begin{equation*}
        \lambda_n(\Oe)=\lambda_n(\Omega)-\pi \e^2\left(2\abs{\nabla\varphi_n(x_0)}^2-(\lambda_n(\Omega)-1)\varphi_n^2(x_0)\right)+o(\e^2),\quad\text{as }\e\to 0.
    \end{equation*}
    \item[\rm (ii)] If $x_0\in\mathrm{Sing}\,(\varphi_n)$, then
    \begin{equation*}
        \lambda_n(\Oe)=\lambda_n(\Omega)-2k\pi \e^{2k}\left( 
    \bigg|\frac{\partial^{k}\varphi_n}{\partial x_1^k}(x_0) \bigg|^2+\frac1{k^2}\bigg|\frac{\partial^{k}\varphi_n}{\partial x_1^{k-1}
    \partial x_2}(x_0) \bigg|^2\right)+o(\e^{2k}),\quad\text{as }\e\to 0,
    \end{equation*}
    where $k\geq2$ is the vanishing order of $\varphi_n-\varphi_n(x_0)$ at $x_0$. 
    \end{itemize}
\end{theorem}
\Cref{thm:spher} and \Cref{thm:spher-Dim2} provide a more explicit expression for the interface $\Gamma$ introduced in Remark \ref{rem:Gamma} and its $2$-dimensional counterpart, in the case of a spherical hole: if $\Sigma=B_1$, we have 
\[
    \Gamma=\left\{ x\in\Omega\colon \frac{N}{N-1}\abs{\nabla\varphi_n(x)}^2-(\lambda_n(\Omega)-1)\varphi_n^2(x)=0\right\}
    \setminus \mathrm{Sing}\,(\varphi_n).
\]
Some examples of interfaces $\Gamma$ are described in  \Cref{sec:spher}, for $\Omega$ being a $3$-dimensional box or a $2$-dimensional disk.

\subsection*{Notation}
In what follows, for any family $\{\Sigma_\e\}_{\e\in (0,\e_0)}$ satisfying assumption {\bf (H)}, we  denote
\begin{equation*}
    \lambda_i:=\lambda_i(\Omega)\quad\text{and}\quad \lambda_i^\e:=\lambda_i(\Omega_\e)
\end{equation*}
for all $i\in\N$, where  $\Omega_\e:=\Omega\setminus\overline{\Sigma}_\e$.
Moreover, we fix an index $n\in\N$, $n\geq 1$, such that \eqref{eq:simple} is satisfied; we recall that $\varphi_n$ is a corresponding eigenfunction such that $\int_\Omega\varphi_n^2\dx=1$. We may also denote by
\begin{equation}\label{eq:not-tau-eps}
    \mathcal{T}_\e:=\mathcal{T}_{\overline{\Omega}_\e}(\partial\Sigma_\e,\partial_\nnu\varphi_n)
\end{equation}
the Sobolev $\partial_\nnu\varphi_n$-torsional rigidity of $\partial\Sigma_\e$ relative to $\overline{\Omega}_\e$, and by
\begin{equation}\label{eq:not-U-eps}
    U_\e:=U_{\Omega,\Sigma_\e,\partial_\nnu\varphi_n}\in H^1(\Omega_\e)
\end{equation}
the function achieving it, see \eqref{eq:U-OEf} and \Cref{prop:torsion_functions}.


\section{Preliminaries}

The first part of this section is devoted to  some basic properties of the $f$-torsional rigidity of a set.

\begin{lemma}\label{lemma:equiv}
	Let $E\sub\R^N$ be an open Lipschitz  set such that $\overline{E}\sub\Omega$ and let $f\in L^2(\partial E)$. Then
	\begin{equation}\label{eq:torsion_sup}
		\mathcal{T}_{\overline{\Omega}\setminus E}(\partial E,f)=\sup \left\{\frac{\displaystyle \left(\int_{\partial E}uf\ds\right)^2}{\displaystyle \int_{\Omega\setminus E}(\abs{\nabla u}^2+u^2)\dx}\colon u\in H^1(\Omega\setminus\overline{E})\setminus\{0\}\right\}.
	\end{equation}
\end{lemma}
\begin{proof}
By the substitution $u \mapsto t u$, the characterization of $\mathcal{T}_{\overline{\Omega}\setminus E}(\partial E,f)$ as in \Cref{def:sobolev_torsion} is equivalent to
\begin{multline}
    \label{eq:equiv_1}
		\mathcal{T}_{\overline{\Omega}\setminus E}(\partial E,f)\\
  =-2\inf\left\{\frac{t^2}{2} \int_{\Omega\setminus E}(|\nabla u|^2+u^2)\dx
	-t\int_{\partial E}uf\ds
	\colon u\in H^1(\Omega\setminus\overline{E})\setminus\{0\},\, t \in \R \right\}\\
  =-2\inf_{u\in H^1(\Omega\setminus\overline{E})\setminus\{0\}}\inf\left\{\frac{t^2}{2} \int_{\Omega\setminus E}(|\nabla u|^2+u^2)\dx
	-t\int_{\partial E}uf\ds:t\in\R\right\}.
\end{multline}
Minimizing in $t$ for a fixed $u\not\equiv0$, we find that 
$\inf_{t\in\R}\left\{\frac{t^2}{2} \int_{\Omega\setminus E}(|\nabla u|^2+u^2)\dx
	-t\int_{\partial E}uf\ds\right\}$ is
attained for
	\[
	t =\frac{\displaystyle\int_{\partial E}uf\ds}{\displaystyle\int_{\Omega\setminus E}(|\nabla u|^2+u^2)\dx}.
	\]
	Thus, substituting this into \eqref{eq:equiv_1} we complete the proof.
\end{proof}

The characterization of $\mathcal{T}_{\overline{\Omega}\setminus E}(\partial E,f)$ given in \eqref{eq:torsion_sup} easily implies  the following monotonicity property with respect to domain inclusion.
\begin{corollary}\label{c:monotonicity-domain}
    Let $\Omega_1,\Omega_2\subset\R^N$ be two connected open 
bounded Lipschitz sets and  $E\subset\R^N$ be an open Lipschitz  set such that $\overline{E}\subset\Omega_1\subset\Omega_2$.  Then, for any $f\in L^2(\partial E)$,
    \begin{equation*}
		\mathcal{T}_{\overline{\Omega_2}\setminus E}(\partial E,f)
            \leq\mathcal{T}_{\overline{\Omega_1}\setminus E}(\partial E,f).
    \end{equation*}
\end{corollary}
\begin{proof}
    If $u\in H^1(\Omega_2\setminus \overline{E})\setminus\{0\}$, then its restriction, still denoted as $u$, belongs to
    $H^1(\Omega_1\setminus \overline{E})$. If $u\equiv 0$ in $\Omega_1\setminus \overline{E}$ then $u$ has null trace on $\partial E$ so that 
    \begin{equation*}
        \frac{\left(\int_{\partial E}uf\ds\right)^2}{\int_{\Omega_2\setminus E}(\abs{\nabla u}^2+u^2)\dx}=0.
    \end{equation*}
    If $u\not\equiv 0$ in $\Omega_1\setminus \overline{E}$, then $u\in H^1(\Omega_1\setminus \overline{E})\setminus\{0\}$ and, by \eqref{eq:torsion_sup}, 
    \begin{equation*}
        \frac{\left(\int_{\partial E}uf\ds\right)^2}{\int_{\Omega_2\setminus E}(\abs{\nabla u}^2+u^2)\dx}\leq
        \frac{\left(\int_{\partial E}uf\ds\right)^2}{\int_{\Omega_1\setminus E}(\abs{\nabla u}^2+u^2)\dx}\leq \mathcal{T}_{\overline{\Omega_1}\setminus E}(\partial E,f).
    \end{equation*}
    In both cases, we have 
    \begin{equation*}
        \frac{\left(\int_{\partial E}uf\ds\right)^2}{\int_{\Omega_2\setminus E}(\abs{\nabla u}^2+u^2)\dx}\leq\mathcal{T}_{\overline{\Omega_1}\setminus E}(\partial E,f)
        \quad\text{for all $u\in H^1(\Omega_2\setminus \overline{E})\setminus\{0\}$},
    \end{equation*}
    which yields the conclusion by taking the supremum over  $H^1(\Omega_2\setminus \overline{E})\setminus\{0\}$.
\end{proof}

Another relevant  consequence of the characterization  \eqref{eq:torsion_sup} is the vanishing of the $\partial_\nnu\varphi$-torsional rigidity  of $\partial\Sigma_\e$ as $\e\to0$, whenever  the family
$\{\Sigma_\e\}_{\e\in(0,\e_0)}$ satisfies assumption {\bf (H)} and $\varphi$ is any eigenfunction of problem \eqref{eq:P0}.
\begin{corollary}\label{cor:torsion-to-zero}
Let $\{\Sigma_\e\}_{\e\in(0,\e_0)}$ be a family of sets satisfying assumptions {\bf (H)} 
and 
$\varphi$ be an eigenfunction of problem \eqref{eq:P0}. Then
    \begin{equation}\label{eq:cap-tor-to0}
    \mathcal{T}_{\overline{\Omega}_\e}(\partial\Sigma_\e,\partial_\nnu\varphi)\to 0
    \quad\text{as }\e\to 0.
\end{equation}
\end{corollary}
\begin{proof}
 For every $u\in H^1(\Omega_\e)$, the Divergence Theorem,   H\"older's inequality, and assumption (H2) yield
 \begin{align*}
     \bigg|\int_{\partial \Sigma_\e}u\,\partial_{\nnu}\varphi\ds\bigg|
     &=\bigg|\int_{\Sigma_\e}\mathop{\rm div}((\Ee u)\nabla \varphi)\dx\bigg|
     =\bigg|\int_{\Sigma_\e}\Big((\Delta \varphi)(\Ee u)+\nabla(\Ee u)\cdot \nabla \varphi\Big)\dx\bigg|\\
     &\leq \big(\|\Delta\varphi\|_{L^2(\Sigma_\e)}
     +
     \|\nabla\varphi\|_{L^2(\Sigma_\e;\R^N)}\big)\|\Ee u\|_{H^1(\Omega)}
     \\
     &\leq \mathfrak{C} \|u\|_{H^1(\Omega_\e)} 
     \big((|\lambda-1|\|\varphi\|_{L^2(\Sigma_\e)}
     +
     \|\nabla\varphi\|_{L^2(\Sigma_\e;\R^N)}\big),
 \end{align*}
  where  $\lambda$ is the eigenvalue corresponding to the eigenfunction $\varphi$.
  
 	The characterization of $\mathcal{T}_{\overline{\Omega}_\e}(\partial \Sigma_\e,\partial_{\nnu}\varphi)$ given in  \eqref{eq:torsion_sup} then implies 
	\begin{equation*}
		\mathcal{T}_{\overline{\Omega}_\e}(\partial\Sigma_\e,\partial_{\nnu}\varphi)=\sup_{\substack{u\in H^1(\Oe)\\u\neq0}}\frac{\left(\int_{\partial\Sigma_\e}u\,\partial_{\nnu}\varphi\ds\right)^2}{\|u\|_{H^1(\Omega_\e)}^2} \leq 
\mathfrak{C}^2 
     \big(|\lambda-1|\|\varphi\|_{L^2(\Sigma_\e)}
     +
     \|\nabla\varphi\|_{L^2(\Sigma_\e;\R^N)}\big)^2,
  \end{equation*}
  so that the conclusion follows from assumption (H3) and the absolute continuity of Lebesgue integral.
  \end{proof}
The following proposition states that the infimum appearing in the definition of the torsional rigidity of a set is actually achieved.

\begin{proposition}\label{prop:torsion_functions}\quad

\begin{enumerate}
    \item[\rm (i)] Let $E\sub\R^N$ be an open Lipschitz set 
such that $\overline{E}\sub\Omega$ and let $f\in L^2(\partial E)$. 
Then, there exists a unique 
$U=U_{\Omega,E,f}\in H^1(\Omega\setminus\overline{E})$ 
such that
\[
\mathcal{T}_{\overline{\Omega}\setminus E}(\partial E,f)=-2 J_{\Omega,E,f}(U),
\]
with $J_{\Omega,E,f}$ being as in \Cref{def:sobolev_torsion}. 
In addition, $U\in H^1(\Omega\setminus\overline{E})$ 
is the unique function weakly satisfying
\[
\begin{bvp}
-\Delta U+U &=0, &&\text{in }\Omega\setminus \overline{E}, \\
\partial_{\nnu}U&=0, &&\text{on }\partial\Omega, \\
\partial_{\nnu}U&=f, &&\text{on }\partial E,
\end{bvp}
\]
that is
\begin{equation}
\label{eq:variationalU}
\int_{\Omega\setminus E}(\nabla U\cdot\nabla v+Uv)\dx
=\int_{\partial E}v f\ds,\quad\text{for all }v
\in H^1(\Omega\setminus\overline{E}).
\end{equation}
\item[\rm (ii)] Let $N\geq3$, $E\sub\R^N$ be an open bounded Lipschitz set and  $f\in L^2(\partial E)$. Then,  there exists a unique 
$\tilde{U}=\tilde{U}_{E,f}\in\mathcal{D}^{1,2}(\R^N\setminus E)$ 
such that
\begin{equation*}
\tau_{\R^N\setminus E}(\partial E,f)=-2 \tilde{J}_{E,f}(\tilde{U}),
\end{equation*}
where $\tilde{J}_{E,f}$ is as in \Cref{def:torsion_Rn}. 
In addition, $\tilde{U}\in\mathcal{D}^{1,2}(\R^N\setminus E)$ 
is the unique function weakly satisfying
\[
\begin{bvp}
-\Delta\tilde{U} &=0, &&\text{in }\R^N\setminus \overline{E}, \\
\partial_{\nnu}\tilde{U}&=f, &&\text{on }\partial E,
\end{bvp}
\]
that is
\begin{equation*}
\int_{\R^N\setminus E}\nabla \tilde{U}\cdot\nabla v\dx
=\int_{\partial E}vf\ds\quad\text{for all }v
\in \mathcal{D}^{1,2}(\R^N\setminus E).
\end{equation*}
\end{enumerate}
\end{proposition}

\begin{proof}
The proof is a direct application of the Lax-Milgram lemma. In particular, concerning the proof of point (ii), we observe that the functional $v\mapsto \int_{\partial E}vf\ds$ is linear and continuous on $\mathcal{D}^{1,2}(\R^N\setminus E)$. Indeed, since $E$ is bounded, $\overline{E}\subset B$ for some  ball $B$, hence the restriction map $\mathcal{D}^{1,2}(\R^N\setminus E)\to H^1(B\setminus\overline{E})$ is continuous and  there exists a continuous trace operator from $\mathcal{D}^{1,2}(\R^N\setminus E)$ to $L^2(\partial E)$.
\end{proof}

\begin{remark}\label{rmk:torsion_norm}
We observe that 
\begin{equation*}
  \mathcal{T}_{\overline{\Omega}\setminus E}(\partial E,f)=
  \int_{\Omega\setminus E}(|\nabla
  U_{\Omega,E,f}|^2+U_{\Omega,E,f}^2)\dx=\int_{\partial E}f U_{\Omega,E,f}\ds,
\end{equation*}
as one easily obtains by choosing $v=U_{\Omega,E,f}$ in \eqref{eq:variationalU}.
Similarly,
\begin{equation*}
  \tau_{\R^N\setminus E}(\partial E,f)=
\int_{\R^N\setminus E}|\nabla \tilde{U}_{E,f}|^2\dx
  =\int_{\partial E}f \tilde{U}_{E,f}\ds.
\end{equation*}
\end{remark}

The following lemma provides a comparison between the $L^2$-norm of the torsion function and the torsional rigidity as $\e\to0$.

\begin{lemma}\label{lemma:norm2}
	Let $\{\Sigma_\e\}_{\e\in(0,\e_0)}$ satisfy assumptions {\bf (H)} and \eqref{eq:cap-to-0}. If $\mathcal{T}_\e\to 0$ as $\e\to 0$, then 
	\begin{equation*}
		\int_{\Oe}U_\e^2\dx =o(\Te),\quad\text{as }\e\to 0,
	\end{equation*}		
 with $\Te$ and $U_\e$ being as in \eqref{eq:not-tau-eps} and \eqref{eq:not-U-eps} respectively.
\end{lemma}
\begin{proof}
	Let us assume by contradiction that there exist a constant $C>0$  and a sequence
	$\{\e_j\}_{j\geq1}$ such that $\lim_{j\to\infty}\e_j=0$, $U_{\e_j}\not\equiv0$ and
	\[
	\frac{\displaystyle\int_{\Omega_{\e_j}}U_{\e_j}^2\dx}{\mathcal{T}_{\e_j}} = 
	\frac{\displaystyle\int_{\Omega_{\e_j}}U_{\e_j}^2\dx}{\norm{U_{\e_j}}_{H^1(\Omega_{\e_j})}^2} 
	\geq C\quad\text{for all }j \geq 1,
	\]
 see Remark \ref{rmk:torsion_norm}.
	For any $\e$, let us consider the extension to the whole $\Omega$ 
	of $U_\e$, i.e. 
	\[
	\tilde{U}_\e:=\Ee U_\e\in H^1(\Omega),
	\]
 being $\Ee$ as in \eqref{eq:ass_sigma_3}.
	Letting $W_j:=\tilde{U}_{\e_j}/\|\tilde{U}_{\e_j}\|_{L^2(\Omega)}$, 
	we have $\norm{W_j}_{L^2(\Omega)}=1$ and
	\[
	\norm{W_j}_{H^1(\Omega)}
	=\frac{\|\tilde{U}_{\e_j}\|_{H^1(\Omega)}}{\|\tilde{U}_{\e_j}\|_{L^2(\Omega)}}
	\leq \frac{\mathfrak{C} \norm{U_{\e_j}}_{H^1(\Omega_{\e_j})}}{\norm{U_{\e_j}}_{L^2(\Omega_{\e_j})}}\leq \frac{\mathfrak{C}}{\sqrt{ C}}.
	\]
	Therefore, there exists $W\in H^1(\Omega)$ such that, along a subsequence (still denoted by $\{W_j\}$), 	
	\begin{equation*}
		W_{j}\weak W\quad\text{weakly in }H^1(\Omega)\quad\text{and}\quad
		W_{j}\to W\quad\text{strongly in }L^2(\Omega)
	\end{equation*}
	as $j\to \infty$. From the strong $L^2(\Omega)$-convergence 
	we immediately infer that $\norm{W}_{L^2(\Omega)}=1$, 
	which in turn tells us that $W\not\equiv 0$.

Let $v\in C^\infty(\overline{\Omega})$. By assumption \eqref{eq:cap-to-0},
there exists $\{u_\e\}_{\e\in(0,1)}\subset C^\infty_{\rm c}(\R^N)$ such that $u_\e=1$  in a neighborhood of $\overline{\Sigma}_\e$ and $\|u_\e\|_{H^1(\R^N)}\to 0$ as $\e\to 0$.
Letting $v_j=v(1-u_{\e_j}\big|_\Omega)$, we observe that  $v_j\in C^\infty({\overline{\Omega}})$, $v_j\equiv 0$ in a neighbourhood of $\overline{\Sigma}_{\e_j}$, and $v_j\to v$ strongly in $H^1(\Omega)$, as $j\to\infty$. Then, from the weak convergence $W_{j}\weak W$ in $H^1(\Omega)$ it follows that
	\[
	\int_\Omega(\nabla W_{j}\cdot\nabla v_j+W_{j}v_j)\dx
	\to \int_\Omega(\nabla W\cdot\nabla v+Wv)\dx,\quad\text{as }j\to \infty.
	\]
	On the other hand, equation \eqref{eq:variationalU} and the fact that 
 $v_j\equiv 0$ in a neighbourhood of $\overline{\Sigma}_{\e_j}$
 imply that
	\begin{align*}
		\int_\Omega(\nabla W_{j}\cdot\nabla v_j+W_{j}v_j)\dx&=
		\frac{1}{\|\tilde{U}_{\e_j}\|_{L^2(\Omega)}}
		\int_{\Omega_{\e_j}}(\nabla U_{\e_j}\cdot\nabla v_j+U_{\e_j}v_j)\dx \\
		&=\frac{1}{\|\tilde{U}_{\e_j}\|_{L^2(\Omega)}}
		\int_{\partial\Sigma_{\e_j}}v_j\partial_{\nnu}\varphi_n\ds=0,
	\end{align*}
	for all $j\in\N$. Therefore, we conclude that
	\[
	\int_\Omega(\nabla W\cdot\nabla v+Wv)\dx=0
	\]
for every $v\in C^\infty(\overline{\Omega})$, and, by density, for every $v\in  H^1(\Omega)$. This  implies that $W=0$, thus giving rise to a contradiction.
\end{proof}

We conclude this section by proving \eqref{eq:choose-phi-eps-2}.
\begin{lemma}\label{l:con-auto}
    Let  $\{\Sigma_\e\}_{\e\in(0,\e_0)}$ satisfy  {\bf (H)} and \eqref{eq:cap-to-0} and $n\geq1$ be such that \eqref{eq:simple} holds. If, for every $\e\in(0,\e_0)$, $\varphi_n^\e$ is an eigenfunction of \eqref{eq:Peps} associated to the eigenvalue $\lambda_n^\e$ and chosen in such a way that $\int_{\Omega_\e}|\varphi_n^\e|^2\dx=1$ and \eqref{eq:choose-phi-eps-1} is satisfied, then $\lim_{\e\to0}\norm{\varphi_n^\e-\varphi_n}_{H^1(\Omega_\e)}=0$.    
\end{lemma}
\begin{proof}
    Since $\varphi_n^\e$ solve \eqref{eq:Peps} with $\lambda=\lambda_n^\e$, from \eqref{eq:conv-eigenvalues} and \eqref{eq:ass_sigma_3} it follows that 
    (possibly choosing $\e_0$ smaller)
    $\{\mathsf{E}_\e\varphi_n^\e\}_{\e\in (0,\e_0)}$ is bounded in $H^1(\Omega)$. Therefore, for every sequence $\e_j\to0^+$, there exist a subsequence (still denoted as $\e_j$) and $\tilde\varphi\in H^1(\Omega)$ such that $\mathsf{E}_{\e_j}\varphi_n^{\e_j}\rightharpoonup
\tilde\varphi$ weakly in $H^1(\Omega)$ as $j\to\infty$.

Let $v\in C^\infty(\overline{\Omega})$. Arguing as in the proof of Lemma \ref{lemma:norm2}, thanks to assumption \eqref{eq:cap-to-0} we can find a sequence $\{v_j\}$ such that  $v_j\in C^\infty({\overline{\Omega}})$, $v_j\equiv 0$ in a neighbourhood of $\overline{\Sigma}_{\e_j}$, and $v_j\to v$ strongly in $H^1(\Omega)$, as $j\to\infty$. From the equation satisfied by $\varphi_n^\e$ we have 
\begin{equation*}
    \int_\Omega \left(\nabla(\mathsf{E}_{\e_j}\varphi_n^{\e_j})\cdot\nabla v_j+(\mathsf{E}_{\e_j}\varphi_n^{\e_j}) v_j
    \right)\dx=\lambda_n^{\e_j}\int_\Omega (\mathsf{E}_{\e_j}\varphi_n^{\e_j}) v_j
    \dx,
\end{equation*}
passing to the limit in which we obtain, taking into account \eqref{eq:conv-eigenvalues},
\begin{equation}\label{eq:limi}
    \int_\Omega \left(\nabla\tilde\varphi\cdot\nabla v+\tilde\varphi v
    \right)\dx=\lambda_n\int_\Omega \tilde\varphi v
    \dx,
\end{equation}
for every $v\in C^\infty(\overline{\Omega})$ and hence, by density, 
for every $v\in H^1(\Omega)$.

Since, for any $p>2$, 
\begin{equation*}
    \int_{\Sigma_\e}|\mathsf{E}_{\e}\varphi_n^{\e}|^2\dx\leq 
    \left(\int_{\Omega}|\mathsf{E}_{\e}\varphi_n^{\e}|^p\dx\right)^{2/p}|\Sigma_\e|^{(p-2)/p},
\end{equation*}
by assumption \eqref{eq:ass_sigma_4}, Sobolev embeddings and boundedness of 
$\{\mathsf{E}_\e\varphi_n^\e\}_{\e\in (0,\e_0)}$ in $H^1(\Omega)$ we deduce that 
\begin{equation}\label{eq:l2su-sigma}\lim_{\e\to0}\int_{\Sigma_\e}|\mathsf{E}_{\e}\varphi_n^{\e}|^2\dx=0.
\end{equation}
Hence 
\begin{align}\label{eq:limi2}
    \int_\Omega |\tilde \varphi|^2\dx&=\lim_{j\to\infty}
    \int_\Omega |\mathsf{E}_{\e_j}\varphi_n^{\e_j}|^2\dx
    \\
    \notag &=\lim_{j\to\infty}
    \left(\int_{\Omega_{\e_j}} |\varphi_n^{\e_j}|^2\dx+
    \int_{\Sigma_{\e_j}} |\mathsf{E}_{\e_j}\varphi_n^{\e_j}|^2\dx
    \right)=\lim_{j\to\infty}(1+o(1))=1
\end{align}
and, in view of \eqref{eq:choose-phi-eps-1}, 
\begin{equation}\label{eq:limi3}
    \int_\Omega \tilde \varphi \varphi_n\dx=\lim_{j\to\infty}\int_\Omega
    (\mathsf{E}_{\e_j}\varphi_n^{\e_j})\varphi_n\dx
    =\lim_{j\to\infty}\left(\int_{\Omega_{\e_j}}
    \varphi_n^{\e_j}\varphi_n\dx+o(1)\right)\geq0.
\end{equation}
In view of assumption \eqref{eq:simple}, \eqref{eq:limi}, \eqref{eq:limi2}, and \eqref{eq:limi3} imply that $\tilde\varphi=\varphi_n$. In view of Urysohn's subsequence principle, we conclude that 
\begin{equation}\label{eq:con-est}
    \mathsf{E}_{\e}\varphi_n^\e\rightharpoonup \varphi_n\quad\text{as }\e\to0\quad\text{weakly in } H^1(\Omega).
\end{equation}
By \eqref{eq:con-est} and compactness of the embedding $H^1(\Omega)\subset L^2(\Omega)$ we have $\lim_{\e\to 0}\|\mathsf{E}_{\e}\varphi_n^\e- \varphi_n\|_{L^2(\Omega)}=0$, hence 
\begin{equation}\label{eq:con-est2}
    \|\varphi_n^\e- \varphi_n\|_{L^2(\Omega_\e)}\to 0\quad\text{as }\e\to0.
\end{equation}
Testing the equation satisfied by $\varphi_n^\e$ with
$\varphi_n^\e-\varphi_n$ and taking into account \eqref{eq:l2su-sigma} and \eqref{eq:con-est} we obtain 
\begin{align}\label{eq:f1}
    \int_{\Omega_\e}\nabla \varphi_n^\e&\cdot \nabla(\varphi_n^\e-\varphi_n)\dx=(\lambda_n^\e-1) \int_{\Omega_\e} \varphi_n^\e(\varphi_n^\e-\varphi_n)\dx\\
    \notag &=(\lambda_n^\e-1) \left(1-\int_{\Omega_\e} \varphi_n^\e\varphi_n\dx\right)\\
\notag&    =(\lambda_n^\e-1) \left(1-\int_{\Omega} (\mathsf{E}_{\e}\varphi_n^\e)\varphi_n\dx+
    \int_{\Sigma_\e} (\mathsf{E}_{\e}\varphi_n^\e)\varphi_n\dx\right)=o(1)
\end{align}
as $\e\to0$. Furthermore, 
\begin{equation*}
    \left|\int_{\Sigma_\e}\nabla\varphi_n\cdot\nabla(\mathsf{E}_{\e}\varphi_n^\e-\varphi_n)\dx   \right|\leq \| \mathsf{E}_{\e}\varphi_n^\e-\varphi_n\|_{H^1(\Omega)}\|\nabla\varphi_n\|_{L^2(\Sigma_\e)}=o(1)\quad\text{as }\e\to0,
\end{equation*}
so that, in view of  \eqref{eq:con-est},
\begin{equation}\label{eq:f2}
    \int_{\Omega_\e}\nabla \varphi_n\cdot \nabla(\varphi_n^\e-\varphi_n)\dx=\int_\Omega
    \nabla \varphi_n\cdot \nabla(\mathsf{E}_{\e}\varphi_n^\e-\varphi_n)\dx-\int_{\Sigma_\e}\nabla\varphi_n\cdot\nabla(\mathsf{E}_{\e}\varphi_n^\e-\varphi_n)\dx  =o(1)
\end{equation}
as $\e\to0$. Combining \eqref{eq:f1} and \eqref{eq:f2} we obtain 
\begin{equation}\label{eq:con-est3}
    \int_{\Omega_\e}|\nabla(\varphi_n^\e-\varphi_n)|^2\dx\to 0\quad\text{as }\e\to0
\end{equation}
The conclusion follows from \eqref{eq:con-est2} and \eqref{eq:con-est3}.
\end{proof}



\section{Asymptotics of simple eigenvalues}

The aim of this section is to prove  \Cref{thm:main1}. To this end, we apply the \enquote{Lemma on small eigenvalues}  due to  Colin de Verdiére \cite{ColindeV1986}, which is stated  in the Appendix, see \Cref{lemma:sm_eig}. The underlying idea is that  good approximations of perturbed eigenfunctions induce good approximations of perturbed eigenvalues. 

\begin{proof}[Proof of \Cref{thm:main1}]  
We first observe that, in view of \eqref{eq:cap-tor-to0} and Remark \ref{rmk:torsion_norm}, 
\begin{equation}\label{eq:main9}
\lim_{\e\to0}\|\varphi_n-U_\e\|_{L^2(\Omega_\e)}^2=
\|\varphi_n\|_{L^2(\Omega)}^2=1;
\end{equation}
 hence, possibly choosing $\e_0$ smaller from the beginning, $\varphi_n-U_\e\not\equiv0$ in $\Omega_\e$ and
\begin{equation}\label{eq:varphi-n-Ueps}
   2\geq \|\varphi_n-U_\e\|_{L^2(\Omega_\e)}\geq \frac{1}{2}
\end{equation}
for all $\e\in(0,\e_0)$.
 In order to apply \Cref{lemma:sm_eig} in our setting, we  fix $\e\in(0,\e_0)$ and define:
\begin{align*}
&\mathcal{H} := L^2(\Oe), \text{ with } (\cdot,\cdot) := (\cdot,\cdot)_{L^2(\Oe)}~\text{and }\norm{\cdot}:=\norm{\cdot}_{L^2(\Omega_\e)}; \\
&\D:= H^1(\Oe); \\
&q(u,v) := \int_{\Oe} (\nabla u\cdot\nabla v+uv) \dx -\lambda_n \int_{\Oe} uv \dx,\quad \text{for every } u,v \in \D; \\
&f := \frac{\varphi_n-U_\e}{\norm{\varphi_n-U_\e}}.
\end{align*}
We observe that
$\lambda_n^\e-\lambda_n$
is an eigenvalue of $q$ and an associated  normalized eigenfunction is given by $\varphi_n^\e$; hence assumption (i) in Lemma \ref{lemma:sm_eig} is satisfied with 
\begin{equation*}
    \lambda:=\lambda_n^\e-\lambda_n,\quad 
    \phi:=\varphi_n^\e.
\end{equation*}
Letting $H_1=\mathop{\rm span}\{\varphi_k^\e:0\leq k<n\}$
and $H_2=\overline{\mathop{\rm span}\{\varphi_k^\e:k>n\}}$, we observe that $H_1,H_2$ are mutually orthogonal in $L^2(\Omega_\e)$, $\{\phi\}^\perp=H_1\oplus H_2$, and condition \eqref{eq:ortho-q} is satisfied.

We are going to estimate the corresponding values $\delta$, $\gamma_1$, and $\gamma_2$ defined in \eqref{eqLdef-delta}, \eqref{eq:sm_eig_hp1}, and \eqref{eq:sm_eig_hp1-bis},  respectively. For what concerns the former, for any $v\in\D\setminus \{0\}$ we have
\begin{align*}
&q(f,v) = \frac{1}{\norm{\varphi_n-U_\e}}\,q(\varphi_n - U_\e,v) \\
&= \frac{1}{\norm{\varphi_n-U_\e}}\int_{\Oe} \left(\nabla (\varphi_n - U_\e) \cdot\nabla v + (1-\lambda_n) (\varphi_n - U_\e)v \right)\dx \\
&= \frac{1}{\norm{\varphi_n-U_\e}}\left(\int_{\Oe} \left(\nabla \varphi_n \cdot\nabla v + \varphi_n v \right)\dx - \int_\Oe \big(\nabla U_\e \cdot\nabla v + U_\e v\big) \dx - \lambda_n \int_\Oe (\varphi_n - U_\e)v \dx\right) \\
&= \frac{\lambda_n}{\norm{\varphi_n-U_\e}} \int_\Oe U_\e v \dx,
\end{align*}
where the last equality follows from the equations satisfied by
$\varphi_n$ and $U_\e$ respectively, see \eqref{eq:weak_form} and \eqref{eq:variationalU}. 
Combining this with \eqref{eq:varphi-n-Ueps} and the Cauchy-Schwarz inequality, we obtain that
\[
    \delta\leq 2\lambda_n\norm{U_\e}\quad\text{for every $\e\in(0,\e_0)$}.
\]
Since $\lambda_n$ is simple and  $\lim_{\e\to0}\lambda_i^\e=\lambda_i$  for all $i\in\N$, if $\e$ is sufficiently small we have 
 \begin{align*}
&\gamma_1=\inf\left\{ \frac{\abs{q(v,v)}}{\norm{v}^2}\colon v\in H_1\setminus\{0\}\right\}=\lambda_n-\lambda_{n-1}^\e>0,\\
&\gamma_2=\inf\left\{ \frac{\abs{q(v,v)}}{\norm{v}^2}\colon v\in (H_2\cap\mathcal D)\setminus\{0\}\right\}=\lambda_{n+1}^\e-\lambda_n>0,
\end{align*}
so that, if $\e$ is sufficiently small,
\begin{equation*}
    \gamma=\min\{\gamma_1,\gamma_2\}\geq \gamma_0,
\end{equation*}
where 
\begin{equation*}
    \gamma_0=\frac{1}{2} \min\left\{ \lambda_{n+1}-\lambda_n,\lambda_n-\lambda_{n-1} \right\}
\end{equation*}
is
a positive number independent of $\e$. Hence, with these estimates for $\delta$ and $\gamma$
and denoting as $\Pi_\e$ the orthogonal projection onto 
$\Span\{\varphi_n^\e\}$, i.e. 
\begin{equation}\label{eq:Pi-eps}
    \Pi_\e\colon L^2(\Omega_\e)\to L^2(\Omega_\e),\quad 
    \Pi_\e(v)= (\varphi_n^\e,v)_{L^2(\Omega_\e)}\,\varphi_n^\e,
\end{equation} 
from \Cref{lemma:sm_eig} and \eqref{eq:varphi-n-Ueps} we obtain 
\begin{equation}\label{eq:main1}
    \norm{\varphi_n-U_\e-\Pi_\e(\varphi_n-U_\e)}=\norm{f-\Pi_\e f}\|\varphi_n-U_\e\|\leq \frac{4\sqrt2\lambda_n}{\gamma_0}\,\norm{U_\e}
\end{equation}
and
\begin{equation}\label{eq:main2}
    \abs{\lambda_n^\e-\lambda_n-\xi_\e}\leq \frac{8\lambda_n^2 \norm{U_\e}^2}{\gamma_0}\left( \frac{|\lambda_n^\e-\lambda_n|}{\gamma_0}+1\right),
\end{equation}
for $\e$  sufficiently small, where
\[
    \xi_\e:=q(f,f)=\frac{q(\varphi_n-U_\e,\varphi_n-U_\e)}{\norm{\varphi_n-U_\e}^2}.
\]
At this point, we analyze  what happens asymptotically as $\e\to 0$. 
Bearing in mind that \Cref{lemma:norm2} ensures that $\norm{U_\e}^2=\norm{U_\e}_{L^2(\Omega_\e)}^2=o(\mathcal{T_\e})$ as $\e\to 0$, estimates \eqref{eq:main1} and \eqref{eq:main2} yield
\begin{equation}\label{eq:main3}
    \norm{\varphi_n-U_\e-\Pi_\e(\varphi_n-U_\e)}_{L^2(\Omega_\e)}^2=o(\mathcal{T_\e})
\end{equation}
and
\begin{equation}\label{eq:main4}
    \lambda_n^\e=\lambda_n+\xi_\e+o(\mathcal{T_\e}),
\end{equation}
as $\e\to 0$. We are now ready to establish expansions \eqref{eq:asym_eigenvalues} and \eqref{eq:asym_eigenfunctions}--\eqref{eq:asym_eigenfunctions-2}. 

\smallskip
\noindent
\textbf{Proof of \eqref{eq:asym_eigenvalues}.} We begin by expanding $\xi_\e$ as $\e\to 0$. 
By \eqref{eq:main9} we have 
\begin{equation}\label{eq:main5}
    \xi_\e=q(\varphi_n-U_\e,\varphi_n-U_\e)(1+o(1))\quad\text{as }\e\to 0.
\end{equation}
Furthermore
\begin{align}
    q(\varphi_n-U_\e,\varphi_n-U_\e) =& \int_{\Oe}(\abs{\nabla\varphi_n}^2+\varphi_n^2)\dx -\lambda_n\int_{\Oe}\varphi_n^2\dx \label{eq:main6}\\
    &\notag+\int_{\Oe}(\abs{\nabla U_\e}^2+U_\e^2)\dx-\lambda_n\int_{\Oe}U_\e^2\dx \\
    &\notag-2\left(\int_{\Oe}(\nabla U_\e\nabla\varphi_n+U_\e\varphi_n)\dx-\lambda_n\int_{\Oe}U_\e\varphi_n\dx \right).
\end{align}
Since $\varphi_n$ is an eigenfunction associated to $\lambda_n$ we have \begin{equation*}
    \int_{\Oe}(\abs{\nabla\varphi_n}^2+\varphi_n^2)\dx -\lambda_n\int_{\Oe}\varphi_n^2\dx=-\int_{\Sigma_\e}\left(\abs{\nabla\varphi_n}^2-(\lambda_n-1)\varphi_n^2 \right)\dx.
\end{equation*}
In view of  \Cref{rmk:torsion_norm} and \Cref{lemma:norm2}, the term on the second line of \eqref{eq:main6} satisfies
\begin{equation*}
    \int_{\Oe}(\abs{\nabla U_\e}^2+U_\e^2)\dx-\lambda_n\int_{\Oe}U_\e^2\dx=\mathcal{T}_\e+o(\mathcal{T}_\e)\quad\text{as }\e\to 0.
\end{equation*}
Finally, an integration by parts and \Cref{rmk:torsion_norm} allow us to rewrite  the term on the last line  of \eqref{eq:main6}  as
\begin{equation*}
    \int_{\Oe}(\nabla U_\e\cdot\nabla\varphi_n+U_\e\varphi_n)\dx-\lambda_n\int_{\Oe}U_\e\varphi_n\dx=\int_{\partial\Sigma_\e}U_\e\partial_\nnu\varphi_n\ds=\mathcal{T}_\e.
\end{equation*}
Plugging these identities into \eqref{eq:main5} and \eqref{eq:main4}, we conclude the proof of \eqref{eq:asym_eigenvalues}. 

\smallskip
\noindent
\textbf{Proof of \eqref{eq:asym_eigenfunctions}.} Let $\Pi_\e$ be as in \eqref{eq:Pi-eps}.  We claim that 
        \begin{equation}\label{eq:asym_H1}
            \norm{h_\e-\Pi_\e h_\e}_{H^1(\Omega_\e)}^2=o(\mathcal{T}_\e )\quad\text{as }\e\to 0,
        \end{equation}     
where $h_\e=\varphi_n-U_\e$. We observe that
\begin{equation*}
      \begin{bvp}
            -\Delta ( h_\e-\Pi_\e h_\e)+( h_\e-\Pi_\e h_\e)&=\lambda_n^\e( h_\e-\Pi_\e h_\e)+\lambda_n U_\e+(\lambda_n-\lambda_n^\e) h_\e, &&\text{in }\Omega_\e, \\
            \partial_\nnu( h_\e-\Pi_\e h_\e)&=0,&&\text{on }\partial\Omega_\e,
        \end{bvp}
    \end{equation*}
    in a weak sense. By testing the above equation  with $h_\e-\Pi_\e h_\e$ itself, we obtain
    \begin{multline}\label{eq:asym_H1_1}
        \norm{h_\e-\Pi_\e h_\e}_{H^1(\Omega_\e)}^2=\lambda_n^\e\norm{ h_\e-\Pi_\e h_\e}_{L^2(\Omega_\e)}^2+\lambda_n(U_\e, h_\e-\Pi_\e h_\e)_{L^2(\Omega_\e)}\\
        +(\lambda_n-\lambda_n^\e)( h_\e, h_\e-\Pi_\e h_\e)_{L^2(\Omega_\e)}.
    \end{multline}
    We are going to estimate each of the three terms on the right-hand side. Concerning the first one, thanks to \eqref{eq:main3} and \eqref{eq:conv-eigenvalues}, we have
    \begin{equation*}
        \lambda_n^\e\norm{h_\e-\Pi_\e h_\e}_{L^2(\Omega_\e)}^2=o(\mathcal{T}_\e)\quad\text{as }\e\to 0.
    \end{equation*}
    To estimate  the second term on the right hand side of \eqref{eq:asym_H1_1},  we use the Cauchy-Schwarz inequality, \Cref{lemma:norm2} and \eqref{eq:main3}. This leads to 
    \begin{equation*}
        \lambda_n(U_\e, h_\e-\Pi_\e h_\e)_{L^2(\Omega_\e)}=o(\mathcal{T}_\e)\quad\text{as }\e\to 0.
    \end{equation*}
    As far as the third term is concerned, we preliminarily observe that, by \Cref{lemma:equiv},
    \begin{equation*}
        \mathcal{T}_\e=\sup_{u\in H^1(\Omega_\e)\setminus\{0\}}\frac{\displaystyle\left(\int_{\partial\Sigma_\e}u\partial_\nnu\varphi_n\ds\right)^2}{\displaystyle\int_{\Omega_\e}(\abs{\nabla u}^2+u^2)\dx}\geq \frac{\displaystyle\left(\int_{\partial\Sigma_\e}
        \varphi_n\partial_\nnu\varphi_n\ds\right)^2}{\lambda_n},
    \end{equation*}
    which, by an integration by parts, implies that
       \begin{equation}\label{eq:st-pr-b}
    \left|\int_{\Sigma_\e}\left(\abs{\nabla\varphi_n}^2-(\lambda_n-1)\varphi_n^2 \right)\dx\right|=    
\abs{\int_{\partial\Sigma_\e}\varphi_n\partial_\nnu\varphi_n\ds}=O(\sqrt{\mathcal{T}_\e})\quad\text{as }\e\to 0.
    \end{equation}
Combining \eqref{eq:asym_eigenvalues} and \eqref{eq:st-pr-b} we obtain the rough estimate
    \begin{equation}\label{eq:st-rough}
        \lambda_n^\e-\lambda_n=O(\sqrt{\mathcal{T}_\e})\quad\text{as }\e\to 0.
    \end{equation}
The last term in \eqref{eq:asym_H1_1} can be estimated using  \eqref{eq:st-rough}, the Cauchy-Schwarz inequality, \eqref{eq:main9} and \eqref{eq:main3},
thus obtaining 
\[
        (\lambda_n-\lambda_n^\e)(h_\e, h_\e-\Pi_\e h_\e)_{L^2(\Omega_\e)}=o(\mathcal{T}_\e)\quad\text{as }\e\to 0.
    \]
    This concludes the proof of \eqref{eq:asym_H1}. 
    
By the triangle inequality, Lemma \ref{lemma:norm2}, and \eqref{eq:asym_H1}, we have
    \begin{equation}\label{eq:main10}
        \norm{\Pi_\e h_\e-\varphi_n}_{L^2(\Omega_\e)}\leq \norm{h_\e-\varphi_n}_{L^2(\Omega_\e)}+\norm{\Pi_\e h_\e-h_\e}_{L^2(\Omega_\e)}=o(\sqrt{\mathcal{T}_\e})\quad\text{as }\e\to 0,
    \end{equation}
which implies that
\begin{align}\label{eq:main11}
        \norm{\Pi_\e h_\e}_{L^2(\Omega_\e)}&=\left(1-\|\varphi_n\|_{L^2(\Sigma_\e)}^2+o(\sqrt{\mathcal{T}_\e})\right)^{1/2}\\
        &\notag=1-\frac12 \|\varphi_n\|_{L^2(\Sigma_\e)}^2+o(\sqrt{\mathcal{T}_\e})+o(\|\varphi_n\|_{L^2(\Sigma_\e)}^2)
        \quad\text{as }\e\to 0.
    \end{align}
    From \eqref{eq:main10} and the fact that $\lim_{\e\to0}\norm{\varphi_n}_{H^1(\Sigma_\e)}= 0$
    we also deduce that
    \[
        \int_{\Omega_\e}\varphi_n\,\Pi_\e h_\e\dx=1+o(1)\quad\text{as }\e\to 0,
    \]
    which, combined with \eqref{eq:main11}, implies that
    \begin{equation*}
        \int_{\Omega_\e}\varphi_n\,\frac{\Pi_\e h_\e}{\norm{\Pi_\e h_\e}_{L^2(\Omega_\e)}}\dx>0
    \end{equation*}
    for $\e$ sufficiently small. Hence, since $\varphi_n^\e\in H^1(\Omega_\e)$ is uniquely determined by the condition above, see \eqref{eq:choose-phi-eps-1}, then necessarily
    \[
        \varphi_n^\e=\frac{\Pi_\e h_\e}{\norm{\Pi_\e h_\e}_{L^2(\Omega_\e)}},
    \]
    for $\e$ sufficiently small. We finally observe that
    \begin{align*}
&        \|\varphi_n^\e-\varphi_n+U_\e\|_{H^1(\Omega_\e)}^2= \frac{1}{\norm{\Pi_\e h_\e}_{L^2(\Omega_\e)}^2}\norm{\Pi_\e h_\e-\norm{\Pi_\e h_\e}_{L^2(\Omega_\e)}\varphi_n+\norm{\Pi_\e h_\e}_{L^2(\Omega_\e)}U_\e}_{H^1(\Omega_\e)}^2 \\
        &\quad=  \frac{1}{\norm{\Pi_\e h_\e}_{L^2(\Omega_\e)}^2}\Big\lVert\Pi_\e h_\e-h_\e+\left(1-\norm{\Pi_\e h_\e}_{L^2(\Omega_\e)}\right)\varphi_n 
        +\left(\norm{\Pi_\e h_\e}_{L^2(\Omega_\e)}-1\right)U_\e\Big\rVert_{H^1(\Omega_\e)}^2.
    \end{align*}
    By the previous identity, \eqref{eq:asym_H1} and \eqref{eq:main11}  we obtain \eqref{eq:asym_eigenfunctions}. 
    To prove \eqref{eq:asym_eigenfunctions-2} we observe that, since $\norm{U_\e}_{H^1(\Omega_\e)}^2=\mathcal{T}_\e$, see Remark \ref{rmk:torsion_norm},
    \begin{align*}
   &(U_\e,\varphi_n^\e-\varphi_n+U_\e)_{H^1(\Omega_\e)}\\
   &=   
   \frac{1}{\norm{\Pi_\e h_\e}_{L^2(\Omega_\e)}}
   \Big(U_\e,\Pi_\e h_\e-h_\e+\left(1-\norm{\Pi_\e h_\e}_{L^2(\Omega_\e)}\right)\varphi_n 
        +\left(\norm{\Pi_\e h_\e}_{L^2(\Omega_\e)}-1\right)U_\e\Big)_{H^1(\Omega_\e)}\\
        &=\frac{\norm{\Pi_\e h_\e}_{L^2(\Omega_\e)}-1}{\norm{\Pi_\e h_\e}_{L^2(\Omega_\e)}}
   \mathcal T_\e+\frac{(U_\e,\Pi_\e h_\e-h_\e)_{H^1(\Omega_\e)}}{\norm{\Pi_\e h_\e}_{L^2(\Omega_\e)}}+\frac{1-\norm{\Pi_\e h_\e}_{L^2(\Omega_\e)}}{\norm{\Pi_\e h_\e}_{L^2(\Omega_\e)}}(U_\e,\varphi_n)_{H^1(\Omega_\e)},
    \end{align*}
    hence, in view of \eqref{eq:main11} and \eqref{eq:asym_H1},
    \begin{equation}\label{eq:stimanorme2}
        (U_\e,\varphi_n^\e-\varphi_n+U_\e)_{H^1(\Omega_\e)}=o(\mathcal T_\e)+
        O\left(\|\varphi_n\|_{L^2(\Sigma_\e)}^2\sqrt{\mathcal{T}_\e}\right)
    \end{equation}
    as $\e\to0$. Writing $\|\varphi_n^\e-\varphi_n\|_{H^1(\Omega_\e)}^2$ as 
\begin{equation*}
    \|\varphi_n^\e-\varphi_n\|_{H^1(\Omega_\e)}^2=
    \|U_\e\|_{H^1(\Omega_\e)}^2+
    \|\varphi_n^\e-\varphi_n+U_\e\|_{H^1(\Omega_\e)}^2-2(U_\e,\varphi_n^\e-\varphi_n+U_\e)_{H^1(\Omega_\e)},
\end{equation*}
    estimate \eqref{eq:asym_eigenfunctions-2} follows from \eqref{eq:asym_eigenfunctions}, \eqref{eq:stimanorme2}, and the fact that $\norm{U_\e}_{H^1(\Omega_\e)}^2=\mathcal{T}_\e$.
\end{proof}


\section{Blow-up analysis}

\noindent
In the present section, we focus on a particular choice of holes $\Sigma_\e$. More precisely, let 
\begin{equation}\label{eq:ipo-blow}
N\geq3,\quad  x_0\in\Omega,\quad\text{and}\quad\Sigma\text{ be a bounded open Lipschitz set
such that  \eqref{eq:inOmega} is satisfied}    
\end{equation}
 for some $\e_0,r_0>0$. Then, for every $\e\in(0,\e_0)$, we consider the hole $\Sigma_\e:=x_0+\e\Sigma$ as in \eqref{eq:hole-shrinking} and the corresponding perforated domain 
\begin{equation}\label{eq:ipo-blow2}
\Omega_\e=\Omega\setminus\overline{\Sigma_\e}=\Omega\setminus(x_0+\e\overline{\Sigma}).
\end{equation}
Without loss of generality, we can assume that $x_0=0$. We observe that the family $\{\Sigma_\e\}_{\e\in(0,\e_0)}$ defined as above satisfies assumption {\bf (H)}. Indeed, \eqref{eq:ass_sigma_1} and \eqref{eq:ass_sigma_4} follow directly from the  definition of $\Sigma_\e$ and \eqref{eq:inOmega}.  Condition \eqref{eq:ass_sigma_3} is, instead, a consequence of Lemma \ref{lemma:ext} in the Appendix.

The local behaviour of the eigenfunction $\varphi_n$ near $0$ is described  in the following proposition.
\begin{proposition}\label{prop:asympt_phi}
If ${\varphi_n}$ vanishes of order $k\geq 0$ at $0$,
then, for every $R>0$,
\[
\frac{{\varphi_n}(rx)}{r^k}\to P_k^{\varphi_n}(x)\ \text{uniformly in } \overline{B_R} \text{ and in }H^1(B_R)
\]
as $r\to 0$; furthermore, $P_k^{\varphi_n}$ is a harmonic polynomial, 
homogeneous of degree $k$. 

If $\varphi_n-\varphi_n(0)$ vanishes of order $k\geq 1$ at $0$, 
then, for every $R>0$,
\[
        \frac{\varphi_n(r x)-\varphi_n(0)}{r^k}\to P_{k}^{\varphi_n}(x)\quad\text{uniformly in } \overline{B_R} \text{ and in }H^1(B_R)
\]
and
\[
\frac{\nabla{\varphi_n}(rx)}{r^{k-1}}
\to \nabla P_{k}^{\varphi_n}(x) \quad\text{uniformly in } \overline{B_R} \text{ and in }H^1\big(B_R; \R^N\big)
\]
as $r\to 0$.
\end{proposition}

\begin{proof}
  The proof of the convergences follows from the analyticity of $\varphi_n$. Moreover, the fact that $P^{\varphi_n}_k$ is
  harmonic follows from standard scaling arguments, together with
  the fact that ${\varphi_n}$ is an eigenfunction.
\end{proof}

\begin{remark}\label{rem:k12}
    It is obvious that
 $\varphi_n$ vanishes at $0$ of order
$k\geq 1$ if and only if  $\varphi_n(0)=0$; in such a case, $\varphi_n-\varphi_n(0)$ vanishes of the same order
$ k $.
On the other hand, 
if $\varphi_n(0)\neq 0$, the vanishing order of $\varphi_n-\varphi_n(0)$ 
is necessarily equal to either $k=1$ or $k=2$; this can be easily verified by taking into account that ${\varphi_n}$ is
a solution to \eqref{eq:P0} (which is  not constant since $n\geq1$) and   comparing the Taylor expansions of $-\Delta\varphi_n+\varphi_n$ and $\lambda_n\varphi_n$.
In the case $\varphi_n(0)\neq0$ and $k=2$, $0$ is a critical point for the function
${\varphi_n}$, whereas, if $\varphi_n(0)\neq0$ and $k=1$, $0$ is a regular point outside the nodal set. 
\end{remark}

The following Hardy-type inequality on perforated balls will be crucial to identify the limit blow-up profiles.

\begin{lemma}[Hardy-type inequality]\label{lemma:hardy}
	Let $N\geq 3$ and $\Sigma\sub\R^N$ be an open,  Lipschitz set such that $\overline\Sigma\subset B_{R_0}$ 
    and $B_{R_0}\setminus\overline{\Sigma}$ is connected
    for some $R_0>0$. There exists $C_{\textup{H}}>0$, depending only on $N$ and $\Sigma$, such that
	\begin{equation}\label{eq:hardy_balls}
		\int_{B_R\setminus\Sigma}\frac{u^2}{\abs{x}^2}\dx\leq C_{\textup{H}}\left[\int_{B_R\setminus\Sigma}\abs{\nabla u}^2\dx+\frac{1}{R^2}\int_{B_R\setminus\Sigma}u^2\dx\right]
	\end{equation}
	for all $u\in H^1(B_R\setminus\overline{\Sigma})$ and $R>2R_0$. Moreover,
	\begin{equation}\label{eq:hardy_RN}	
		\int_{\R^N\setminus\Sigma}\frac{u^2}{\abs{x}^2}\dx \leq C_{\textup{H}}\int_{\R^N\setminus\Sigma}\abs{\nabla u}^2\dx	
	\end{equation}
	for all $u\in C_c^\infty(\R^N\setminus\Sigma)$.
\end{lemma}
Inequality \eqref{eq:hardy_RN} allows us to characterize
the space $\mathcal D^{1,2}(\R^N\setminus\Sigma)$ introduced in Definition \ref{def:beppo-levi} as
	\begin{equation}\label{eq:D12}
		\mathcal{D}^{1,2}(\R^N\setminus\Sigma)=\left\{u\in L^1_{\textup{loc}}(\R^N\setminus\Sigma)\colon \int_{\R^N\setminus\Sigma} \bigg(\abs{\nabla u}^2+\frac{u^2}{\abs{x}^2}\bigg)\dx<\infty \right\}.
	\end{equation}
	Furthermore,
	\[
		u\mapsto \bigg(\int_{\R^N\setminus\Sigma} \bigg(\abs{\nabla u}^2+\frac{u^2}{\abs{x}^2}\bigg)\dx\bigg)^{\frac{1}{2}}
	\]
	is an equivalent norm on $\mathcal{D}^{1,2}(\R^N\setminus\Sigma)$.

\begin{proof}[Proof of Lemma \ref{lemma:hardy}]
Let $R>2R_0$ and $u \in H^1(B_R \setminus \overline{\Sigma})$.
We define the scaled function
\[
u_R(x) := u(Rx) \in H^1\big(B_1 \setminus \tfrac1 R \overline{\Sigma} \big),
\]
as well as its extension to the whole $B_1$
\[
v_R := \mathsf{E}_{\frac 1 R} u_R \in H^1(B_1).
\]
Lemma \ref{lemma:ext}  ensures that the norm of the 
extension operator $\mathsf{E}_{\frac 1 R}$ does not depend on $R$.
Moreover
\begin{equation}
\label{eq:hardyv}
\int_{B_1}\frac{v_R^2}{\abs{x}^2}\dx
\leq C_N\left(\int_{B_1}\abs{\nabla v_R}^2\dx+\int_{B_1}v_R^2\dx\right),
\end{equation}
for some constant $C_N>0$ depending only on $N$.
The above Hardy-type inequality is classical, see, for instance,
\cite[Lemma 6.7]{FNO1} for a proof in half-balls.  In view of
\eqref{eq:hardyv}, we have
\begin{equation*}
\int_{B_1 \setminus \tfrac1 R \Sigma}\frac{v_R^2}{\abs{x}^2}\dx
\leq C_N\mathfrak{C}^2\left(\int_{B_1 \setminus \tfrac1 R \Sigma}\abs{\nabla v_R}^2\dx
+\int_{B_1 \setminus \tfrac1 R \Sigma}v_R^2\dx\right),
\end{equation*}
 with $\mathfrak{C}$ being as in Lemma
\ref{lemma:ext} with $\Omega=B_1$, $\e_0=\frac1{2R_0}$, and $r_0=\frac12$. Being $v_R$ the extension of $u_R$, 
the above inequality holds for $u_R$ as well.  Scaling back the inequality to $B_R$ yields
\begin{equation*}
R^{2-N}\int_{B_R \setminus \Sigma}\frac{u ^2 (x)}{\abs{x}^2}\dx
\leq C_N\mathfrak{C}^2 R^{-N}\left(\int_{B_R \setminus \Sigma}R^2\abs{\nabla u(x)}^2\dx
+\int_{B_R \setminus \Sigma}u ^2(x)\dx\right),
\end{equation*}
which, after a straightforward simplification, is precisely \eqref{eq:hardy_balls}.
Inequality \eqref{eq:hardy_RN} follows from \eqref{eq:hardy_balls} by letting $R\to\infty$. 
\end{proof}

The following result provides a first rough estimate of $\mathcal{T}_{\overline{\Omega}_\e}(\partial(\e\Sigma),\partial_{\nnu}\varphi_n)$.
\begin{lemma}\label{lemma:big_O}
Under assumptions \eqref{eq:ipo-blow}--\eqref{eq:ipo-blow2} with $x_0=0$,
 let $k\geq 1$ be the vanishing order of $\varphi_n-\varphi_n(0)$ at $0$. Then
	\begin{equation*}
		\mathcal{T}_{\overline{\Omega}_\e}(\partial(\e\Sigma),\partial_{\nnu}\varphi_n)=O(\e^{N+2k-2})\quad\text{as }\e\to 0.
	\end{equation*}
\end{lemma}
\begin{proof}
 For every $u\in H^1(\Omega_\e)$, by the Divergence Theorem,  H\"older's inequality, and Lemma \ref{lemma:ext} we have
 \begin{align*}
     \bigg|\int_{\partial(\e\Sigma)}u\,\partial_{\nnu}\varphi_n\ds\bigg|
     &=\bigg|\int_{\e\Sigma}\mathop{\rm div}((\Ee u)\nabla \varphi_n)\dx\bigg|
     =\bigg|\int_{\e\Sigma}\Big((\Delta \varphi_n)(\Ee u)+\nabla(\Ee u)\cdot \nabla \varphi_n\Big)\dx\bigg|\\
     &\leq \|\Delta\varphi_n\|_{L^2(\e\Sigma)}
     \|\Ee u\|_{L^{2^*}(\e\Sigma)}|\e\Sigma|^{1/N}+
     \|\nabla\varphi_n\|_{L^2(\e\Sigma;\R^N)}
     \|\nabla(\Ee u)\|_{L^{2}(\e\Sigma;\R^N)}\\
     &\leq \|\Ee u\|_{H^1(\Omega)} \left(S_{N,\Omega}\|\Delta\varphi_n\|_{L^2(\e\Sigma)}\e |\Sigma|^{1/N}+
     \|\nabla\varphi_n\|_{L^2(\e\Sigma;\R^N)}\right)\\
     &\leq \mathfrak{C} \|u\|_{H^1(\Omega_\e)} \left(S_{N,\Omega}\|\Delta\varphi_n\|_{L^2(\e\Sigma)}\e |\Sigma|^{1/N}+
     \|\nabla\varphi_n\|_{L^2(\e\Sigma;\R^N)}\right)
 \end{align*}
 where $2^*=\frac{2N}{N-2}$ is the critical Sobolev exponent (remember that in the present section we are assuming $N\geq3$) and $S_{N,\Omega}$ is the operator norm of the embedding $H^1(\Omega)\hookrightarrow L^{2^*}(\Omega)$.
 	In view of  the characterization of $\mathcal{T}_{\overline{\Omega}_\e}(\partial(\e\Sigma),\partial_{\nnu}\varphi_n)$ given in  \eqref{eq:torsion_sup}, the above estimate yields
	\begin{align}\label{eq:rough1}
		\mathcal{T}_{\overline{\Omega}_\e}(\partial(\e\Sigma),\partial_{\nnu}\varphi_n)&=\sup_{u\in H^1(\Oe)\setminus\{0\}}\frac{\left(\int_{\partial(\e\Sigma)}u\,\partial_{\nnu}\varphi_n\ds\right)^2}{\|u\|_{H^1(\Omega_\e)}^2} \\
  &\notag\leq 
  \mathfrak{C}^2  \left(S_{N,\Omega}\|\Delta\varphi_n\|_{L^2(\e\Sigma)}\e |\Sigma|^{1/N}+
     \|\nabla\varphi_n\|_{L^2(\e\Sigma;\R^N)}\right)^2.
  \end{align}
  Since $\varphi_n-\varphi_n(0)$ vanishes at $0$ with order $k\geq1$, we have 
  \begin{equation*}
  \Delta \varphi_n(x)=O(|x|^{k-2})
  \quad\text{and}\quad 
  |\nabla\varphi_n(x)|=O(|x|^{k-1})\quad\text{as $x\to0$},
  \end{equation*}
  which implies that 
  \begin{equation}\label{eq:rough2}
  \|\Delta\varphi_n\|_{L^2(\e\Sigma)}=O\left(\e^{k-2+\frac N2}\right)\quad\text{and}\quad \|\nabla\varphi_n\|_{L^2(\e\Sigma;\R^N)}=O\left(\e^{k-1+\frac N2}\right)\quad\text{as }\e\to0.
  \end{equation}
  The conclusion follows by combining \eqref{eq:rough1} and \eqref{eq:rough2}.
 \end{proof}

\begin{remark}\label{ref:rem-Teps-to-0-N=2}
Arguing as in the proof of Lemma \ref{lemma:big_O}, we can prove that, if $N=2$, 
\begin{equation*}
\mathcal{T}_{\overline{\Omega}_\e}(\partial(\e\Sigma),\partial_{\nnu}\varphi_n)=O(\e^{2(k-\delta)})\quad \text{as } \e\to 0,
\end{equation*}
for every $\delta\in(0,1)$. To prove this, it is sufficient to retrace the steps of the previous proof, using the Sobolev embedding $H^1(\Omega)\hookrightarrow L^p(\Omega)$ with $p=2/\delta$. In particular, we have, even in dimension $N=2$,  $\lim_{\e\to 0}\mathcal{T}_{\overline{\Omega}_\e}(\partial(\e\Sigma),\partial_{\nnu}\varphi_n)=0$.
\end{remark}

We are now in position to state and prove the main result of this section.

\begin{theorem}[Blow-up]
\label{th:blowup}
Under assumptions \eqref{eq:ipo-blow}--\eqref{eq:ipo-blow2} with $x_0=0$, let $k\geq 1$ be the vanishing order of $\varphi_n-\varphi_n(0)$ at $0$ and $P_k^{\varphi_n}$ be   as in
 \eqref{eq:polinomi}--\eqref{eq:index0}. Then
 	\begin{equation*}
		\lim_{\e\to0}\e^{-N-2k+2}\mathcal{T}_{\overline{\Omega}_\e}(\partial(\e\Sigma),\partial_{\nnu}\varphi_n)=\tau_{\R^N\setminus\Sigma}(\partial\Sigma,\partial_{\nnu}P_k^{\varphi_n}).
	\end{equation*}
    Furthermore, if $U_\e:=U_{\Omega,\e\Sigma,\partial_\nnu\varphi_n}\in H^1(\Omega_\e)$  is the function achieving $\mathcal{T}_{\overline{\Omega}_\e}(\partial(\e\Sigma),\partial_\nnu\varphi_n)$, see \eqref{eq:U-OEf}, and 
    \begin{equation}\label{eq:def-tilde-U-eps}
        \tilde{U}_\e(x):=\e^{-k}U_\e(\e x), \quad x\in\left(\tfrac{1}{\e}\Omega\right)\setminus\Sigma,
    \end{equation}
    then 
	\begin{equation*}
		\tilde{U}_\e\to \tilde{U}_{\Sigma,\partial_\nnu P_k^{\varphi_n}}\quad\text{ in }H^1(B_R\setminus\overline{\Sigma}),~\text{as }\e \to 0,
	\end{equation*}
	for all $R>0$ such that $\overline{\Sigma}\sub B_R$, where $\tilde{U}_{\Sigma,\partial_\nnu P_k^{\varphi_n}}\in\D^{1,2}(\R^N\setminus\Sigma)$ is the function achieving $\tau_{\R^N\setminus\Sigma}(\partial\Sigma,\partial_{\nnu} P_k^{\varphi_n})$ as in  \eqref{eq:achieve-tau}. 
\end{theorem}
\begin{proof}
Let $r_0,R_0>0$ be such that $\overline{B_{r_0}}\subset\Omega$ and $\overline\Sigma\subset B_{R_0}$, so that \eqref{eq:inOmega} is satisfied with $x_0=0$ and $\e_0=r_0/R_0$.
Let $R>R_0$ and $\e<\frac{r_0}{2R}$. Since $R<\frac{r_0}\e$ and $\frac{r_0}{\e}>2R_0$, by \Cref{lemma:hardy} and a change of variable, we have
	\begin{align*}
		\int_{B_R\setminus\Sigma}\bigg(|\nabla \tilde{U}_\e|^2&+\frac{\tilde{U}_\e^2}{\abs{x}^2}\bigg)\dx \leq \int_{B_{\frac{r_0}{\e}}\setminus\Sigma}\left(|\nabla \tilde{U}_\e|^2+\frac{\tilde{U}_\e^2}{\abs{x}^2}\right)\dx \\
		&\leq\int_{B_{\frac{r_0}{\e}}\setminus\Sigma}|\nabla \tilde{U}_\e|^2\dx+ C_{\textup{H}} \int_{B_{\frac{r_0}{\e}}\setminus\Sigma}\left(|\nabla \tilde{U}_\e|^2+\frac{\e^2}{r_0^2}\tilde{U}_\e^2\right)\dx\\
  &=\e^{-N-2k+2}\left(\int_{B_{r_0}\setminus\e\Sigma}|\nabla U_\e|^2\dx+ C_{\textup{H}} \int_{B_{r_0}\setminus\e\Sigma}\left(|\nabla {U}_\e|^2+\frac{1}{r_0^2}U_\e^2\right)\dx\right).
	\end{align*}
	Hence, by \Cref{lemma:big_O} we have
	\begin{equation}\label{eq:blow_up_1}
			\int_{B_R\setminus\Sigma}\left(|\nabla \tilde{U}_\e|^2+\frac{\tilde{U}_\e^2}{\abs{x}^2}\right)\dx\leq C_1 \e^{-N-2k+2}\mathcal{T}_{\overline{\Omega}_\e}(\partial(\e\Sigma),\partial_{\nnu}\varphi_n)\leq C_2,
	\end{equation}
	where $C_1,C_2>0$ are constants independent of $R$ and $\e$. Therefore, by a diagonal argument, for every  sequence $\e_j\to0^+$ there exists a subsequence (still denoted as $\{\e_j\}$) and
 a limit profile $\tilde{U}\in L^1_{\textup{loc}}(\R^N\setminus\Sigma)$ such that, for all $R>R_0$, $\tilde{U}\in H^1(B_R\setminus\overline{\Sigma})$  and 
\begin{equation}\label{eq:blow_up_2}
    \tilde{U}_{\e_j} \weak \tilde{U}\quad\text{as $j\to\infty$ weakly in }H^1(B_R\setminus\overline{\Sigma}).
\end{equation}	
Furthermore, by  compactness of the embedding $H^1(B_R\setminus\overline{\Sigma})\hookrightarrow L^2(B_R\setminus\overline{\Sigma})$ and of the trace map from $H^1(B_R\setminus\overline{\Sigma})$ into $L^2(\partial\Sigma)$, we also have, as $j\to\infty$,
 \begin{align}
				&\tilde{U}_{\e_j} \to \tilde{U}\quad\text{strongly in }L^2(B_R\setminus\overline{\Sigma})~\text{for all }R>R_0,\label{eq:blow_up_5} \\
		&\tilde{U}_{\e_j} \to \tilde{U}\quad\text{strongly in }L^2(\partial\Sigma).\label{eq:blow_up_3}
        \end{align}
From \eqref{eq:blow_up_1} and the weak lower semicontinuity of the norm we deduce that
	\[
		\int_{B_R\setminus\Sigma}\left(|\nabla \tilde{U}|^2+\frac{\tilde{U}^2}{\abs{x}^2}\right)\dx\leq C_2\quad\text{for all }R>R_0,
	\]
	which implies that
	\[
		\int_{\R^N\setminus\Sigma}\left(|\nabla \tilde{U}|^2+\frac{\tilde{U}^2}{\abs{x}^2}\right)\dx<+\infty
	\]
	and, consequently, that $\tilde{U}\in \mathcal{D}^{1,2}(\R^N\setminus\Sigma)$, see \eqref{eq:D12}. 
 
For any  $v\in C_c^\infty(\R^N\setminus\Sigma)$ fixed,  let $j$ be sufficiently large in order to ensure that 
 \[
		\mathop{\rm supp}v\sub B_{R_v}\setminus\Sigma\sub\tfrac{1}{\e_j}\Omega\setminus\Sigma,
	\]
	for some $R_v>R_0$. From the equation satisfied by $U_\e$, see \eqref{eq:variationalU}, and  a change of variable 
 it follows that 
	\begin{equation}\label{eq:blow_up_4}
			\int_{B_{R_v}\setminus\Sigma}(\nabla \tilde{U}_{\e_j}\cdot\nabla v+\e_j^2 \tilde{U}_{\e_j} v)\dx-\int_{\partial\Sigma}v\frac{\partial_{\nnu}\varphi_n(\e_j x)}{\e_j^{k-1}}\ds=0.
	\end{equation}
In view of \eqref{eq:blow_up_2}, \eqref{eq:blow_up_5},  and \Cref{prop:asympt_phi}, we can pass to the limit as $j\to \infty$ in \eqref{eq:blow_up_4}. Hence, by density, we obtain
	\[
		\int_{\R^N\setminus\Sigma}\nabla \tilde{U}\cdot\nabla v\dx-\int_{\partial\Sigma}v\,\partial_{\nnu} P_k^{\varphi_n}\ds=0\quad\text{for all }v\in \mathcal{D}^{1,2}(\R^N\setminus\Sigma),
	\]
	which, together with \Cref{prop:torsion_functions}, implies that $\tilde{U}=\tilde{U}_{\Sigma,\partial_\nnu P_k^{\varphi_n}}$. On the other hand, by \eqref{eq:blow_up_3} and \Cref{prop:asympt_phi} we have 
	\begin{multline*}
		\e_j^{-N-2k+2}\mathcal{T}_{\overline{\Omega}_{\e_j}}(\partial(\e_j\Sigma),\partial_{\nnu}\varphi_n)=\int_{\partial\Sigma}\tilde{U}_{\e_j}\frac{\partial_{\nnu}\varphi_n(\e_j x)}{\e_j^{k-1}}\ds\\\to \int_{\partial\Sigma}\tilde{U}_{\Sigma,\partial_\nnu P_k^{\varphi_n}}\partial_{\nnu} P^{\varphi_n}_k=\tau_{\R^N\setminus\Sigma}(\partial\Sigma,\partial_{\nnu}P^{\varphi_n}_k)
	\end{multline*}
	as $j\to\infty$. 
	Being the limit profile uniquely determined, by Urysohn's subsequence principle we conclude that the convergence statements above hold as $\e\to 0$, independently of the sequence $\{\e_j\}$ and of the subsequence. 
 
 In order to prove the strong $H^1$-convergence, we observe that, in view of the equations satisfied by $\tilde{U}_\e$ and $\tilde{U}_{\Sigma,\partial_\nnu P_k^{\varphi_n}}$, for $R>R_0$ we have
    \begin{multline}\label{eq:diff-norma}
        \int_{B_R\setminus\Sigma}\abs{\nabla(\tilde{U}_\e-\tilde{U}_{\Sigma,\partial_\nnu P_k^{\varphi_n}})}^2\dx  = \int_{\partial\Sigma}(\tilde{U}_\e-\tilde{U}_{\Sigma,\partial_\nnu P_k^{\varphi_n}})(\partial_\nnu \tilde{\varphi}_\e-\partial_\nnu P_k^{\varphi_n})\ds \\
         +\int_{\partial B_R}(\tilde{U}_\e-\tilde{U}_{\Sigma,\partial_\nnu P_k^{\varphi_n}})(\partial_\nnu  \tilde U_\e-\partial_\nnu  \tilde{U}_{\Sigma,\partial_\nnu P_k^{\varphi_n}})\ds-\e^2\int_{B_R\setminus\Sigma}(\tilde{U}_\e^2-\tilde{U}_\e \tilde{U}_{\Sigma,\partial_\nnu P_k^{\varphi_n}})\dx,
    \end{multline}
    where 
\begin{equation}\label{eq:def-tilde-varphi-eps}
 \tilde{\varphi}_\e(x):=\frac{\varphi_n(\e x)-\varphi_n(0)}{\e^k}.
\end{equation}
    Since $\tilde U_\e$ weakly solves the equation $-\Delta\tilde U_\e=-\e^2\tilde U_\e$ in $B_{2R}\setminus B_{R_0}$ and 
    the family 
    $\{\e^2\tilde U_\e\}_{0<\e<\frac{r_0}{2R}}$ is bounded in $L^2(B_{2R}\setminus B_{R_0})$, by classical elliptic
    regularity theory $\{\tilde U_\e\}_{0<\e<\frac{r_0}{2R}}$ is bounded in $H^2(B_{\frac32R}\setminus B_{(R_0+R)/2})$, so that, by continuity of the trace operator, $\{\partial_\nnu  \tilde U_\e\}_{0<\e<\frac{r_0}{2R}}$ is bounded in $L^2(\partial B_R)$. 
    This, combined with \Cref{prop:asympt_phi}, convergences \eqref{eq:blow_up_5}--\eqref{eq:blow_up_3}, and the compactness of the embedding $H^1( B_R\setminus\overline{\Sigma})\hookrightarrow
    L^2( B_R\setminus\overline{\Sigma})$, allows us to pass to the limit in \eqref{eq:diff-norma}, proving that $\nabla \tilde U_\e\to \nabla \tilde{U}_{\Sigma,\partial_\nnu P_k^{\varphi_n}}$ strongly in $L^2(B_R\setminus\overline{\Sigma})$ and completing the proof in view of \eqref{eq:blow_up_5}.
    \end{proof}

We finally have  all the necessary ingredients for the proofs of \Cref{thm:blowup_intr} and \Cref{thm:blow_up_eigenfunct}.

\begin{proof}[Proof of \Cref{thm:blowup_intr}]
    By translation, it is not restrictive to assume $x_0=0$. We first observe that the family $\{\Sigma_\e\}_{\e\in (0,\e_0)}=\{\e\Sigma\}_{\e\in (0,\e_0)}$ satisfies  the assumptions of \Cref{thm:main1}. Indeed, by scaling arguments, one can easily verify that $\abs{\e\Sigma}\to 0$ and $\mathrm{Cap}\,(\e\overline{\Sigma})\to 0$ as $\e\to 0$. Moreover, \eqref{eq:ass_sigma_3} follows from \Cref{lemma:ext}.
    In view of \Cref{thm:main1} and \Cref{th:blowup}, to obtain an explicit expansion for the perturbed eigenvalue we only have to analyze the asymptotic behavior, as $\e\to 0$, of the term
    \begin{equation*}
                \int_{\Sigma_\e}\left(\abs{\nabla\varphi_n}^2-(\lambda_n-1)\varphi_n^2 \right)\dx.
    \end{equation*}
    To start, let  us  consider   the case $0\in\Omega\setminus\mathrm{Sing}\,(\varphi_n)$.     Since $\varphi_n$ is smooth, we have 
    \begin{equation*}
        \varphi_n(x)=\varphi_n(0)+O(|x|)\quad\text{and}\quad 
                \nabla\varphi_n(x)=\nabla\varphi_n(0)+O(|x|)\quad\text{as }|x|\to0,
    \end{equation*}
    which directly yields
    \begin{equation}\label{eq:bl_intr_2}
       \int_{\e\Sigma}\left(\abs{\nabla\varphi_n}^2-(\lambda_n-1)\varphi_n^2 \right)\dx= \e^N\,\abs{\Sigma}\left( \abs{\nabla\varphi_n(0)}^2-(\lambda_n-1)\varphi_n^2(0)\right)+o(\e^N)
    \end{equation}
    as $\e\to0$.
    On   the  other  hand, to identify the order of the term $\mathcal{T}_{\overline{\Omega}_\e}(\partial(\e\Sigma),\partial_\nnu  \varphi_n)$ appearing in the expansion \eqref{eq:asym_eigenvalues}, we  distinguish  two cases: $\nabla\varphi_n(0)\neq 0$ and $\nabla\varphi_n(0)=0$. If $\nabla\varphi_n(0)\neq 0$,  we can apply \Cref{th:blowup} with $k=1$ and, since $P_1^{\varphi_n}(x)=\nabla\varphi_n(0)\cdot x$, we obtain 
    \begin{equation}\label{eq:stimaTeps1}
        \mathcal{T}_{\overline{\Omega}_\e}(\partial(\e\Sigma),\partial_\nnu  \varphi_n)=\e^N \tau_{\R^N\setminus\Sigma}(\partial\Sigma,\nabla\varphi_n(0)\cdot\nnu)+o(\e^N)\quad\text{as }\e\to 0.
    \end{equation}
    If, instead,  $\nabla\varphi_n(0)=0$,  then \Cref{th:blowup} applies with some $k \geq 2$,  thus implying  that
    \begin{equation}\label{eq:stimaTeps2}
        \mathcal{T}_{\overline{\Omega}_\e}(\partial(\e\Sigma),\partial_\nnu  \varphi_n)=o(\e^N)\quad\text{as }\e\to  0.
    \end{equation}
    Moreover, trivially,
    \begin{equation}\label{eq:stimaTeps3}
        \tau_{\R^N\setminus \Sigma}(\partial \Sigma,\nabla\varphi_n(0)\cdot\nnu)=
             \tau_{\R^N\setminus \Sigma}(\partial \Sigma,0)=0.
    \end{equation}  
    Combining \eqref{eq:stimaTeps1}, \eqref{eq:stimaTeps2}, and \eqref{eq:stimaTeps3} with \eqref{eq:bl_intr_2} we obtain (i). 
    
If $0\in\mathrm{Sing}\,(\varphi_n)$ and $k\geq 2$ is the vanishing order of $\varphi_n$ at $0$, then 
\begin{equation*}
        \varphi_n(x)=O(|x|^k)\quad\text{and}\quad 
                \nabla\varphi_n(x)=\nabla P^{\varphi_n}_k(x)+O(|x|^{k})\quad\text{as }|x|\to0,
    \end{equation*}
 thus implying that, as $\e\to0$,
\begin{align*}
 & \int_{\e\Sigma}|\nabla\varphi_n(x)|^2\dx=
       \int_{\e\Sigma}|\nabla P^{\varphi_n}_k|^2\dx+O(\e^{2k-1+N})=
\e^{N+2k-2}\left(\int_{\Sigma}|\nabla P^{\varphi_n}_k|^2\dx+o(1)\right),\\
&\int_{\e\Sigma}|\varphi_n(x)|^2\dx=O(\e^{2k+N})=o(\e^{N+2k-2}),
\end{align*}
and hence
    \begin{equation*}
       \int_{\e\Sigma}\left(\abs{\nabla\varphi_n}^2-(\lambda_n-1)\varphi_n^2 \right)\dx= \e^{N+2k-2}\left(\int_{\Sigma}|\nabla P^{\varphi_n}_k|^2\dx+o(1)\right)\quad\text{as }\e\to0.
    \end{equation*}
 Combining this and \Cref{th:blowup} with \Cref{thm:main1} we obtain (ii).
\end{proof}

\begin{proof}[Proof of \Cref{thm:blow_up_eigenfunct}]
    By translation, it is not restrictive to assume $x_0=0$. Let $u_\e:=\Phi_\e-\tilde{\varphi}_\e+\tilde U_\e$, where $\tilde{\varphi}_\e$ and $\tilde{U}_\e$ are defined in \eqref{eq:def-tilde-varphi-eps} and \eqref{eq:def-tilde-U-eps}, respectively.
    From   \eqref{eq:asym_eigenfunctions} and  the change of variable $x\mapsto \e x$ it follows that, as $\e\to 0$,
    \begin{equation*}
        \int_{\frac{1}{\e}\Omega\setminus\Sigma}\abs{\nabla u_\e}^2\dx+\e^2\int_{\frac{1}{\e}\Omega\setminus\Sigma}u_\e^2\dx \\ 
        =o\big(\e^{-N-2k+2}\mathcal{T}_{\overline{\Omega}_\e}(\partial(\e\Sigma),\partial_\nnu\varphi_n)\big)+O\big(\e^{-N-2k+2}\|\varphi_n\|_{L^2(\e\Sigma)}^4\big).
    \end{equation*}   
If $\varphi_n(0)=0$, $\varphi_n(x)=O(|x|^k)$ as $|x|\to0$, hence $\|\varphi_n\|_{L^2(\e\Sigma)}^4=O(\e^{2N+4k})=o(\e^{N+2k-2})$ as $\e\to0$. If $\varphi_n(0)\neq0$,  either $k=1$ or $k=2$ by Remark \ref{rem:k12}, so that $\|\varphi_n\|_{L^2(\e\Sigma)}^4=O(\e^{2N})=o(\e^{N+2k-2})$ as $\e\to0$. In both cases we have 
\begin{equation}\label{eq:norma-phi-n-4}
\|\varphi_n\|_{L^2(\e\Sigma)}^4=o(\e^{N+2k-2})\quad\text{as }\e\to0.    
\end{equation}
 From this and \Cref{th:blowup} we deduce that
    \begin{equation}\label{eq:bl_eigenfunct_1}
                \int_{\frac{1}{\e}\Omega\setminus\Sigma}\abs{\nabla u_\e}^2\dx+\e^2\int_{\frac{1}{\e}\Omega\setminus\Sigma}u_\e ^2\dx\to 0\quad\text{as }\e\to 0.
    \end{equation}
Let $r_0,R_0>0$ be such that $\overline{B_{r_0}}\subset\Omega$ and $\overline\Sigma\subset B_{R_0}$; let $R>R_0$ and $\e<\frac{r_0}{2R}$. By \Cref{lemma:hardy} we have
    \begin{align*}
		\int_{B_R\setminus\Sigma}\bigg(|\nabla u_\e|^2+\frac{u_\e^2}{\abs{x}^2}\bigg)\dx &\leq \int_{B_{\frac{r_0}{\e}}\setminus\Sigma}\bigg(|\nabla u_\e|^2+\frac{u_\e^2}{\abs{x}^2}\bigg)\dx \\
		&\leq\int_{B_{\frac{r_0}{\e}}\setminus\Sigma}|\nabla u_\e|^2\dx+ C_{\textup{H}} \int_{B_{\frac{r_0}{\e}}\setminus\Sigma}\bigg(|\nabla u_\e|^2+\frac{\e^2}{r_0^2}u_\e^2\bigg)\dx \\
   &\leq (C_{\textup{H}}+1)\left(\int_{\frac{1}{\e}\Omega\setminus\Sigma}\abs{\nabla u_\e}^2\dx+\frac{\e^2}{r_0^2}\int_{\frac{1}{\e}\Omega\setminus\Sigma}u_\e^2\dx\right).
	\end{align*}
    From this estimate and \eqref{eq:bl_eigenfunct_1} we deduce that
    \[
        \int_{B_R\setminus\Sigma}\bigg(|\nabla u_\e|^2+\frac{u_\e^2}{\abs{x}^2}\bigg)\dx\to 0\quad\text{as }\e\to 0
    \]
    for any $R>0$ such that $\overline{\Sigma}\sub B_R$, which implies that $u_\e\to0$ strongly in $H^1(B_R\setminus\overline{\Sigma})$ as $\e\to0$. Combining this with \Cref{prop:asympt_phi} and \Cref{th:blowup} we obtain \eqref{eq:bl_eigenfunct_th1}.

Finally,  \eqref{eq:rate-conv-autofunz} follows from \eqref{eq:asym_eigenfunctions-2}, Theorem \ref{th:blowup}, and \eqref{eq:norma-phi-n-4}.
\end{proof}


\section{The case of a spherical hole}
\label{sec:spher}

In this section, we focus on spherical  holes, deriving in this specific situation more explicit expressions for the coefficients of the asymptotic expansions obtained above. We distinguish between the cases $N\geq3$ and $N=2$.

\subsection{The case $N\geq3$}

As proved in Theorem \ref{thm:blowup_intr}, different behaviors occur depending
on the vanishing order or $\varphi_n$ at $x_0$.
The most interesting and diverse phenomena are observed when $x_0\in\Omega\setminus\mathrm{Sing}(\varphi_n)$, as in this situation the sign of the leading term in the asymptotic expansion is not always the same regardless of where the domain is perforated. In view of Theorem \ref{thm:blowup_intr}-(i), the interface $\Gamma$ defined in \eqref{eq:def_Gamma}
divides the points of $\Omega$ where a hole produces a positive sign of the eigenvalue variation $\lambda_n^\e-\lambda_n$ from those where there would be a negative sign, see Remark \ref{rem:Gamma}.
Here we focus on the specific case 
\[
\Sigma = B_1,
\]
providing the proof of Theorem \ref{thm:spher}, to which we precede the following preliminary lemma.
\begin{lemma}\label{l:spherical-torsion}
    If $N\geq 3$ and $P:\R^N\to\R$ is a harmonic polynomial homogeneous of degree $k\in \N\setminus\{0\}$, then
    \begin{equation*}
        \tau_{\R^N\setminus B_1}(\partial B_1,\partial_\nnu P)= 
        \frac{k^2}{N+k-2}\int_{\partial B_1}Y^2\ds,
    \end{equation*}
    where $Y$ is the spherical harmonic of degree $k$ given by $Y=P\big|_{\partial B_1}$.
\end{lemma}
\begin{proof}
    To determine  the torsion function 
$U := \tilde U_{B_1,\partial_\nnu P}$, we recall that $U\in \mathcal D^{1,2}(\R^N\setminus B_1)$ is the unique weak solution to  
\begin{equation}\label{eq:bvpB1}
\begin{bvp}
-\Delta U &=0, &&\text{in }\R^N\setminus \overline{B_1}, \\
\partial_{\nnu}U&=\partial_\nnu P, &&\text{on } \partial B_1.
\end{bvp}
\end{equation}
We work in spherical coordinates $(r,\boldsymbol{\theta})$ and look for solutions to \eqref{eq:bvpB1} of the form
\[
U(r, \boldsymbol{\theta}) = u(r)Y(\boldsymbol\theta),
\]
where $Y=P\big|_{\partial B_1}$. We observe that, since $P$ is harmonic and $k$-homogeneous, $Y$ is a spherical harmonic of degree $k$ and solves 
\begin{equation*}
-\Delta_{\partial B_1} Y =k (N+k-2) Y,  \quad\text{on }\partial B_1,  
\end{equation*}
where $\Delta_{\partial B_1}$ is the Laplace-Beltrami operator.
Then, we can rewrite \eqref{eq:bvpB1} as
\begin{equation}
\label{eq:bvpB2}
\begin{cases}
u''(r) + \dfrac{N-1}{r} u'(r)-\dfrac{k(N+k-2)}{r^2}u(r) =0, &\text{in }(1,+\infty), \\
u'(1)=k.&
\end{cases}
\end{equation}
The solutions to the equation in the first line of \eqref{eq:bvpB2} are of the form 
\begin{equation*}
    u(r)=c_1 r^k+c_2 r^{-(N+k-2)}
\end{equation*}
for some $c_1,c_2$. The fact that $U(r, \boldsymbol{\theta}) = u(r)Y(\boldsymbol\theta)\in \mathcal D^{1,2}(\R^N\setminus B_1)$ implies that necessarily $c_1=0$, whereas the condition $u'(1)=k$ yields $c_2=-\frac{k}{N+k-2}$. Hence, by uniqueness of the torsion function,
\[
U(r,\boldsymbol \theta) = 
-\frac{k}{N+k-2}r^{-(N+k-2)}Y(\boldsymbol\theta).
\]
We conclude that
\[
\tau_{\R^N\setminus B_1}(\partial B_1,\partial_\nnu P)= \int_{\partial B_1} U
(\partial_\nnu P)\ds 
= \frac{k^2}{N+k-2} \int_{\partial B_1}Y^2(\boldsymbol\theta)\ds,
\]
thus completing the proof.
\end{proof}

\begin{proof}[Proof of Theorem \ref{thm:spher}]
 We observe that $P(x)=\nabla \varphi_n(x_0)\cdot x$ is a harmonic polynomial  of degree $k=1$. Then Lemma \ref{l:spherical-torsion} applies and yields
 \[
\tau_{\R^N\setminus B_1}(\partial B_1,\nabla\varphi_n(x_0)\cdot\nnu)
= \frac{1}{N-1} \int_{\partial B_1} |\nabla \varphi_n(x_0) \cdot \bm{\theta}|^2 \ds.
\]
Exploiting the symmetry of the domain of integration, 
a simple computation yields
\[
\tau_{\R^N\setminus B_1}(\partial B_1,\nabla\varphi_n(x_0)\cdot\nnu)= \frac{\mathcal{H}^{N-1}(\partial B_1)}{N(N-1)}|\nabla \varphi_n(x_0)|^2 = \frac{\omega_{N}}{N-1}|\nabla \varphi_n(x_0)|^2 ,
\]
where 
$\omega_N:=|B_1|$ denotes the $N$-dimensional measure of $B_1$. Substituting the above expression for $\tau_{\R^N\setminus B_1}(\partial B_1,\nabla\varphi_n(x_0)\cdot\nnu)$ in the expansion of Theorem \ref{thm:blowup_intr}-(i), we obtain  (i).

If $x_0\in \mathrm{Sing}(\varphi_n)$, $\varphi_n$ vanishes at $x_0$ with order $k\geq2$. Then, as observed in  Proposition \ref{prop:asympt_phi}, $P_{x_0,k}^{\varphi_n}$  is a harmonic polynomial homogeneous of degree $k$. From Lemma  \ref{l:spherical-torsion} it follows that 
\begin{equation}\label{eq:sf-sing1}
    \tau_{\R^N\setminus B_1}(\partial B_1,\partial_\nnu P_{x_0,k}^{\varphi_n} )= 
        \frac{k^2}{N+k-2}\int_{\partial B_1}Y^2\ds
\end{equation}
where $Y=P_{x_0,k}^{\varphi_n}\big|_{\partial B_1}$ is a spherical harmonic of degree $k$. Furthermore, by the fact that $\Delta P_{x_0,k}^{\varphi_n}=0$ and the Divergence Theorem, we have 
\begin{equation}\label{eq:sf-sing2}
    \int_{B_1}|\nabla P_{x_0,k}^{\varphi_n}(x)|^2\dx=\int_{B_1}\mathop{\rm div}(P_{x_0,k}^{\varphi_n}\nabla P_{x_0,k}^{\varphi_n})\dx=
    \int_{\partial B_1}P_{x_0,k}^{\varphi_n}\nabla P_{x_0,k}^{\varphi_n}\cdot
    \bm{\theta}\ds=k\int_{\partial B_1}Y^2\ds.
\end{equation}
    Substituting \eqref{eq:sf-sing1}--\eqref{eq:sf-sing2} in the expansion of Theorem \ref{thm:blowup_intr}-(ii), we obtain  (ii).
\end{proof}

Thanks to Lemma \ref{l:spherical-torsion} and Theorem \ref{thm:spher}, the interface $\Gamma$ defined in \eqref{eq:def_Gamma} can be described quite explicitly in the case of spherical holes. More precisely, if $\Sigma=B_1$ we have 
\[
    \Gamma=\left\{x\in \Omega\setminus \mathrm{Sing}\,(\varphi_n)\colon h(x)=0\right\},
\]
where
\[
    h(x):=\frac{N}{N-1}|\nabla \varphi_n(x)|^2  - (\lambda_n- 1) \varphi_n^2(x).
\]
We present below the example  of spherical holes excised from $3$-dimensional boxes.
\begin{example}
     Let us consider the $3$-dimensional open box
\[
\Omega = (0,1)\times(0,\sqrt[4]{2})\times(0,\sqrt[4]{3}).
\]
It is a well-known fact (see e.g. \cite{grebenkov}) that the eigenvalues of problem \eqref{eq:P0} on $\Omega$ are simple and of the form
\[
\lambda_{n_1,n_2,n_3} = \pi^2 n_1^2 + \frac{\pi^2 n_2^2}{\sqrt{2}} + \frac{\pi^2n_3^2}{\sqrt{3}}+1,\quad n_1,n_2,n_2\in\N,
\]
and the associated eigenfunctions are, up to a normalization constant, 
\[
\varphi_{n_1,n_2,n_3}(x,y,z) = \cos\left(\pi n_1 x\right)\cos\left(\frac{\pi n_2}{\sqrt[4]{2}}y\right)\cos\left(\frac{\pi n_3}{\sqrt[4]{3}}z\right).
\]
Then the interface $\Gamma$ associated to $\varphi_{n_1,n_2,n_3}$ is characterized by the equation
\[
n_1^2\tan^2(\pi n_1 x) 
+\frac{n_2^2}{\sqrt2}\tan^2\left(\frac{\pi n_2 y}{\sqrt[4]2}\right) 
+\frac{n_3^2}{\sqrt3}\tan^2\left(\frac{\pi n_3 z}{\sqrt[4]3}\right)
-\frac23 \left(n_1^2 + \frac{n_2^2}{\sqrt 2} + \frac{n_3^2}{\sqrt 3}\right)=0.
\]
Let us consider two specific cases. The first situation of interest is the one corresponding to the the smallest nontrivial eigenvalue, namely,
\[
\lambda_{0,0,1} =\frac{\pi^2}{\sqrt{3}}+1.
\]
Here, $\Gamma$ turns out to be the union of two planes
\[
\Gamma = \left\{
(x,y,z) \in \R^3 \ : \ 
z = \frac{\sqrt[4]3}{\pi}\arctan\sqrt{\frac23} \right\} \cup 
\left\{
(x,y,z) \in \R^3 \ : \ 
z = \frac{\sqrt[4]3}{\pi} \Big(\pi - \arctan\sqrt{\frac23} \,\Big) \right\}.
\]
In Figure \ref{img2} we can see the plot of $\Gamma$ (in blue), along with the nodal set of the eigenfunction $\varphi_{0,0,1}$ (in green). 
By our analysis, if the hole is punctured between the green and a blue plane, then $\lambda_n^\e<\lambda_n$. 

\begin{figure}[ht]
\begin{center}
\includegraphics[width=6.5cm]{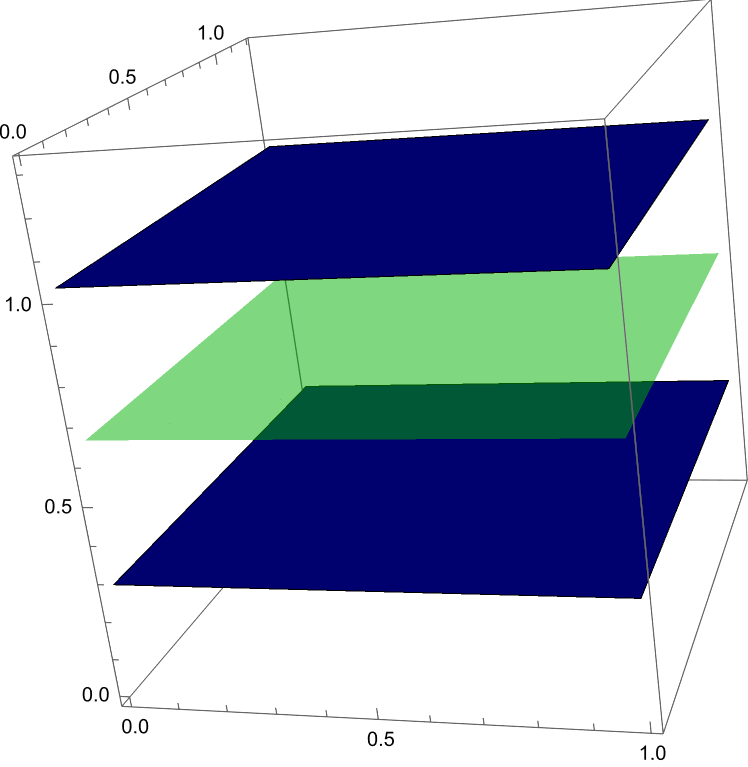}
\caption{The case $\lambda_{0,0,1}$}
\label{img2}
\end{center}
\end{figure}

Finally, we describe $\Gamma$ for $\lambda_n = \lambda_{1,1,1}$. In this case the situation is more complex, but the general picture does not change.
With the help of Mathematica\texttrademark, we can plot the set $\Gamma = \{h = 0\}$, along with the nodal set of $\varphi_{1,1,1}$ (once again in blue and green respectively). The resulting image is presented in Figure \ref{img3}.

\begin{figure}[ht]
\begin{center}
\includegraphics[width=6.5cm]{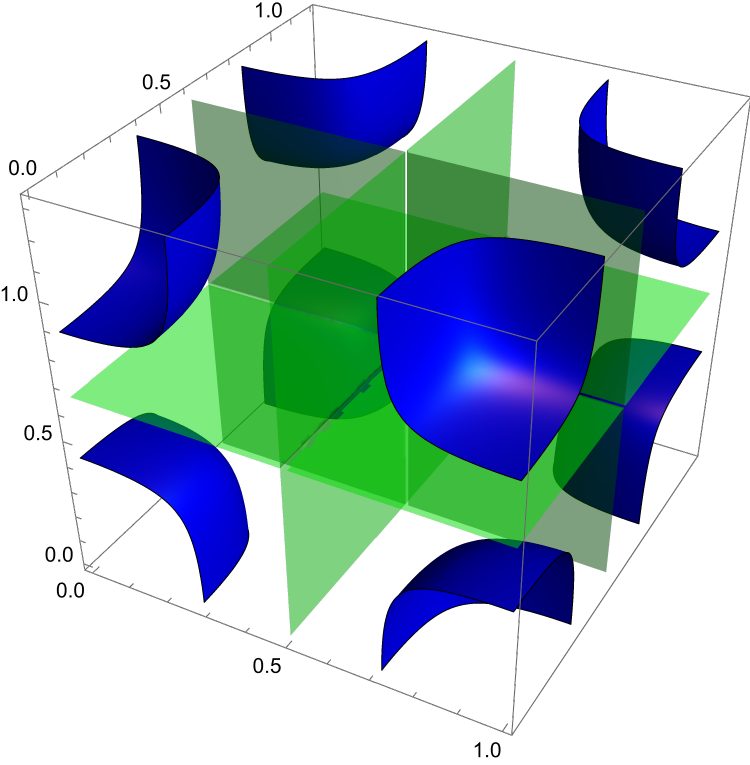}
\caption{The case $\lambda_{1,1,1}$}
\label{img3}
\end{center}
\end{figure}
\end{example}

\subsection{The case  \texorpdfstring{$N=2$}{N=2}}

In this subsection we consider the case $N=2$ and $\Sigma_\e$ being of the form
\eqref{eq:hole-shrinking} with $\Sigma=B_1$ and $x_0=0$, i.e. $\Sigma_\e=B_\e$, proving the following asymptotic expansion for $\mathcal{T}_{\overline{\Omega}\setminus B_\e}(\partial B_\e,\partial_\nnu \varphi_n)$. 
\begin{proposition}\label{p:exp-2D}
    If $N=2$ and  the vanishing order of $\varphi_n-\varphi_n(0)$  at $0$ is $k\geq1$, then
\begin{itemize}
    \item[(i)] if $\varphi_n(0)\neq0$  and $0$ is a critical point of $\varphi_n$ (hence, necessarily, $k=2$), then  
    \begin{equation*}
       \mathcal{T}_{\overline{\Omega}\setminus B_\e}(\partial B_\e,\partial_\nnu \varphi_n)=    \frac\pi2 (\lambda_n-1)^2(\varphi_n(0))^2\e^{4}|\log\e |+o(\e^{4}|\log\e |)\quad\text{as }\e\to0;
    \end{equation*}
\item[(ii)] if either  $\varphi_n(0)=0$ or $\nabla\varphi_n(0)\neq(0,0)$, then 
    \begin{equation*}
       \mathcal{T}_{\overline{\Omega}\setminus B_\e}(\partial B_\e,\partial_\nnu \varphi_n)=
    \pi k\left( 
    \bigg|\frac{\partial^{k}\varphi_n}{\partial x_1^k}(0) \bigg|^2+\frac1{k^2}\bigg|\frac{\partial^{k}\varphi_n}{\partial x_1^{k-1}
    \partial x_2}(0) \bigg|^2\right) \e^{2k}+o(\e^{2k})\quad\text{as }\e\to0.
    \end{equation*}
\end{itemize}
\end{proposition}
Let  $k\geq 1$ be the vanishing order of $\varphi_n-\varphi_n(0)$  at $0$. 
We observe that the polynomial $P^{\varphi_n}_k$ is harmonic in $\R^2$ in all cases except when $k=2$ and $\varphi_n(0)\neq0$. 
More precisely, recalling that in all critical points outside the nodal set  $\varphi_n-\varphi_n(0)$ necessarily vanishes  of order $2$, we have 
\begin{equation}\label{eq:laplacianoP}
    \Delta P^{\varphi_n}_k=
    \begin{cases}
    0,&\text{if either }    \varphi_n(0)=0\text{ or }\nabla\varphi_n(0)\neq(0,0),\\
    (1-\lambda_n)\varphi_n(0),&\text{if  }    \varphi_n(0)\neq0\text{ and $0$ is a critical point of $\varphi_n$}.
    \end{cases}
\end{equation}
By \eqref{eq:polinomi}-\eqref{eq:polinomi-0} we have, for all $j\geq1$, 
\begin{equation*}
    P^{\varphi_n}_j(r\cos t,r\sin t)=r^jf_j(t),
\end{equation*}
where 
\begin{equation*}
    f_j(t)=\sum_{i=0}^j\frac{\partial^{j}\varphi_n}{\partial x_1^i
    \partial x_2^{j-i}}(0)(\cos t)^i(\sin t)^{j-i}.
\end{equation*}
Let us consider the Fourier coefficients of $f_j$:
\begin{align}
 \label{eq:fourier-a}   &a_i^j=\frac1\pi\int_0^{2\pi}f_j(t)\cos(it)\dt,\quad i\geq0,\\
  \label{eq:fourier-b}       &b_i^j=\frac1\pi\int_0^{2\pi}f_j(t)\sin(it)\dt,\quad i\geq1.
\end{align}
We observe that 
\begin{equation}\label{eq:fourier-nulli}
    a_i^j=b_i^j=0\quad\text{if }i>j,
\end{equation}
and, by the Divergence Theorem, 
\begin{equation*}
    a_0^j=\frac1\pi\int_0^{2\pi}f_j(t)\dt=\frac1{\pi j}\int_{\partial B_1}\nabla P^{\varphi_n}_j\cdot \frac{x}{|x|}\ds=
    \frac1{\pi j}\int_{B_1}\Delta  P^{\varphi_n}_j\ds.
\end{equation*}
\begin{remark}\label{rem:j=k}
In particular, for $j=k$ we have
\begin{equation}\label{eq:a0k}
    a_0^k=
    \begin{cases}
    0,&\text{if either }    \varphi_n(0)=0\text{ or }\nabla\varphi_n(0)\neq(0,0),\\
    \dfrac{1-\lambda_n}{k}\varphi_n(0),&\text{if  }    \varphi_n(0)\neq0\text{ and $0$ is a critical point of $\varphi_n$}.
    \end{cases}
\end{equation}
Furthermore, if $j=k$ and if either  $\varphi_n(0)=0$ or $\nabla\varphi_n(0)\neq(0,0)$, then, 
by \eqref{eq:laplacianoP}, $P^{\varphi_n}_k$ is harmonic and, consequently,  there exist $c_{1},c_{2}\in\R$ such that $(c_{1},c_{2})\neq(0,0)$ and 
\begin{equation*}
  P^{\varphi_n}_k(r\cos t,r\sin t)=r^k\big(c_{1}\cos(k t)+c_{2}\sin(k t)\big),\quad r\geq0,\ t\in[0,2\pi].  
\end{equation*}
 Since
\begin{equation*}
    P^{\varphi_n}_k(r\cos t,r\sin t)=r^k\sum_{i=0}^k\frac{\partial^{k}\varphi_n}{\partial x_1^i
    \partial x_2^{k-i}}(0)(\cos t)^i(\sin t)^{k-i},
\end{equation*}
 direct computations yield
\begin{equation}\label{eq_c1c2}
    c_{1}=\frac{\partial^{k}\varphi_n}{\partial x_1^k}(0)\quad\text{and}\quad c_{2}=\frac1 k\frac{\partial^{k}\varphi_n}{\partial x_1^{k-1}
    \partial x_2}(0).
\end{equation}
Therefore, if either  $\varphi_n(0)=0$ or $\nabla\varphi_n(0)\neq(0,0)$, for  $i\geq1$ we have
\begin{equation}\label{eq:aik-bik}
a^k_i=\begin{cases}
    0,&\text{if }i\neq k,\\
    \frac{\partial^{k}\varphi_n}{\partial x_1^k}(0),&\text{if }i=k,
\end{cases}\quad \quad
b^k_i=\begin{cases}
    0,&\text{if }i\neq k,\\
    \frac1 k\frac{\partial^{k}\varphi_n}{\partial x_1^{k-1}
    \partial x_2}(0),&\text{if }i=k.
\end{cases}\end{equation}
\end{remark}
For every $j\geq1$, $R>0$, and $\e\in(0,R)$, we define 
\begin{equation*}
    T_{\e,R}^j=-2\inf\left\{ \frac{1}{2}\int_{B_R\setminus B_\e}\!\!|\nabla u|^2\dx
-\int_{\partial B_\e}\!\!(\partial_\nnu P^{\varphi_n}_j)\,u\ds\colon u \in H^1(B_R\setminus \overline{B_\e}),\ \int_{B_R\setminus B_\e}\!\!u\dx=0 \right\}.
\end{equation*}
The above infimum is achieved by a unique function $W_{\e,R,j}\in H^1(B_R\setminus \overline{B_\e})$ satisfying 
\begin{equation*}
\int_{B_R\setminus B_\e}W_{\e,R,j}\dx=0, 
\end{equation*}
 and 
\begin{align}\label{eq:eq-We}
    \int_{B_R\setminus B_\e}\nabla W_{\e,R,j}\cdot\nabla v\dx&=\int_{\partial B_\e}\partial_\nnu P^{\varphi_n}_j\left(v-\frac1{|B_R\setminus B_\e|}\int_{B_R\setminus B_\e}v\dx\right)\ds\\
   \notag &=\int_{\partial B_\e}(\partial_\nnu P^{\varphi_n}_j)\,v\ds+\frac{j a_0^j \e^j}{R^2-\e^2}\int_{B_R\setminus B_\e}v\dx
\end{align}
for every $v\in H^1(B_R\setminus \overline{B_\e})$, 
i.e. $W_{\e,R,j}$ is the unique zero-average weak solution to
\begin{equation}\label{eq:Weps}
    \begin{cases}
   -\Delta W_{\e,R,j}=\frac{j a_0^j \e^j}{R^2-\e^2}, &\text{in }B_R\setminus B_\e, \\
\partial_{\nnu}W_{\e,R,j}=0, &\text{on }\partial B_R, \\
\partial_{\nnu}W_{\e,R,j}=\partial_\nnu P^{\varphi_n}_j, &\text{on }\partial B_\e.
    \end{cases}
\end{equation}
\begin{lemma}\label{l:teRj}
 For every $j\geq1$ and $R>0$ 
 \begin{equation}\label{eq:asyTeR}
  T_{\e,R}^j=
    \begin{cases}
        \frac12\pi j^2(a^j_0)^2\e^{2j}|\log\e|+o(\e^{2j}|\log\e|),&\text{if } a^j_0\neq0,\\[5pt]
        \pi j^2 \left({\displaystyle{\sum_{i=1}^j\frac{(a^j_i)^2+ (b^j_i)^2 }{i}}}\right)\e^{2j}+o(\e^{2j}),&\text{if } a^j_0=0,
    \end{cases}   
 \end{equation}
as $\e\to0$,  
 with $a^j_i,b^j_i$ being as in \eqref{eq:fourier-a}--\eqref{eq:fourier-b}.   Moreover,
 \begin{equation}\label{eq:oTeR}
     \int_{B_R\setminus B_\e}W_{\e,R,j}^2\dx=o(T_{\e,R}^j)\quad 
  \text{as }\e\to0.
 \end{equation}
\end{lemma}
\begin{proof}
   For $j\geq1$ and $R>0$ fixed, let us expand $W_{\e,R,j}$ in Fourier series:
   \begin{equation*}
       W_{\e,R,j}(r\cos t,r\sin t)=\frac{\varphi_{0,\e}(r)}{2}+\sum_{i=1}^\infty\Big(\varphi_{i,\e}(r)\cos (it)+
       \psi_{i,\e}(r)\sin (it)\Big),
   \end{equation*}
where 
\begin{align*}
    &\varphi_{i,\e}(r)=\frac1\pi\int_0^{2\pi}W_{\e,R,j}(r\cos t,r\sin t)\cos(it)\dt,\quad i\geq0,\\
        &\psi_{i,\e}(r)=\frac1\pi\int_0^{2\pi}W_{\e,R,j}(r\cos t,r\sin t)\sin(it)\dt,\quad i\geq1.
\end{align*}
From \eqref{eq:Weps} and the fact that $\int_{B_R\setminus B_\e}W_{\e,R,j}\dx=0$ it follows that  
the function $\varphi_{0,\e}$ solves the problem
\begin{equation}\label{eq:equazione-phi0}
    \begin{cases}
        -\varphi_{0,\e}''(r)-\dfrac1r \varphi_{0,\e}'(r)=
        \dfrac{2ja_0^j\e^j}{R^2-\e^2},&\text{in }(\e,R),\\[5pt]
        \varphi_{0,\e}'(\e)=j\e^{j-1}a^j_0,\\[5pt]
    \varphi_{0,\e}'(R)=0,\\[5pt]
    {\displaystyle{\int_\e^R}} r\varphi_{0,\e}(r)\dr=0,
    \end{cases}
\end{equation}
while 
the functions $\varphi_{i,\e}$ and $\psi_{i,\e}$ with $i\geq1$ solve 
\begin{equation*}
    \begin{cases}
        -\varphi_{i,\e}''(r)-\dfrac1r \varphi_{i,\e}'(r)+\dfrac{i^2}{r^2}\varphi_{i,\e}(r)=0,&\text{in }(\e,R),\\[5pt]
        \varphi_{i,\e}'(\e)=j\e^{j-1}a^j_i,\\[5pt]
    \varphi_{i,\e}'(R)=0,
    \end{cases}
\end{equation*}
and 
\begin{equation*}
    \begin{cases}
        -\psi_{i,\e}''(r)-\dfrac1r \psi_{i,\e}'(r)+\dfrac{i^2}{r^2}\psi_{i,\e}(r)=0,&\text{in }(\e,R),\\[5pt]
        \psi_{i,\e}'(\e)=j\e^{j-1}b^j_i,\\[5pt]
    \psi_{i,\e}'(R)=0,
    \end{cases}
\end{equation*}
respectively. For $i\geq1$, direct computations yield
\begin{equation}\label{varphi_ie-psi_ie}
    \varphi_{i,\e}(r)=-\frac{j a^j_i\e^{i+j}}{i(R^{2i}-\e^{2i})}(r^i+R^{2i}r^{-i})
    \quad\text{and}\quad 
    \psi_{i,\e}(r)=-\frac{j b^j_i\e^{i+j}}{i(R^{2i}-\e^{2i})}(r^i+R^{2i}r^{-i}).
\end{equation}
In particular, by \eqref{eq:fourier-nulli} we have $\varphi_{i,\e}\equiv \psi_{i,\e}\equiv 0$ if $i>j$. 
Moreover, the unique solution to \eqref{eq:equazione-phi0} is
\begin{equation}\label{eq:varphi_0e}
    \varphi_{0,\e}(r)=
    \frac{j a^j_0 \e^j}{1-\big(\frac{\e}{R}\big)^2}
    \left(
    \log r -\frac{r^2}{2R^2}+\frac{1}{2}+\frac{\e^2\log \e-R^2\log R}{R^2-\e^2}+\frac{1}{4R^2}(R^2+\e^2) \right).
\end{equation}
We conclude that the unique zero-average weak solution to
\eqref{eq:Weps} is given by
\begin{equation*}
     W_{\e,R,j}(r\cos t,r\sin t)=\frac{\varphi_{0,\e}(r)}{2}+\sum_{i=1}^j\Big(\varphi_{i,\e}(r)\cos (it)+
       \psi_{i,\e}(r)\sin (it)\Big),
\end{equation*}
with $\varphi_{0,\e}$ as in \eqref{eq:varphi_0e} and 
$\varphi_{i,\e},\psi_{i,\e}$ as in 
\eqref{varphi_ie-psi_ie}.
Furthermore, by Parseval's Theorem,
\begin{align*}
    T_{\e,R}^j&=\int_{\partial B_\e}(\partial_\nnu P^{\varphi_n}_j) W_{\e,R,j}\ds\\
&=\e\int_0^{2\pi}\partial_\nnu P^{\varphi_n}_j(\e\cos t,\e\sin t) W_{\e,R,j}(\e\cos t,\e\sin t)\,dt\\
&=\e(-j\e^{j-1})\pi
\left(
\frac{a^j_0\varphi_{0,\e}(\e)}{2}+\sum_{i=1}^j(a_i^j\varphi_{i,\e}(\e)+
b_i^j\psi_{i,\e}(\e))\right).
\end{align*}
We observe that, by \eqref{eq:varphi_0e},
\begin{equation*}
\frac{a^j_0\varphi_{0,\e}(\e)}{2}\sim \frac12 j (a^j_0)^2 \e^j\log \e\quad\text{as }\e\to0,
\end{equation*}
while \eqref{varphi_ie-psi_ie} implies 
\begin{equation*}
    \sum_{i=1}^j(a_i^j\varphi_{i,\e}(\e)+
b_i^j\psi_{i,\e}(\e))\sim
-j\e^j\sum_{i=1}^j\frac{(a^j_i)^2+ (b^j_i)^2 }{i}\quad\text{as }\e\to0,
\end{equation*}
thus proving \eqref{eq:asyTeR}. 

By Parseval's Theorem we have
\begin{equation}\label{eq:parseval}
    \int_{B_R\setminus B_\e}W_{\e,R,j}^2\dx=\pi
    \int_\e^R r\left(
\frac{\varphi_{0,\e}^2(r)}{2}+\sum_{i=1}^j\Big(\varphi_{i,\e}^2(r)+ \psi_{i,\e}^2(r)\Big)\dr    
    \right).
    \end{equation}
By \eqref{eq:asyTeR} and 
\eqref{eq:varphi_0e}
\begin{align*}
    \int_\e^R r\varphi_{0,\e}^2(r)\dr&=
    \begin{cases}
        0,&\text{if } a^j_0=0,\\[5pt]
        O(\e^{2j})=o(T_{\e,R}^j),&\text{if } a^j_0\neq0,
    \end{cases}     \\
&    =o(T_{\e,R}^j)\quad\text{as }\e\to0,
\end{align*}
and, by \eqref{eq:asyTeR} and \eqref{varphi_ie-psi_ie},
\begin{align*}
   & \int_\e^R r
    (\varphi_{i,\e}^2(r)+ \psi_{i,\e}^2(r))\dr
=\frac{j^2\e^{2i+2j}((a^j_i)^2+(b^j_i)^2)}{i^2(R^{2i}-\e^{2i})^{2}}
\int_\e^R r(r^{2i}+R^{4i}r^{-2i}+2R^{2i})\dr   \\
&\quad=\frac{j^2\e^{2i+2j}((a^j_i)^2+(b^j_i)^2)}{i^2(R^{2i}-\e^{2i})^{2}}
\left(
\frac{R^{2i+2}-\e^{2i+2}}{2i+2}+R^{4i}\frac{R^{2-2i}-\e^{2-2i}}{2-2i}
+R^{2i}(R^2-\e^2)
\right)\\
&\quad=O(\e^{2j+2})=o(T_{\e,R}^j)\quad\text{as }\e\to0,
    \end{align*}
if  $i\geq2$, 
while, for $i=1$,
\begin{align*}
   & \int_\e^R r
    (\varphi_{1,\e}^2(r)+ \psi_{1,\e}^2(r))\dr
=\frac{j^2\e^{2+2j}((a^j_1)^2+(b^j_1)^2)}{(R^{2}-\e^{2})^2}
\int_\e^R r(r^{2}+R^{4}r^{-2}+2R^{2})\dr   \\
&\quad=\frac{j^2\e^{2+2j}((a^j_1)^2+(b^j_1)^2)}{(R^{2}-\e^{2})^2}
\left(
\frac{R^{4}-\e^{4}}{4}+R^{4}(\log R-\log \e)
+R^{2}(R^2-\e^2)
\right)\\
&\quad=O(\e^{2j+2}|\log\e|)=o(T_{\e,R}^j)\quad\text{as }\e\to0.
    \end{align*}
    Therefore
 \eqref{eq:oTeR} follows from \eqref{eq:parseval}.
\end{proof}
\begin{remark}\label{rem:caso-j=k}
    In view of \eqref{eq:a0k} and \eqref{eq:aik-bik}, in the case $j=k$ Lemma \ref{l:teRj}
provides the following information:
\begin{itemize}
    \item[(i)] if $\varphi_n(0)\neq0$  and $0$ is a critical point of $\varphi_n$ (hence, necessarily, $k=2$), then  
    \begin{equation*}
        T_{\e,R}^k=T_{\e,R}^2=
    \frac\pi2 (\lambda_n-1)^2(\varphi_n(0))^2\e^{4}|\log\e |+o(\e^{4}|\log\e |)\quad\text{as }\e\to0;
    \end{equation*}
\item[(ii)] if either  $\varphi_n(0)=0$ or $\nabla\varphi_n(0)\neq(0,0)$, then 
    \begin{equation*}
        T_{\e,R}^k=
    \pi k\left( 
    \bigg|\frac{\partial^{k}\varphi_n}{\partial x_1^k}(0) \bigg|^2+\frac1{k^2}\bigg|\frac{\partial^{k}\varphi_n}{\partial x_1^{k-1}
    \partial x_2}(0) \bigg|^2\right) \e^{2k}+o(\e^{2k})\quad\text{as }\e\to0.
    \end{equation*}
\end{itemize}
    \end{remark}

\begin{lemma}\label{l:confr-tor}
   For every $j\geq1$ and  $R>0$, 
   \begin{equation*}
   \mathcal{T}_{\overline{B_R}\setminus B_\e}(\partial B_\e,\partial_\nnu P^{\varphi_n}_j)=T_{\e,R}^j+o(T_{\e,R}^j)\quad\text{as $\e\to0$}.
   \end{equation*}
\end{lemma}
\begin{proof}
    By \eqref{eq:variationalU} we have
\begin{equation*}
    \int_{B_R\setminus B_\e}(\nabla U_{B_R,B_\e,\partial_\nnu P^{\varphi_n}_j}\cdot\nabla W_{\e,R,j}+U_{B_R,B_\e,\partial_\nnu P^{\varphi_n}_j}W_{\e,R,j})\dx
=\int_{\partial B_\e}(\partial_\nnu P^{\varphi_n}_j) W_{\e,R,j}\ds=T_{\e,R}^j,
\end{equation*}
while \eqref{eq:eq-We} and Remark \ref{rmk:torsion_norm}   yield
\begin{align*}
    &\int_{B_R\setminus B_\e}\nabla W_{\e,R,j}\cdot\nabla U_{B_R,B_\e,\partial_\nnu P^{\varphi_n}_j}\dx\\
    &\quad =\int_{\partial B_\e}(\partial_\nnu P^{\varphi_n}_j)\,U_{B_R,B_\e,\partial_\nnu P^{\varphi_n}_j}\ds+\frac{j a_0^j \e^j}{R^2-\e^2}\int_{B_R\setminus B_\e}U_{B_R,B_\e,\partial_\nnu P^{\varphi_n}_j}\dx\\
    &\quad=\mathcal{T}_{\overline{B_R}\setminus B_\e}(\partial B_\e,\partial_\nnu P^{\varphi_n}_j)+\frac{j a_0^j \e^j}{R^2-\e^2}\int_{B_R\setminus B_\e}U_{B_R,B_\e,\partial_\nnu P^{\varphi_n}_j}\dx.
\end{align*}
  From the above identities we deduce that 
  \begin{multline*}
     \mathcal{T}_{\overline{B_R}\setminus B_\e}(\partial B_\e,\partial_\nnu P^{\varphi_n}_j)-T_{\e,R}^j\\
     =-\int_{B_R\setminus B_\e}
        U_{B_R,B_\e,\partial_\nnu P^{\varphi_n}_j}W_{\e,R,j}\dx 
        -\frac{j a_0^j \e^j}{R^2-\e^2}\int_{B_R\setminus B_\e}U_{B_R,B_\e,\partial_\nnu P^{\varphi_n}_j}\dx.
 \end{multline*}
From  Cauchy-Schwarz's inequality, \eqref{eq:oTeR}, and Lemma \ref{lemma:norm2} 
(Remark \ref{ref:rem-Teps-to-0-N=2} guaranteeing the validity of assumption $\lim_{\e\to0} \mathcal{T}_{\overline{B_R}\setminus B_\e}(\partial B_\e,\partial_\nnu P^{\varphi_n}_j)=~\!\!0$)
it follows that 
    \begin{equation*}
        \int_{B_R\setminus B_\e}
        U_{B_R,B_\e,\partial_\nnu P^{\varphi_n}_j}W_{\e,R,j}\dx=o\left(\sqrt{T_{\e,R}^j\,\mathcal{T}_{\overline{B_R}\setminus B_\e}(\partial B_\e,\partial_\nnu P^{\varphi_n}_j)}\right) 
    \end{equation*}
    as $\e\to0$. Moreover, since $\e^{2j}=O(T_{\e,R}^j)$ as $\e\to 0$ in view of \eqref{eq:asyTeR}, from Cauchy-Schwarz's inequality and Lemma \ref{lemma:norm2} we deduce that 
    \begin{align*}
        \e^j\left|\int_{B_R\setminus B_\e}U_{B_R,B_\e,\partial_\nnu P^{\varphi_n}_j}\dx\right|&\leq \e^j \sqrt{\pi(R^2-\e^2)}
        \sqrt{\int_{B_R\setminus B_\e}U^2_{B_R,B_\e,\partial_\nnu P^{\varphi_n}_j}\dx}\\    &=O\left(\sqrt{T_{\e,R}^j}\right)o\left(\sqrt{\mathcal{T}_{\overline{B_R}\setminus B_\e}(\partial B_\e,\partial_\nnu P^{\varphi_n}_j)}\right)\\
    &=o\left(\sqrt{T_{\e,R}^j\,\mathcal{T}_{\overline{B_R}\setminus B_\e}(\partial B_\e,\partial_\nnu P^{\varphi_n}_j)}\right) 
    \end{align*}
    as $\e\to0$.     
    Hence 
$\mathcal{T}_{\overline{B_R}\setminus B_\e}(\partial B_\e,\partial_\nnu P^{\varphi_n}_j)=T_{\e,R}^j+o(T_{\e,R}^j)+o(\mathcal{T}_{\overline{B_R}\setminus B_\e}(\partial B_\e,\partial_\nnu P^{\varphi_n}_j))$, i.e.
    \begin{equation*}
        (1+o(1))\mathcal{T}_{\overline{B_R}\setminus B_\e}(\partial B_\e,\partial_\nnu P^{\varphi_n}_j)=(1+o(1))T_{\e,R}^j\quad\text{as }\e\to0.
    \end{equation*}
    The lemma is thereby proved.
\end{proof}
Combining Lemmas \ref{l:confr-tor} and \eqref{eq:asyTeR} we derive the following asymptotic expansion as $\e\to0$
\begin{equation}\label{eq:asyTBr-Be}
    \mathcal{T}_{\overline{B_R}\setminus B_\e}(\partial B_\e,\partial_\nnu P^{\varphi_n}_j)=
    \begin{cases}
        \frac12\pi j^2(a^j_0)^2\e^{2j}|\log\e|+o(\e^{2j}|\log\e|),&\text{if } a^j_0\neq0,\\[5pt]
        \pi j^2 \left({\displaystyle{\sum_{i=1}^j\frac{(a^j_i)^2+ (b^j_i)^2 }{i}}}\right)\e^{2j}+o(\e^{2j}),&\text{if } a^j_0=0,
    \end{cases} 
\end{equation}
for all $j\geq1$. If $j=k$, in view of Remark \ref{rem:caso-j=k}, we have, more precisely,
\begin{enumerate}
    \item[(i)] if $\varphi_n(0)\neq0$  and $\nabla\varphi_n(0)=(0,0)$ (hence, necessarily, $k=2$), then, as $\e\to0$,  
    \begin{equation}\label{eq:asyTBr-Be-klog}
    \mathcal{T}_{\overline{B_R}\setminus B_\e}(\partial B_\e,\partial_\nnu P^{\varphi_n}_k)=\frac\pi2 (\lambda_n-1)^2(\varphi_n(0))^2\e^{4}|\log\e |+o(\e^{4}|\log\e |);
\end{equation}
 \item[(ii)] if either $\varphi_n(0)=0$  or $\nabla\varphi_n(0)\neq(0,0)$, then, as $\e\to0$,  
    \begin{equation}\label{eq:asyTBr-Be-kpow}
    \mathcal{T}_{\overline{B_R}\setminus B_\e}(\partial B_\e,\partial_\nnu P^{\varphi_n}_k)=
\pi k\left( 
    \bigg|\frac{\partial^{k}\varphi_n}{\partial x_1^k}(0) \bigg|^2+\frac1{k^2}\bigg|\frac{\partial^{k}\varphi_n}{\partial x_1^{k-1}
    \partial x_2}(0) \bigg|^2\right) \e^{2k}+o(\e^{2k}).
    \end{equation}
\end{enumerate}

\begin{lemma}\label{l:confr-om-nonom}
    For every $R>0$
    \begin{align*}
 &   \mathcal{T}_{\overline{B_R}\setminus B_\e}(\partial B_\e,\partial_\nnu \varphi_n)\\&=
    \mathcal{T}_{\overline{B_R}\setminus B_\e}(\partial B_\e,\partial_\nnu P^{\varphi_n}_k)+
    \begin{cases}
        O(\e^{9/2}|\log\e|^{3/4}),&\text{if $\varphi_n(0)\neq0$ and $\nabla\varphi_n(0)=(0,0)$},\\
        O(\e^{2k+\frac12}),&\text{if either $\varphi_n(0)=0$ or $\nabla\varphi_n(0)\neq(0,0)$},
        \end{cases}\\
    &=
    \begin{cases}
        \frac\pi2 (\lambda_n-1)^2(\varphi_n(0))^2\e^{4}|\log\e |+o(\e^{4}|\log\e |),&\text{if $\varphi_n(0)\neq0$ and $\nabla\varphi_n(0)=(0,0)$},\\[5pt]
        \pi k\Big( 
    \big|\frac{\partial^{k}\varphi_n}{\partial x_1^k}(0) \big|^2\!+\!\frac1{k^2}\big|\frac{\partial^{k}\varphi_n}{\partial x_1^{k-1}
    \partial x_2}(0) \big|^2\Big) \e^{2k}+o(\e^{2k}),&\text{if either $\varphi_n(0)=0$ or $\nabla\varphi_n(0)\neq(0,0)$},
        \end{cases}
    \end{align*}
    as $\e\to0$. 
    \end{lemma}
\begin{proof}
We first observe that, if $N=2$, by Lemma \ref{lemma:ext} and Sobolev trace theorems,
there exists $C_R>0$ (depending on $R$ but independent of $\e$) such that 
\begin{equation}\label{eq:trace-eps}
    \int_{\partial B_\e} u^2\ds\leq \frac{C_R}{\e}\|u\|^2_{H^1(B_R\setminus \overline{B_\e})}
    \quad\text{for all }u\in H^1(B_R\setminus \overline{B_\e}).
\end{equation}  
  Let $\Psi_\e=U_{B_R,B_\e,\partial_\nnu \varphi_n}-U_{B_R,B_\e,\partial_\nnu P^{\varphi_n}_k}-U_{B_R,B_\e,\partial_\nnu P^{\varphi_n}_{k+1}}$.    From \eqref{eq:variationalU} it follows that 
\begin{equation*}
    \int_{B_R\setminus B_\e}(\nabla \Psi_\e\cdot\nabla v+\Psi_\e v)\dx
=\int_{\partial B_\e}v\,\partial_\nnu (\varphi_n- P^{\varphi_n}_k- P^{\varphi_n}_{k+1}) \ds
\end{equation*}  
for every $v\in H^1(B_R\setminus \overline{B_\e})$, so that \eqref{eq:trace-eps} yields
\begin{align*}
          \int_{B_R\setminus B_\e}(|\nabla \Psi_\e|^2+\Psi_\e^2)\dx
&=\int_{\partial B_\e}\Psi_\e\,\partial_\nnu (\varphi_n- P^{\varphi_n}_k- P^{\varphi_n}_{k+1}) \ds
\leq \mathop{\rm const} \e^{k+\frac32}
\sqrt{\int_{\partial B_\e} \Psi_\e^2\ds}\\
&\leq \mathop{\rm const} \e^{k+1}\|\Psi_\e\|_{H^1(B_R\setminus \overline{B_\e})}
\end{align*}
for some $\mathop{\rm const}>0$ independent of $\e$ which varies from line to line. Hence
\begin{equation}\label{eq:prima-st-Phi-eps}
 \|\Psi_\e\|_{H^1(B_R\setminus \overline{B_\e})}=O(\e^{k+1})\quad\text{as }\e\to0.   
\end{equation}
From \eqref{eq:prima-st-Phi-eps}, Remark \ref{rmk:torsion_norm}, \eqref{eq:asyTBr-Be}, \eqref{eq:asyTBr-Be-klog}, and \eqref{eq:asyTBr-Be-kpow}  it follows that 
\begin{align}
    &\label{eq:stimaU}\|U_{B_R,B_\e,\partial_\nnu \varphi_n}\|_{H^1(B_R\setminus \overline{B_\e})}\\
    &\notag\quad \leq \|\Psi_\e\|_{H^1(B_R\setminus \overline{B_\e})}+
    \|U_{B_R,B_\e,\partial_\nnu P^{\varphi_n}_k}\|_{H^1(B_R\setminus \overline{B_\e})}+\|U_{B_R,B_\e,\partial_\nnu P^{\varphi_n}_{k+1}}\|_{H^1(B_R\setminus \overline{B_\e})}\\
    &\notag\quad =
    \begin{cases}
        O(\e^2|\log\e|^{1/2}),&\text{if $\varphi_n(0)\neq0$ and $\nabla\varphi_n(0)=(0,0)$},\\
        O(\e^k),&\text{if either $\varphi_n(0)=0$ or $\nabla\varphi_n(0)\neq(0,0)$},
    \end{cases}
\end{align}
as $\e\to0$.
Cauchy-Schwarz's inequality and estimates 
\eqref{eq:prima-st-Phi-eps}--\eqref{eq:stimaU}, 
\eqref{eq:asyTBr-Be}, \eqref{eq:asyTBr-Be-klog}, and \eqref{eq:asyTBr-Be-kpow}  imply 
\begin{align*}
    \|U_{B_R,B_\e,\partial_\nnu \varphi_n}&-U_{B_R,B_\e,\partial_\nnu P^{\varphi_n}_k}\|_{H^1(B_R\setminus \overline{B_\e})}^2-\|U_{B_R,B_\e,\partial_\nnu P^{\varphi_n}_{k+1}}\|_{H^1(B_R\setminus \overline{B_\e})}^2\\
    \notag&=(\Psi_\e,U_{B_R,B_\e,\partial_\nnu \varphi_n}-U_{B_R,B_\e,\partial_\nnu P^{\varphi_n}_k}+U_{B_R,B_\e,\partial_\nnu P^{\varphi_n}_{k+1}})_{H^1(B_R\setminus \overline{B_\e})}\\
    &\leq \|\Psi_\e\|_{H^1(B_R\setminus \overline{B_\e})}
    \Big(\|U_{B_R,B_\e,\partial_\nnu \varphi_n}\|_{H^1(B_R\setminus \overline{B_\e})}+\|U_{B_R,B_\e,\partial_\nnu P^{\varphi_n}_k}\|_{H^1(B_R\setminus \overline{B_\e})}\\
    &\hskip5cm
    +\|U_{B_R,B_\e,\partial_\nnu P^{\varphi_n}_{k+1}}\|_{H^1(B_R\setminus \overline{B_\e})}\Big)\\
    &=\begin{cases}
        O(\e^5|\log\e|^{1/2}),&\text{if $\varphi_n(0)\neq0$ and $\nabla\varphi_n(0)=(0,0)$},\\
        O(\e^{2k+1}),&\text{if either $\varphi_n(0)=0$ or $\nabla\varphi_n(0)\neq(0,0)$},
    \end{cases}
\end{align*}
as $\e\to0$. Hence, in view of \eqref{eq:asyTBr-Be},
\begin{multline}
    \label{eq:stima-diff1}
    \|U_{B_R,B_\e,\partial_\nnu \varphi_n}-U_{B_R,B_\e,\partial_\nnu P^{\varphi_n}_k}\|_{H^1(B_R\setminus \overline{B_\e})}
    \\=\begin{cases}
        O(\e^{5/2}|\log\e|^{1/4}),&\text{if $\varphi_n(0)\neq0$ and $\nabla\varphi_n(0)=(0,0)$},\\
        O(\e^{k+\frac12}),&\text{if either $\varphi_n(0)=0$ or $\nabla\varphi_n(0)\neq(0,0)$},
    \end{cases}
\end{multline}
as $\e\to0$. 
From  Remark \ref{rmk:torsion_norm}, Cauchy-Schwarz's inequality, \eqref{eq:stima-diff1}, \eqref{eq:stimaU}, \eqref{eq:asyTBr-Be-klog}, and
\eqref{eq:asyTBr-Be-kpow} it follows that 
\begin{align*}
\mathcal{T}_{\overline{B_R}\setminus B_\e}&(\partial B_\e,\partial_\nnu \varphi_n)-
    \mathcal{T}_{\overline{B_R}\setminus B_\e}(\partial B_\e,\partial_\nnu P^{\varphi_n}_k)\\
    &   =\|U_{B_R,B_\e,\partial_\nnu \varphi_n}\|_{H^1(B_R\setminus \overline{B_\e})}^2-\|U_{B_R,B_\e,\partial_\nnu P^{\varphi_n}_{k}}\|_{H^1(B_R\setminus \overline{B_\e})}^2 \\
   &=(U_{B_R,B_\e,\partial_\nnu \varphi_n}-U_{B_R,B_\e,\partial_\nnu P^{\varphi_n}_{k}},U_{B_R,B_\e,\partial_\nnu \varphi_n}+U_{B_R,B_\e,\partial_\nnu P^{\varphi_n}_{k}})_{H^1(B_R\setminus \overline{B_\e})}\\
&   =\begin{cases}
        O(\e^{9/2}|\log\e|^{3/4}),&\text{if $\varphi_n(0)\neq0$ and $\nabla\varphi_n(0)=(0,0)$},\\
        O(\e^{2k+\frac12}),&\text{if either $\varphi_n(0)=0$ or $\nabla\varphi_n(0)\neq(0,0)$},
        \end{cases}
\end{align*}
as $\e\to0$, thus completing the proof in view of \eqref{eq:asyTBr-Be-klog} and
\eqref{eq:asyTBr-Be-kpow}.
\end{proof}

\begin{proof}[Proof of Proposition \ref{p:exp-2D}]
    Since $0\in\Omega$, there exist $R_1,R_2>0$ such that $B_{R_1}\subset\Omega\subset B_{R_2}$. From Corollary \ref{c:monotonicity-domain} it follows that 
     \begin{equation*}
		\mathcal{T}_{\overline{B_{R_2}}\setminus B_\e}(\partial B_\e,\partial_\nnu \varphi_n)
            \leq\mathcal{T}_{\overline{\Omega}\setminus B_\e}(\partial B_\e,\partial_\nnu \varphi_n)\leq 
    		\mathcal{T}_{\overline{B_{R_1}}\setminus B_\e}(\partial B_\e,\partial_\nnu \varphi_n),
      \end{equation*}
     so that the conclusion follows from  Lemma \ref{l:confr-om-nonom}.
\end{proof}

\begin{proposition}\label{p:phi-dnuphi-D2}
    Let $N=2$.
\begin{itemize}
   \item[\rm (i)] 
If     $0\in\Omega\setminus \mathrm{Sing}\,(\varphi_n)$, then 
   \begin{equation*}
       \int_{B_\e}\left(\abs{\nabla\varphi_n}^2-(\lambda_n
    -1)\varphi_n^2 \right)\dx =    \pi\e^2\Big(
       |\nabla\varphi_n(0)|^2-(\lambda_n-1)|\varphi_n(0)|^2\Big)+o(\e^2)
   \end{equation*}
   as $\e\to0$.
   \item[\rm (ii)] 
If     $0\in\mathrm{Sing}\,(\varphi_n)$, then 
    \begin{equation*}
        \int_{B_\e}\left(\abs{\nabla\varphi_n}^2-(\lambda_n 
    -1)\varphi_n^2 \right)\dx = k\pi \e^{2k}\left( 
    \bigg|\frac{\partial^{k}\varphi_n}{\partial x_1^k}(0) \bigg|^2+\frac1{k^2}\bigg|\frac{\partial^{k}\varphi_n}{\partial x_1^{k-1}
    \partial x_2}(0) \bigg|^2\right)+o(\e^{2k})
    \end{equation*}
    as $\e\to0$,
where $k\geq2$ is the vanishing order at $0$ of $\varphi_n-\varphi_n(0)$.
\end{itemize}
\end{proposition}
\begin{proof}
If $0\not\in \mathrm{Sing}\,(\varphi_n)$, we can argue as in \eqref{eq:bl_intr_2} to deduce (i).

    Let $0\in \mathrm{Sing}\,(\varphi_n)$. In this case $P^{\varphi_n}_k(r\cos t,r\sin t)=r^k\big(c_{1}\cos(k t)+c_{2}\sin(k t)\big)$ with $c_1,c_2$ as in \eqref{eq_c1c2}, see Remark \ref{rem:j=k}.
    Then 
    \begin{multline*}
    \int_{B_\e}\left(\abs{\nabla\varphi_n}^2-(\lambda_n 
    -1)\varphi_n^2 \right)\dx =-
        \int_{\partial B_\e}\varphi_n\partial_\nnu\varphi_n\ds=-
        \int_{\partial B_\e}P^{\varphi_n}_k\partial_\nnu P^{\varphi_n}_k\ds+ o(\e^{2k})\\
        =k \e^{2k}\int_0^{2\pi}(c_{1}\cos(k t)+c_{2}\sin(k t))^2
        \,\mathrm{d}t+ o(\e^{2k})=k \pi \e^{2k}(c_1^2+c_2^2)+ o(\e^{2k}) \quad\text{as }\e\to0, 
    \end{multline*}
    thus proving (ii).
\end{proof}
We are now in position to prove Theorem \ref{thm:spher-Dim2}.
\begin{proof}[Proof of Theorem \ref{thm:spher-Dim2}]
By translation, it is not restrictive to assume $x_0=0$. The conclusion  follows from Theorem \ref{thm:main1}, expanding  the torsional rigidity $\mathcal{T}_{\overline{\Omega}\setminus B_\e}(\partial B_\e,\partial_\nnu \varphi_n)$ as in Proposition~\ref{p:exp-2D} and $\int_{B_\e}\left(\abs{\nabla\varphi_n}^2-(\lambda_n 
    -1)\varphi_n^2 \right)\dx$ as in Proposition \ref{p:phi-dnuphi-D2}.   
\end{proof}

\begin{example}
We conclude this section with an example, in which the hole is excised from a disk. To this end, let us take $\Omega = B_2 \subset \R^2$. 
It is well known (see, e.g., \cite{grebenkov})  that the eigenvalues of the unperturbed Neumann problem \eqref{eq:P0}  are
\[
\lambda_{nk}=\frac{\alpha_{nk}^2}{4}+1,
\]
$\alpha_{nk}$ being the positive roots, enumerated by $k$, of $J_n'(z)$, where $J_n(z)$ is the Bessel function of the first kind of order $n$. These eigenvalues are all simple for $n=0$. In this case,  the eigenfunctions read
\[
\varphi_{k}(r,\theta) = J_0\left(\alpha_{0k} \frac r 2\right). 
\]
Therefore, the $2$-dimensional analogue of the interface $\Gamma$  introduced in Remark \ref{rem:Gamma} is characterized by the equation 
\[
2 J_1^2\left(\alpha_{0k} \frac r 2\right) - J_0^2\left(\alpha_{0k} \frac r 2\right) = 0.
\]
Relying again on the computational software Mathematica\texttrademark \,we can plot the interface (in blue), along with the nodal lines of $\varphi_k$ (in green), for the cases
\[
\alpha_{01} \approx 3.831, \qquad \alpha_{02} \approx 7.016.  
\]
The results can be seen in Figure \ref{img4}.

\begin{figure}[ht]
\begin{center}
\includegraphics[width=6cm]{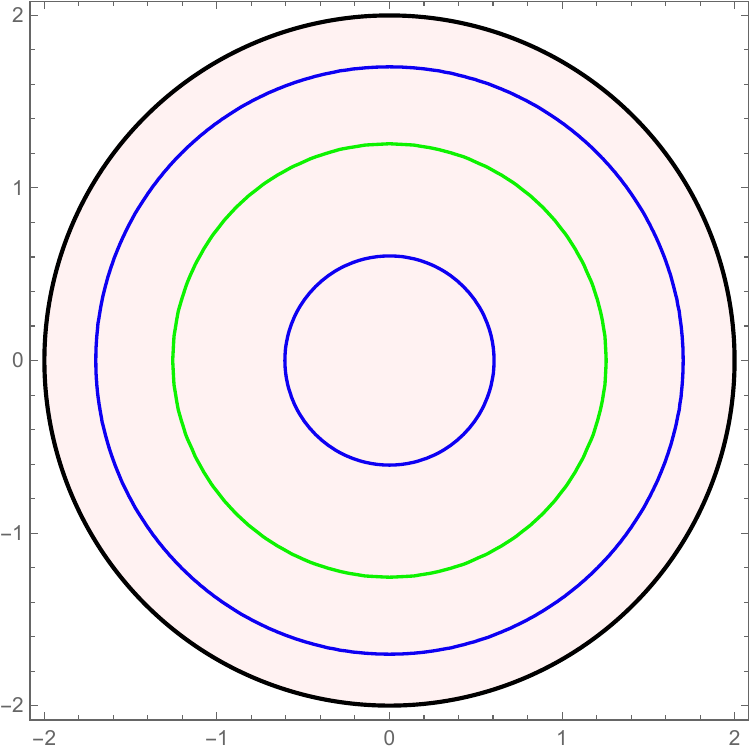} 
\quad 
\includegraphics[width=6cm]{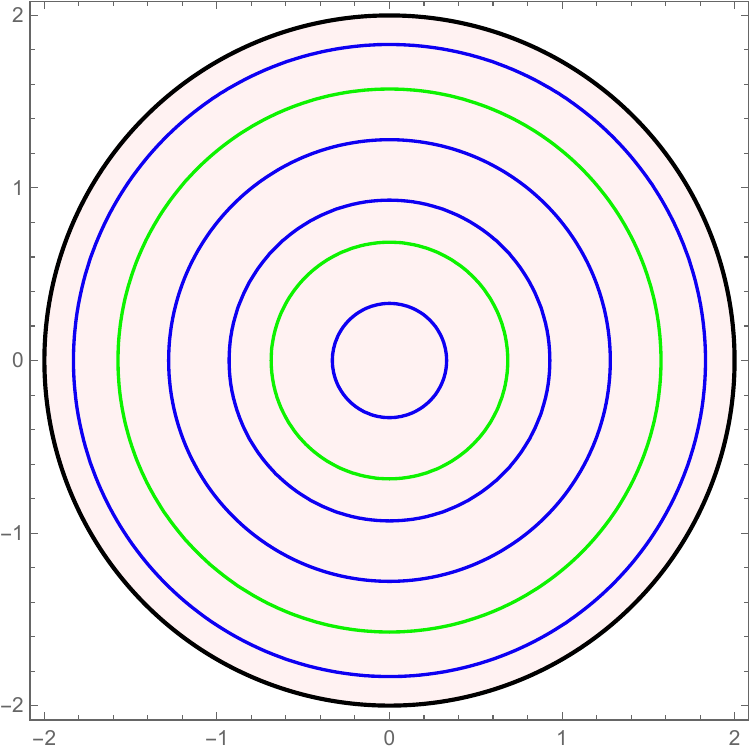}
\caption{Interface $\Gamma$ and nodal lines of the eigenfunction for the cases $\alpha_{01}$ (left) and $\alpha_{02}$ (right).}
\label{img4}
\end{center}
\end{figure}

\end{example}


\section{Appendix}

We recall here a known result about approximation of small eigenvalues of linear operators. This lemma, originally proved by Y. Colin de Verdiére in \cite{ColindeV1986} and then revisited in \cite{Courtois1995} and \cite{ALM2022}, also applies to multiple eigenvalues. We present here a simplified version 
applicable to the case of simple eigenvalues
and  provide a short proof for the readers' convenience.

\begin{lemma}[Lemma on small eigenvalues]\label{lemma:sm_eig}
    Let $(\mathcal{H},(\cdot,\cdot))$ be a real Hilbert space, $\mathcal{D}\sub\mathcal{H}$ a  subspace,  and $q\colon \mathcal{D}\times \mathcal{D}\to \R$ a bilinear symmetric form. 
    Let 
    \begin{enumerate}
    \item[\rm (i)]$\lambda\in\R$ and $\phi\in\mathcal{D}$ be such that 
        \[
            \norm{\phi}=1\quad\text{and}\quad q(\phi,v)=\lambda(\phi,v)\quad\text{for all }v\in\mathcal{D},
        \]
        where $\|\cdot\|=\sqrt{(\cdot,\cdot)}$ denotes the norm associated to the scalar product;
        \item[\rm (ii)] 
 $f\in\mathcal{D}$ be such that  $\norm{f}=1$.
    \end{enumerate}
    Let us  assume that $\{\phi\}^\perp=H_1\oplus H_2$ for some subspaces $H_1,H_2$ mutually orthogonal such that $H_1\subset\mathcal D$, 
    \begin{align}
    \label{eq:ortho-q}
&        q(v_1,v_2)=0\quad\text{for all $v_1\in H_1$ and $v_2\in H_2\cap\mathcal D$},\\    
    \label{eq:sm_eig_hp1}
&        \gamma_1:=\inf\left\{ \frac{\abs{q(v,v)}}{\norm{v}^2}\colon v\in H_1\setminus\{0\}\right\}>0,\\
\label{eq:sm_eig_hp1-bis}
&        \gamma_2:=\inf\left\{ \frac{\abs{q(v,v)}}{\norm{v}^2}\colon v\in (H_2\cap\mathcal D)\setminus\{0\} \right\}>0,
\end{align}
and 
    \begin{equation}\label{eqLdef-delta}
        \delta:=\sup\left\{ \frac{\abs{q(f,v)}}{\norm{v}}\colon v\in\mathcal{D}\setminus\{0\}\right\}<+\infty.
    \end{equation}
Then 
    \begin{equation}\label{eq:sm_eig_th1}
        \norm{f-\Pi f}\leq \frac{\sqrt 2\,\delta}{\gamma},
    \end{equation}
where $\gamma:=\min\{\gamma_1,\gamma_2\}$ and $\Pi$ denotes the orthogonal projection onto 
$\Span\{\phi\}$, i.e. 
    \begin{align*}
        \Pi\colon \mathcal{H}&\to \Span\{\phi\} \\
        v &\mapsto (\phi,v)\,\phi.
    \end{align*}
     Finally, if $\xi:=q(f,f)$, then 
    \begin{equation}\label{eq:sm_eig_th2}
        \abs{\lambda-\xi}\leq  2|\lambda|\,\frac{\delta^2}{\gamma^2}+2\frac{\delta^2}{\gamma}.
    \end{equation}
\end{lemma}
\begin{proof}
Let us denote
    \begin{equation*}
        \mathsf{N}f:=f-\Pi f.
    \end{equation*}
 We observe that $\mathsf{N}f$ is orthogonal to $\phi$, i.e.
    \begin{equation*}
        (\phi,\mathsf{N}f)=0,
    \end{equation*}
    and, letting 
    \begin{equation*}
        \mathsf{N}_1:H\to H_1\quad\text{and}\quad \mathsf{N}_2:H\to H_2
    \end{equation*} 
    be the orthogonal projections on $H_1$ and $H_2$, respectively, we have  
\begin{equation}\label{eq:sommaN1N2}
    \mathsf{N}f=\mathsf{N}_1 f+\mathsf{N}_2f.
\end{equation}
Moreover 
\begin{equation}\label{eq:ortho-2}
        (\phi,\mathsf{N}_1f)=(\phi,\mathsf{N}_2f)=0.
    \end{equation}
Since  $H_1\subset\mathcal D$ by assumption, we have
$\mathsf{N}_1f\in\mathcal D$; moreover $\mathsf{N}f\in\mathcal D$, hence $\mathsf{N}_2f\in\mathcal D\cap H_2$ by \eqref{eq:sommaN1N2}. Therefore, taking into account \eqref{eq:ortho-2},
\begin{equation*}
        q(\phi,\mathsf{N}_1f)=\lambda(\phi,\mathsf{N}_1f)=0,
        \quad 
    q(\phi,\mathsf{N}_2f)=\lambda(\phi,\mathsf{N}_2f)=0,
        \end{equation*}
so that 
\begin{equation}\label{eq:sm_eig2}
        q(\Pi f,\mathsf{N}_1f)=q(\Pi f,\mathsf{N}_2f)=0.
    \end{equation}
    From \eqref{eq:ortho-q}, \eqref{eq:sommaN1N2}, and     
    \eqref{eq:sm_eig2} it follows that 
    \begin{align*}
        &q(\mathsf{N}_1f,\mathsf{N}_1f)=q(f-\Pi f-\mathsf{N}_2f, \mathsf{N}_1f)=
        q(f,\mathsf{N}_1f)-q(\Pi f,\mathsf{N}_1f)-q(\mathsf{N}_2f,\mathsf{N}_1f)
        =q(f,\mathsf{N}_1f)\\
        &q(\mathsf{N}_2f,\mathsf{N}_2f)=q(f-\Pi f-\mathsf{N}_1f, \mathsf{N}_2f)=
        q(f,\mathsf{N}_2f)-q(\Pi f,\mathsf{N}_2f)-q(\mathsf{N}_1f,\mathsf{N}_2f)
        =q(f,\mathsf{N}_2f).
    \end{align*}
    Therefore, from the definition of $\delta$, $\gamma_1$, and $\gamma_2$ we obtain
    \begin{align*}
        &\abs{q(\mathsf{N}_1f,\mathsf{N}_1f)}=\abs{q(f,\mathsf{N}_1f)}\leq \delta\norm{\mathsf{N}_1f}\leq \delta \sqrt{\frac{\abs{q(\mathsf{N}_1f,\mathsf{N}_1f)}}{\gamma_1}},\\
&\abs{q(\mathsf{N}_2f,\mathsf{N}_2f)}=\abs{q(f,\mathsf{N}_2f)}\leq \delta\norm{\mathsf{N}_2f}\leq \delta \sqrt{\frac{\abs{q(\mathsf{N}_2f,\mathsf{N}_2f)}}{\gamma_2}},
\end{align*}    
    which yields
    \begin{equation}\label{eq:sm_eig_1}
        \abs{q(\mathsf{N}_1f,\mathsf{N}_1f)}\leq \frac{\delta^2}{\gamma_1}
        \quad\abs{q(\mathsf{N}_1f,\mathsf{N}_2f)}\leq \frac{\delta^2}{\gamma_2}.
    \end{equation}
    Combining \eqref{eq:sm_eig_1}  with the definition of $\gamma_1,\gamma_2,\gamma$, we obtain the estimates
    \begin{equation}\label{eq:stima-norme-N1N2}
        \norm{\mathsf{N}_1f}^2\leq \frac{\abs{q(\mathsf{N}_1f,\mathsf{N}_1f)}}{\gamma_1}\leq \frac{\delta^2}{\gamma_1^2}\leq \frac{\delta^2}{\gamma^2},\quad
                \norm{\mathsf{N}_2f}^2\leq \frac{\abs{q(\mathsf{N}_2f,\mathsf{N}_2f)}}{\gamma_2}\leq \frac{\delta^2}{\gamma_2^2}\leq \frac{\delta^2}{\gamma^2}.
    \end{equation}
From \eqref{eq:sommaN1N2},  the orthogonality of $H_1$ and $H_2$ and \eqref{eq:stima-norme-N1N2}, we deduce that 
\begin{equation*}
    \|\mathsf{N}f\|^2=\|\mathsf{N}_1f\|^2+\|\mathsf{N}_2f\|^2\leq \frac{2\delta^2}{\gamma^2},
\end{equation*}
thus proving \eqref{eq:sm_eig_th1}. 
    
    Now, the proof of \eqref{eq:sm_eig_th2} follows from direct estimates, making use of \eqref{eq:sm_eig_th1} and \eqref{eq:sm_eig_1}. More precisely, if $\Pi f\neq0$, we first write $|\lambda-\xi|$ as
    \begin{equation*}
        \abs{\lambda-\xi}=\abs{ \frac{q(\Pi f,\Pi f)}{\norm{\Pi f}^2}-\frac{q(f,f)}{\norm{f}^2}}=\abs{ \frac{q(\Pi f,\Pi f)}{\norm{\Pi f}^2}-\frac{q(\mathsf{N}f+\Pi f,\mathsf{N}f+\Pi f)}{\norm{\mathsf{N}f+\Pi f}^2} }.
    \end{equation*}
    Then, by this, the orthogonality condition \eqref{eq:sm_eig2}, assumption \eqref{eq:ortho-q}, and the fact that $\norm{f}=1$, we obtain 
    \begin{equation}\label{eq:lambda-xi}
    \abs{\lambda-\xi}=\abs{\lambda \norm{\mathsf{N}f}^2-q(\mathsf{N}_1f,\mathsf{N}_1f)-q(\mathsf{N}_2f,\mathsf{N}_2f) }.
    \end{equation}
On the other hand, \eqref{eq:lambda-xi} is trivially satisfied if $\Pi f=0$, since, in this case, $f=\mathsf{N}f$.
    Combining \eqref{eq:lambda-xi} with the triangle inequality, \eqref{eq:sm_eig_th1} and \eqref{eq:sm_eig_1}, we obtain \eqref{eq:sm_eig_th2}.
\end{proof}

The following lemma provides an uniform extension property in domains with small holes of the form \eqref{eq:hole-shrinking}, see \cite{SW99} for the proof. 
\begin{lemma}[Extension operators]
\label{lemma:ext}
For $N\geq2$, let $\Omega\subset\R^N$ and $\Sigma\subset\R^N$ be  bounded, open Lipschitz sets.
Let $\e_0>0$ and $r_0>0$ be such that \eqref{eq:inOmega} is satisfied for some   $x_0\in\Omega$. For every $\e\in(0,\e_0)$, let $\Sigma_\e:=x_0+\e\Sigma$ and 
$\Omega_\e=\Omega\setminus\overline{\Sigma_\e}$. Then, for every $\e\in(0,\e_0)$, there exists an (inner) extension operator
\begin{equation*}
\Ee\colon H^1(\Oe)\to H^1(\Omega)
\end{equation*}
such that, for all $u\in H^1(\Oe)$, 
\[
(\Ee u)\restr{\Oe}=u
\]
and 
\begin{equation*}
\norm{\Ee u}_{H^1(\Omega)}
\leq \mathfrak{C}\norm{u}_{H^1(\Oe)},
\end{equation*}
for some constant $\mathfrak{C}>0$ independent of $\e\in(0,\e_0)$.
\end{lemma}

\section*{Acknowledgments}

  R. Ognibene is supported by the project ERC VAREG - \emph{Variational approach to the regularity of the free boundaries} (grant agreement No. 853404) and by the INdAM-GNAMPA project 2023 CUP E53C22001930001. 
    V. Felli is supported by the  PRIN project no. 20227HX33Z - \emph{Pattern formation in nonlinear phenomena}.
    Part of this work was developed while V. Felli and R. Ognibene were in residence at Institut Mittag-Leffler in Djursholm, Stockholm (Sweden) during the semester \emph{Geometric Aspects of Nonlinear Partial Differential Equations} in 2022, supported by the Swedish Research Council under grant no. 2016-06596. L. Liverani is supported by the Alexander von Humboldt fellowship for Postdoctoral Researchers.


\printbibliography

\end{document}